\documentclass{amsart}
\usepackage{amsfonts, amsthm, amsmath, amssymb, enumerate, mathrsfs, tikz, tikz-cd, ifthen, color, chngcntr, mathtools, graphicx, enumitem, xfrac}
\usepackage[colorlinks]{hyperref}
	\hypersetup{linkcolor={blue}, citecolor={blue}}
\usepackage[toc,page]{appendix}
\usepackage{subcaption}

\usetikzlibrary{arrows}
\usetikzlibrary{decorations.markings}

\newtheorem{theorem}[subsection]{Theorem}
\newtheorem{prop}[subsection]{Proposition}
\newtheorem{lemma}[subsection]{Lemma}

\newtheorem{cor}[subsection]{Corollary}
\newtheorem*{theorem*}{Theorem}

\theoremstyle{definition}

\newtheorem{example}[subsection]{Example}

\newtheorem{rmk}[subsection]{Remark}

\definecolor{light-gray}{gray}{.75}

\DeclareMathOperator{\coker}{coker}
\DeclareMathOperator{\im}{im}
\DeclareMathOperator{\cof}{cof}

\newcommand{\M}{\mathbb{M}}
\newcommand{\A}{\mathbb{A}}
\newcommand{\R}{\mathbb{R}}
\newcommand{\Z}{\mathbb{Z}}

\newcommand{\D}{\mathbb{D}}
\renewcommand{\H}{\tilde{H}}
\newcommand{\uZ}{\underline{\Z}}

\renewcommand{\dot}{\small{$\bullet$}}

\counterwithin{equation}{subsection}

\newcommand{\cone}[3]{
	\draw[thick, #3] (#1+1/2,5) -- (#1+1/2,#2+1/2) -- (5-#2+#1,5);
	\draw[thick, #3] (#1+1/2,-5) -- (#1 +1/2, #2 -1.5) -- (-3-#2+#1,-5)
}

\newcommand{\zbox}[3]{
\draw[#3,fill=white] (#1+0.65, #2 + 0.65) -- (#1+0.65, #2+0.35) -- (#1+0.35,#2+0.35) -- (#1+0.35, #2+0.65) -- (#1+0.65, #2 + 0.65)
}

\newcommand{\zcone}[3]{
	\draw[thick,dashed,#3] (#1+1/2,#2+1/2) arc (145:218:1.7); 
	\zbox{#1}{#2}{#3};
	\draw[thick, #3] (#1+0.7,#2+0.7)--(5-#2+#1,5);
	\zbox{#1}{#2+2}{#3};
	\draw[thick,#3] (#1+0.7,#2+2.7)--(5-#2-2+#1,5);
	{\ifthenelse{#2<1}{
	\zbox{#1}{#2+4}{#3};
	\draw[thick,#3] (#1+0.7,#2+4.7)--(5-#2-4+#1,5);}{
	}};
	\zbox{#1}{#2-2}{#3};
	\draw[thick,#3] (#1+1/2,#2-2.5) node{{\dot}} --(#1-#2-2,-5);
	\zbox{#1}{#2-4}{#3};
	\draw[thick,#3] (#1-#2,-5) --(#1+1/2,#2-4.5)node{{\dot}};
}

\newcommand{\zanti}[4]{
	\foreach \y in {#2-4,#2-2,#2, #2+2}
	\zbox{#1}{\y}{#3};
	\foreach \y in {#2-4,#2-2,#2, #2+2}
	\draw[thick, #3] (#1+0.6,\y+0.6)--(#1+0.5+#4, \y+0.5+#4) node{{\dot}};
	\foreach \y in {#2+#4-5,#2+#4-3,#2+#4-1, #2+#4+1}
	\zbox{#1+#4}{\y}{#3};
}

\newcommand{\zantizo}[2]{
	\foreach \y in {-4,-2,0,2,4}
	\zbox{#1}{\y}{#2};
	\foreach \y in {-4,-2,0,2}
	\draw[thick, #2] (#1+0.6,\y+0.6)--(#1+1.5, \y+1.5) node{{\small{$\bullet$}}};
	\draw[thick,#2] (#1+.6,4.6)--(#1+1,5);
	\foreach \y in {-4,-2,0,2,4}
	\zbox{#1+1}{\y}{#2};
	\draw[thick,#2] (#1+1,-5)--(#1+1.5,-4.5) node{{\small{$\bullet$}}};
}

\newcommand{\zantioo}[2]{
	\foreach \y in {-5,-3,-1,1,3}
	\zbox{#1}{\y}{#2};
	\foreach \y in {-5,-3,-1,1,3}
	\draw[thick, #2] (#1+0.6,\y+0.6)--(#1+1.5, \y+1.5) node{{\small{$\bullet$}}};
	\foreach \y in {-5,-3,-1,1,3}
	\zbox{#1+1}{\y}{#2};
}

\newcommand{\lab}[3]{
\draw[#3] (#1+1/2,5.5) node{\tiny{$#2$}}
}

\newcommand{\zline}[2]{
	\foreach \y in{-5,...,4}
	\zbox{#1}{\y}{#2};
}

\newcommand{\dotline}[2]{
	\foreach \y in{-5,...,4}
	\draw[#2] (#1+0.5,\y+0.5) node{{\small{$\bullet$}}};
}

\title{The cohomology of $C_2$-surfaces with $\underline{\Z}$-coefficients}

\author[Hazel]{Christy Hazel}
\address[Hazel]{University of California Los Angeles}

\begin{document}
%%%%%%%%%%%%%%%%
\begin{abstract}
Let $C_2$ denote the cyclic group of order $2$. We compute the $RO(C_2)$-graded cohomology of all $C_2$-surfaces with constant integral coefficients. We show when the action is nonfree, the answer depends only on the genus, the orientability of the underlying surface, the number of isolated fixed points, the number of fixed circles with trivial normal bundles, and the number of fixed circles with nontrivial normal bundles. When the action on the surface is free, we show the answer depends only on the genus, the orientability of the underlying surface, whether or not the action is orientation preserving versus reversing in the orientable case, and one other invariant.
\end{abstract}

\maketitle

%%%%%%%%%%%%%%%%%
\section{Introduction}

When studying spaces with an action of a finite group $G$, Bredon $RO(G)$-graded cohomology plays the role of singular cohomology. Despite this fundamental role, computations in $RO(G)$-graded cohomology are often complicated and mysterious. The goal of this paper is to present a complete family of computations in Bredon $RO(C_2)$-graded cohomology with constant integral coefficients, and to present the answer in a straightforward format that only depends on a few properties of the space and the action.

Specifically, we use Dugger's classification of closed surfaces with an involution \cite{Dug19} to compute the cohomology of all $C_2$-surfaces with $\uZ$-coefficients. We show the answer depends only on a few invariants of the surface and the $C_2$-action. In addition to the surface computations, we prove the existence of a top cohomology class for any closed manifold with a $C_2$-action. We further show this class generates a free submodule when the underlying manifold is orientable. The work in this paper builds on the author's previous work in \cite{Haz20} where the cohomology of all $C_2$-surfaces was computed with $\underline{\Z/2}$-coefficients, and a similar result for a top cohomology class for general $C_2$-manifolds was proven.\smallskip

We start by providing some background in order to give a summary of the main results. Bredon cohomology $H^{\star}(-;\uZ)$ is a cohomology theory graded on $RO(C_2)$, the Grothendieck ring of finite-dimensional, real, orthogonal $C_2$-representations. Any such $C_2$-representation is isomorphic to a direct sum of copies of the one-dimensional trivial representation $\R_{triv}$ and the one-dimensional sign representation $\R_{sgn}$. We use the motivic notation $\R^{p,q}$ for the $p$-dimensional virtual representation $(p-q)[\R_{triv}] + q[\R_{sgn}]$. We can thus regard the cohomology as a bigraded theory and write $H^{p,q}(-;\uZ)$ for $H^{\R^{p,q}}(-;\uZ)$. 

Given a $C_2$-space $X$, we have an equivariant map $X\to pt$ where $pt$ denotes a singleton with a trivial $C_2$-action. This induces a map of bigraded rings $H^{*,*}(pt;\uZ)\to H^{*,*}(X;\uZ)$ making $H^{*,*}(X;\uZ)$ into a module over $H^{*,*}(pt;\uZ)$.

Let $\M=H^{*,*}(pt;\uZ)$. Our aim is to understand the cohomology of $C_2$-spaces as $\M$-modules. We introduce $\M$ in Figure \ref{fig:pointzintro}. Since this is a bigraded ring, we can record properties of this ring on a grid. We plot information about the $(p,q)$ group in the box up and to the right of the $(p,q)$ lattice point. The ring $\M$ is shown on the left. A box denotes a copy of $\Z$, and a dot denotes a copy of $\Z/2$. We have that $H^{p,q}(pt;\uZ)\cong \Z[x,\rho]/(2\rho)$ for $p,q\geq 0$ where $|x|=(0,2)$, $|\rho|=(1,1)$. Diagonal lines indicate multiplication by $\rho$, and dashed vertical lines indicate multiplication by $x$. The portion of the picture for $p,q<0$ is more complicated;
see Section \ref{sec:background} for a complete description of $\M$. 

The right-hand grid in Figure \ref{fig:pointzintro} shows the cohomology of the free orbit $H^{*,*}(C_2;\uZ)$. The $\M$-module $H^{*,*}(C_2;\uZ)$ is isomorphic to a direct sum of two copies of $x^{-1}\M/(\rho)$, one of which is shifted up one. We denote this module by $\A_0$. Again the dashed vertical lines indicate action by $x$.

\begin{figure}[ht]
\begin{tikzpicture}[scale=.5]
\draw[help lines,gray] (-5.125,-5.125) grid (5.125, 5.125);
\draw[<->] (-5,0)--(5,0)node[right]{$p$};
\draw[<->] (0,-5)--(0,5)node[above]{$q$};
	\draw[thick,dashed] (1/2,1/2) to[out=120,in=240](1/2,2.5);
	\draw[thick,dashed] (1/2,2+1/2) to[out=120,in=240](1/2,4.5);
	\draw[thick,dashed] (1+1/2,1+1/2) to[out=120,in=240](1+1/2,3.5);
	\draw[thick,dashed] (1+1/2,3+1/2) to[out=120,in=260](1+1/3,5);
	\draw[thick,dashed] (2+1/2,2+1/2) to[out=120,in=240](2+1/2,4.5);
	\draw[thick,dashed] (3+1/2,3+1/2) to[out=120,in=260](3+1/3,5);
	\draw[thick,dashed] (4+1/2,4+1/2) to[out=120,in=270](4+1/3,5);
	\draw[thick,dashed] (2+1/2,4+1/2) to[out=120,in=270](2+1/3,5);
	\draw[thick,dashed] (1/2,-4+1/2) to[out=120,in=240](1/2,-1.5);
	\draw[thick,dashed] (1/2,-5+1/2) to[out=60,in=300](1/2,-2.5);
	\draw[thick,dashed] (-1/3,-5) to[out=80,in=300](-.5,-3.5);
	\draw[thick,dashed] (-4/3,-5) to[out=100,in=300](-1.5,-4.5);
	\draw[thick,dashed] (.7,-5) to[out=100,in=310](1/2,-4.5);
	\draw[thick,dashed] (1/5,-5) to[out=100,in=240](1/2,-3.5);
	\draw[thick,dashed] (1/2,4+1/2) to[out=120,in=270](1/3,5);
	\foreach \x in {1,2,3,4}
	\draw (\x+.5,\x+.5) node{\small{$\bullet$}};
	\foreach \x in {1,2}
	\draw (\x+.5,\x+2.5) node{\small{$\bullet$}};
	\zcone{0}{0}{black};
	\draw[thick] (.6,-.5) node{\tiny{$\times 2$}};
	\draw[thick] (-.5,-3.5) node{\small{$\bullet$}};
	\draw[thick](-1.5,-4.5)node{\small{$\bullet$}};
	\draw[thick](-1.5,-4.5)--(-2,-5);
	\draw (-.3,0.5) node{$1$};
	\draw (-.3,2.3) node{$x$};
	\draw (-.4,4.35) node{$x^2$};
	\draw (2,1.5) node{$\rho$};
	\draw (3,2.5) node{$\rho^2$};
\end{tikzpicture}\hspace{0.2in}
\begin{tikzpicture}[scale=.5]
\draw[help lines,gray] (-5.125,-5.125) grid (5.125, 5.125);
\draw[<->] (-5,0)--(5,0)node[right]{$p$};
\draw[<->] (0,-5)--(0,5)node[above]{$q$};
	\foreach \y in {-5,-3,-1,1}
		\draw[thick,dashed] (1/2,\y+1/2) to[out=130,in=230](1/2,\y+2.5);
	\foreach \y in {-4,-2,0,2}
		\draw[thick,dashed] (1/2,\y+1/2) to[out=50,in=310](.5,\y+2.5);
	\draw[thick,dashed] (3/4,-5) to[out=70,in=310](.5,-3.5);
	\draw[thick,dashed] (.5,3.5) to[out=130,in=250](1/4,5);
	\draw[thick,dashed] (1/5,-5) to[out=80,in=230](1/2,-4.5);
	\draw[thick,dashed] (1/2,4.5) to[out=50,in=250](4/5,5);
	\zline{0}{black};
\end{tikzpicture}
\caption{$\M=H^{*,*}(pt;\uZ)$ and $\A_0=H^{*,*}(C_2;\uZ)$.}
\label{fig:pointzintro}
\end{figure}
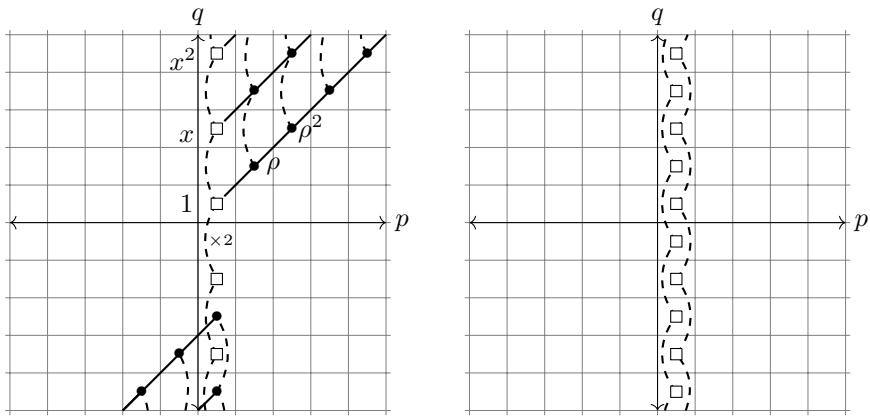

The module $\A_0$ belongs to a family of $\M$-modules that can be described as follows. Let $S^n_{a}$ denote the $n$-sphere with the antipodal action. Observe $C_2=S^0_a$. Denote by $F_i$ the module $x^{-1} \M/(\rho^{i+1})$. Then 
\[\A_n=H^{*,*}(S^n_a;\uZ)\cong F_n \oplus \Sigma^{n,\epsilon}F_0,\]
where $\epsilon =0$ if $n$ is odd and $\epsilon =1$ if $n$ is even. These modules will appear often in our surface computations.

We now turn to the main focus of this paper: closed surfaces with a $C_2$-action. Computations in the orientable case with $\uZ$-coefficients were already done in \cite{LFdS14} using the language of real algebraic curves. In this paper we consider all $C_2$-surfaces and use equivariant surgery techniques that were introduced in \cite{Dug19}.

There are two cases: when the $C_2$-action is free (i.e. there are no fixed points) and when the $C_2$-action is nonfree (i.e. there are fixed points). We start by describing the answer in the free case. Observe that if we have a free $C_2$-surface $X$ and a non-equivariant surface $Y$, then we can construct another free $C_2$-surface via \emph{equivariant connected sum}. This is done by removing two conjugate disks in $X$ and attaching two conjugate copies of $Y$. Non-equivariantly this space is just $Y\# X \# Y$. 

We will prove the following:

\begin{theorem}
Let $X$ be a $C_2$-surface with a free $C_2$-action. Then 
\begin{itemize}[leftmargin=*]
\item $H^{*,*}(X;\uZ)$ is isomorphic to a direct sum of pairs of shifted modules of the form $F_i=\rho^{-1}\M/(x^{i+1})$ for $i=0,1,2$.
\item The exact shifts and numbers of summands depend only on the orientability of the surface, whether the action is orientation reversing versus preserving (in the orientable case), the genus, and whether the surface can be built via equivariant connected sum from a free torus versus a free sphere. 
\end{itemize}
\end{theorem}

The above theorem is stated in full detail in Theorem \ref{freeanswer}. For now, we provide two examples to give the reader a feel for the computations. 

\begin{example} 
Give the genus one torus a free $C_2$-action by rotating $180^\circ$ around an axis through the donut hole, as illustrated in Figure \ref{fig:freeintro1} below. Denote this $C_2$-space by $T_1^{rot}$. We will prove $H^{*,*}(T_1^{rot};\uZ)\cong \A_1 \oplus \Sigma^{1,0}\A_1.$
This module is shown on the right in Figure \ref{fig:freeintro1}. To simplify the picture, the $x$-actions are suppressed because they are isomorphisms in every bidegree.

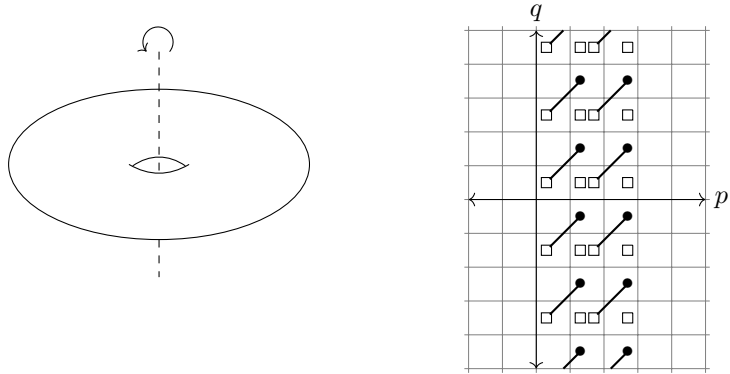
\begin{figure}[ht]
	\begin{subfigure}[b]{0.45\textwidth}
	\centering
	\begin{tikzpicture}[scale=1]
		\draw[->] (.14,1.5) arc (-40:220:.20cm);
		\draw (0,0) ellipse (2cm and 1cm);
		\draw (-0.4,0) to[out=330,in=210] (.4,0) ;
		\draw (-.35,-.02) to[out=35,in=145] (.35,-.02);
		\draw[dashed] (0,1.5)--(0,-0.13);
		\draw[dashed] (0,-1)--(0,-1.5);
	\end{tikzpicture}
	\vspace{0.5in}
	\end{subfigure}
	\begin{subfigure}[b]{0.45\textwidth}
	\centering
	\begin{tikzpicture}[scale=.45]
	\draw[help lines,gray] (-2.125,-5.125) grid (5.125, 5.125);
	\draw[<->] (-2,0)--(5,0)node[right]{$p$};
	\draw[<->] (0,-5)--(0,5)node[above]{$q$};
	\zantizo{-.2}{black};
	\zantizo{1.2}{black};
	\end{tikzpicture}
	\end{subfigure}
	\caption{$T_1^{rot}$ and $H^{*,*}(T_1^{rot};\uZ)\cong \A_1 \oplus \Sigma^{1,0}\A_1$.}
	\label{fig:freeintro1}
\end{figure}

\end{example}

\begin{example}
We next give an example of a Klein bottle with a free $C_2$-action to show how the answer differs when the surface is nonorientable. Construct a Klein bottle with a $C_2$-action by attaching via connected sum conjugate copies of $\R P^2$ to $S^2_a$. The attached copies of $\R P^2$ are shown as circles with $\times$'s in Figure \ref{fig:freeintro2}. The ``$\times$'' is used to remember the circles have antipodal boundary identifications. The lower circle is dashed to indicate that copy of $\R P^2$ is attached to the back of the sphere; this is conjugate to the upper non-dashed copy under the antipodal action. Denote this $C_2$-surface by $K_{free}$. We will show $H^{*,*}(K_{free};\uZ)\cong F_2 \oplus \Sigma^{1,0} F_1$. This module is shown in Figure \ref{fig:freeintro2}. The isomorphic action by $x$ is again suppressed from the picture.

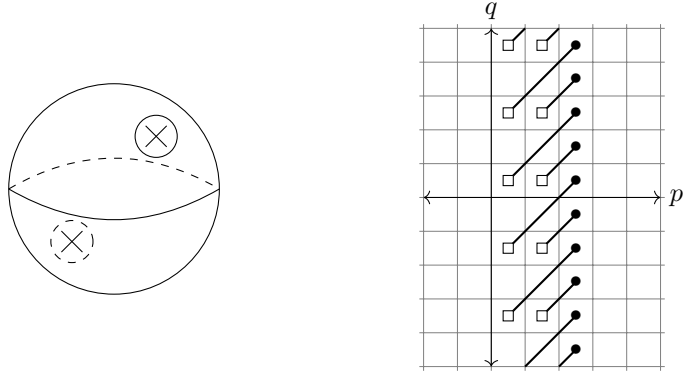
\begin{figure}[ht]
	\begin{subfigure}[b]{0.45\textwidth}
	\centering
	\begin{tikzpicture}[scale=1.4]
		\draw (0,0) circle (1cm);
		\draw (-1,0) to[out=330,in=210](1,0);
		\draw[dashed] (-1,0) to[out=30,in=150](1,0);
		\draw (0.4,0.5) circle (0.2cm);
		\draw (.5,.6)--(.3,.4);
		\draw (.5,.4)--(.3,.6);
		\draw[dashed] (-0.4,-.5) circle (0.2cm);
		\draw (-.5,-.6)--(-.3,-.4);
		\draw (-.5,-.4)--(-.3,-.6);
	\end{tikzpicture}
	\vspace{0.4in}
	\end{subfigure}
	\begin{subfigure}[b]{0.45\textwidth}
	\centering
	\begin{tikzpicture}[scale=0.45]
	\draw[help lines,gray] (-2.125,-5.125) grid (5.125, 5.125);
	\draw[<->] (-2,0)--(5,0)node[right]{$p$};
	\draw[<->] (0,-5)--(0,5)node[above]{$q$};
	\draw[thick] (1.5,4.5)--(2,5);
	\draw[thick](0.5,4.5)--(1,5);
	\zbox{0}{4}{black};
	\foreach \y in {-4,-2,0,2}
		{\draw[thick] (0.5,\y+0.5)--(2.5, \y+2.5) node{\dot};
		\zbox{0}{\y}{black};};
	\foreach\y in {-4,-2,0,2}
		\draw[thick] (1.5,\y+.5)--(2.5,\y+1.5) node{\small{$\bullet$}};
	\foreach\y in {-4,-2,0,2,4}
		\zbox{1}{\y}{black};
	\draw[thick] (1,-5)--(2.5,-3.5) node{\small{$\bullet$}};
	\draw[thick] (2,-5)--(2.5,-4.5) node{\small{$\bullet$}};
	\end{tikzpicture}
	\end{subfigure}
	\caption{$K_{free}$ and $H^{*,*}(K_{free};\uZ)\cong F_2 \oplus \Sigma^{1,0}F_1$.}
	\label{fig:freeintro2}
\end{figure}

Observe $H^{2,j}(K_{free};\uZ)\cong \Z/2$ for all $j$, whereas in the orientable example, the groups swap between $\Z$ and $\Z/2$ along the line $(2,j)$. This pattern can be observed for all free $C_2$-surfaces. 

\end{example}

We now consider the case when the action is nonfree. When the underlying surface is orientable, we show the cohomology is a direct sum of shifted copies of $\M$ and $\A_0$. But two new $\M$-modules appear in the nonorientable case. We illustrate these in examples before stating the main theorem. 

\begin{example}
Construct $\R P^2$ as a quotient of the disk with the antipodal boundary identifications. We can then give $\R P^2$ a $C_2$-action by rotating the disk $180^\circ$, as illustrated in Figure \ref{fig:rp2twintro}. Denote the resulting $C_2$-surface by $\R P^2_{twist}$. The cohomology is shown to the right, where all $x$-actions that are isomorphisms are again omitted. The only $x$-action shown is not an isomorphism: the $x$-action from $(2,-1)$ to $(2,1)$ is $0$.

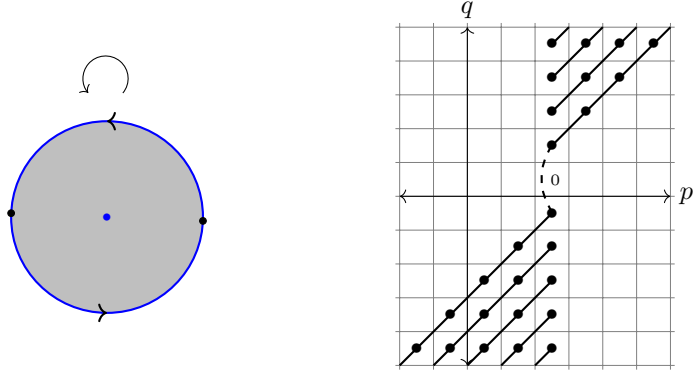
\begin{figure}[ht]
	\begin{subfigure}[b]{0.45\textwidth}
	\centering
\begin{tikzpicture}[scale=1.5]
	\draw[blue,thick, fill=light-gray, decoration={markings, mark=at position 0 with {\arrow[black,scale=.5]{*}}, mark=at position 0.25 with {\arrow[black]{>}}, mark=at position 0.5 with {\arrow[black,scale=.5]{*}}, mark=at position 0.75 with {\arrow[black]{>}}}, postaction={decorate}] (0,0) ellipse (0.85 cm and 0.85 cm);
	\draw[blue] (0,0)node {\tiny{$\bullet$}};
	\draw[->] (.14,1.1) arc (-40:220:.20cm);
\end{tikzpicture}
	\vspace{0.25in}
	\end{subfigure}
	\begin{subfigure}[b]{0.45\textwidth}
	\centering
\begin{tikzpicture}[scale=0.45]
	\draw[help lines,gray] (-2.125,-5.125) grid (6.125, 5.125);
	\draw[<->] (-2,0)--(6,0)node[right]{$p$};
	\draw[<->] (0,-5)--(0,5)node[above]{$q$};
	\draw[thick] (2.5,1.5) node{\dot}--(6,5);
	\draw[thick] (2.5,3.5) node{\dot}--(4,5);
	\draw[thick] (2.5,-1.5) node{\dot}--(-1,-5);
	\draw[thick] (2.5,-3.5) node{\dot}--(1,-5);
	\draw[thick] (2.5,2.5) node{\dot}--(5,5);
	\draw[thick] (2.5,4.5) node{\dot}--(3,5);
	\draw[thick] (2.5,-0.5) node{\dot}--(-2,-5);
	\draw[thick] (2.5,-2.5) node{\dot}--(0,-5);
	\draw[thick] (2.5,-4.5) node{\dot}--(2,-5);
	\foreach \y in {1,2,3}
		\draw (2.5+\y,1.5+\y) node{\dot};
	\foreach \y in {1,2}
		\draw (2.5+\y,2.5+\y) node{\dot};
	\foreach \y in {1,2,3,4}
		\draw (2.5-\y, -0.5-\y) node{\dot};
	\foreach \y in {1,2,3}
		\draw (2.5-\y, -1.5-\y) node{\dot};
	\foreach \y in {1,2}
		\draw (2.5-\y, -2.5-\y) node{\dot};
	\draw (3.5,4.5) node{\dot};
	\draw (1.5,-4.5) node{\dot};
	\draw[thick,dashed] (2+1/2,1+1/2) arc (145:216:1.7);
	\draw (2.6,0.5) node{\tiny{$0$}};
\end{tikzpicture}
	\end{subfigure}
\caption{ $\R P^2_{twist}$ and $\H^{*,*}(\R P^2_{twist};\uZ)\cong \Sigma^{2,1}\M_2$.}
\label{fig:rp2twintro}
\end{figure}
\end{example}

\begin{rmk}
The module in the previous example can be realized in another way. There is a surjective map of Mackey functors $\uZ\to \underline{\Z/2}$. This induces a map of bigraded rings $H^{*,*}(pt;\uZ)\to H^{*,*}(pt;\underline{\Z/2})$ which makes $H^{*,*}(pt;\underline{\Z/2})$ into a module over $\M=H^{*,*}(pt;\uZ)$. The ring $H^{*,*}(pt;\underline{\Z/2})$ is often denoted by $\M_2$, and $\H^{*,*}(\R P^2_{twist};\uZ)\cong \Sigma^{2,1}\M_2$ as an $\M$-module. This module is described more in Remark \ref{rempointz2}.
\end{rmk}

\begin{example}\label{introkleinex}

Construct the Klein bottle $K$ as the quotient of a square using the boundary identifications shown in Figure \ref{fig:kintro}. Give $K$ a $C_2$-action via the illustrated reflection action. Denote the resulting $C_2$-surface by $K_{refl}$. The cohomology is shown to the right, where in the picture the number four within a circle denotes a copy of $\Z/4$. Again the $x$-action is omitted when it is an isomorphism. In bidegree $(2,-1)$ the $x$-action gives the inclusion $\Z/2 \hookrightarrow \Z/4$, $1\mapsto 2$ and is thus labeled ``$2$''. Denote by $\D_4$ the module given by shifting $\H^{*,*}(K_{refl};\uZ)$ to the left one unit so that the lowest $\Z/4$ is in bidegree $(1,1)$ and a copy of $\Z$ is in bidegree $(0,0)$.

\begin{figure}[ht]
	\begin{subfigure}[b]{0.45\textwidth}
	\centering
\begin{tikzpicture}[scale=1.2]
	\draw[fill=light-gray] (2,0)--(0,0)--(0,2)--(2,2)--(2,0);
	\draw (1,2) node{$>$};
	\draw (1,0) node{$<$};
	\draw (0,0.9) node{$\wedge$};
	\draw (0,1.1) node{$\wedge$};
	\draw (2,0.9) node{$\wedge$};
	\draw (2,1.1) node{$\wedge$};
	\draw[blue, very thick] (0,1)--(2,1);
	\draw[dashed] (-.45,1)--(2.45,1);
	\draw[<->] (2.4,1.2)--(2.4,.8);
\end{tikzpicture}
\vspace{0.43in}
\end{subfigure}
\begin{subfigure}[b]{0.45\textwidth}
	\centering
\begin{tikzpicture}[scale=0.45]
\draw[help lines,gray] (-2.125,-5.125) grid (6.125, 5.125);
\draw[<->] (-2,0)--(6,0)node[right]{$p$};
\draw[<->] (0,-5)--(0,5)node[above]{$q$};
	\draw[thick,dashed] (2+1/2,1+1/2) arc (145:216:1.7);
	\draw (2.6,.5) node{\tiny{$2$}};
	\foreach\y in {0,2,4}
	\draw[thick] (1.5,\y+.5)--(6-\y,5);
	\foreach\y in {-2,-4}
	\draw[thick] (1.5,\y+.5)--(2.5,\y+1.5);
	\draw[thick] (2.5,-4.5)--(2,-5);
	\foreach \y in {-4,-2,0,2,4}
	\zbox{1}{\y}{black};
	\foreach \y in {1,3}
	{\draw[fill=white](2.5,\y+.5) circle (.27cm);
	\draw (2.52,\y+.51) node{\scalebox{.55}{$4$}};
	};
	\foreach \y in {-1,-3,-5}
	\draw (2.52,\y+.51) node{\dot};
	\draw[thick] (2.5,-1.5)node{\dot}--(-1,-5);
	\draw[thick] (2.5,-3.5)node{\dot}--(1,-5);
	\foreach \y in {1,2,3}
		\draw (2.5-\y, -1.5-\y) node{\dot};
	\foreach \y in {1,2,3}
		\draw (2.5+\y,1.5+\y) node{\dot};
	\draw (1.5,-4.5) node{\dot};
	\draw (3.5,4.5) node{\dot};
\end{tikzpicture}
\end{subfigure}
\caption{$K_{refl}$ and $\H^{*,*}(K_{refl};\uZ)\cong \Sigma^{1,0}\D_4$.}
\label{fig:kintro}
\end{figure}
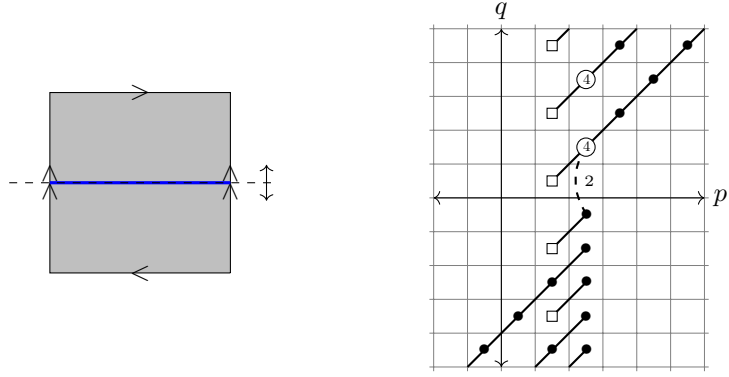
\end{example}

We now state the main computational theorems in the case where the action is nonfree and nontrivial. These are proven in Theorems \ref{nonfreeoranswerz} and \ref{nonfreenonoranswer}, respectively. For a $C_2$-surface $X$ with a nontrivial action, we use the following notation:
\begin{align*}
&F= \# \text{ of isolated fixed points}, \\
&C= \# \text{ of fixed circles},\\
&C_+= \# \text{ of fixed circles with trivial normal bundles},\\
&C_-= \# \text{ of fixed circles with nontrivial normal bundles}, \text{ and}\\
&\beta = \dim_{\Z/2} H^1_{sing}(X;\Z/2).
\end{align*}
Note for orientable surfaces, the value of $\beta$ is two times the genus. Also note $C=C_++C_-$.

\begin{theorem} Let $X$ be a nonfree, nontrivial $C_2$-surface whose underlying space is orientable. There are two cases for the cohomology of $X$ based on the fixed set:
\begin{enumerate}
\item[(i)] $F\neq 0$. Then \[\H^{*,*}(X)\cong \left(\Sigma^{1,1}\M\right)^{\oplus F-2} \oplus \left(\Sigma^{1,0}\A_0\right)^{\oplus \frac{\beta-F}{2}+1} \oplus \Sigma^{2,2}\M.\]
\item[(ii)] $C\neq 0$. Then
\[\H^{*,*}(X)\cong \left(\Sigma^{1,0}\M\right)^{\oplus C-1} \oplus \left(\Sigma^{1,1}\M\right)^{\oplus C-1}\oplus \left(\Sigma^{1,0}\A_0\right)^{\oplus \frac{\beta-2C}{2}+1} \oplus \Sigma^{2,1}\M.\]
\end{enumerate}
\end{theorem}

\begin{theorem}
 Let $X$ be a nontrivial, nonfree $C_2$-surface whose underlying space is nonorientable. Then there are three cases for the cohomology of $X$ based on the fixed set:
 \begin{enumerate}
 \item[(i)] $F \neq 0 $, $C=0$. Then 
\[\tilde{H}^{*,*}(X) \cong \left(\Sigma^{1,0}\A_0\right)^{\oplus \frac{\beta -F}{2}}\oplus \left(\Sigma^{1,1}\M\right)^{\oplus F-2} \oplus\Sigma^{1,1}\D_4.\]
 \item[(ii)] $F =0 $, $C_+\neq0$, $C_-=0$. Then
 \[\tilde{H}^{*,*}(X) \cong \left(\Sigma^{1,0}\A_0\right)^{\oplus \frac{\beta -2C}{2}}\oplus \left(\Sigma^{1,0}\M\right)^{\oplus C-1}\oplus \left(\Sigma^{1,1}\M\right)^{\oplus C-1} \oplus\Sigma^{1,0}\D_4.\]
 \item[(iii)] $F \neq 0 $ and $C_+\neq0$, or $C_-\neq0$. Then 
 \[\tilde{H}^{*,*}(X) \cong \left(\Sigma^{1,0}\A_0\right)^{\oplus \frac{\beta -(F+2C)}{2}+1}\oplus \left(\Sigma^{1,0}\M\right)^{\oplus C-1}\oplus \left(\Sigma^{1,1}\M\right)^{\oplus F+C-2} \oplus\Sigma^{2,1}\M_2.\]
 \end{enumerate}
\end{theorem}

\begin{rmk}
In cases (i) and (ii) there is exactly one summand of the form $\Sigma^{1,\epsilon}\D_4$, which gives a tower of $\Z/4$'s in topological degree $p=2$. In case (iii) there is exactly one summand of the form $\Sigma^{2,1}\M_2$, which gives a tower of $\Z/2$'s in topological degree $p=2$. In fact, we show for a nonorientable $C_2$-manifold of dimension $n$, there is always either a $\Z/4$ or $\Z/2$ in topological degree $p=n$ and in a predicted weight $q$. See Proposition \ref{nonortopm}.

The two possibilities of $\Z/4$ versus $\Z/2$ hint at different levels of nonorientability. The $C_2$-surface answer suggests this depends on properties of the fixed set. Observe the fixed set is especially nice in (i) and (ii): the dimension is constant across components, and the normal bundle of each component is orientable. This is not the case in (iii): here the fixed set either has components of different dimension, or has a component with a nonorientable normal bundle. Thus, for $C_2$-surfaces whose underlying space is nonorientable, the module $\D_4$ appears when the the fixed set is ``nice'', while the module $\M_2$ appears when the fixed set is ``not nice''. It is unclear if such a pattern persists for the cohomology of higher dimensional $C_2$-manifolds.
\end{rmk}

\subsection{Organization of the paper} Section \ref{sec:background} has background on $RO(C_2)$-graded cohomology and computational techniques. Section \ref{sec:somemods} introduces a family of $\M$-modules $\D_{4n}$ and a corresponding family of $C_2$-CW complexes whose cohomology is a shift of $\D_{4n}$. One of these $C_2$-CW complexes will be the Klein bottle described in Example \ref{introkleinex}. In Section \ref{sec:sphereimage} we explore general $C_2$-manifolds and the image of the map induced by quotienting to a representation sphere; this can be thought of as finding a top cohomology class that is the fundamental class for a fixed point in the manifold. The final sections focus on $C_2$-surfaces. Background on surgery methods and the classification is given in Section \ref{sec:backgroundsurg}. The computation for free $C_2$-surfaces is in Section \ref{sec:freecomp}, and the computation for nonfree $C_2$-surfaces is in Section \ref{sec:nonfreecomp}.

\subsection{Notation and conventions} Given a $C_2$-space $X$, we can forget the $C_2$-action and just consider the topological space. We will refer to this as the ``underlying space'' and still write $X$, ensuring that it is always clear from context whether we are referring to the $C_2$-space or the underlying space. We will write $X^{C_2}$ for the fixed set and $X/C_2$ for the quotient space. Throughout this paper, $C_2$-manifold or $C_2$-surface will mean a piecewise linear manifold with a locally linear $C_2$-action whose underlying space is, unless stated otherwise, a closed (connected, no boundary) manifold. These assumptions guarantee a $C_2$-manifold $M$ is a finite $C_2$-CW complex and that $M^{C_2}$ is a disjoint union of finitely many submanifolds, though note the dimensions of these submanifolds need not be constant. 

We will often omit coefficients and write $H^{*,*}(X)$ for $H^{*,*}(X;\uZ)$. We will be careful to distinguish between reduced cohomology and unreduced cohomology; $H^{*,*}(X)$ will always denote unreduced cohomology. 

\subsection{Acknowledgements} Much of this work was done while the author was a graduate student at University of Oregon. She thanks her doctoral advisor, Daniel Dugger, for all of his advice and guidance on this project. She also thanks Mike Hill and Clover May for many helpful conversations.

%%%%%%%%%%%%%%%%
\section{Background on $RO(C_2)$-graded cohomology}\label{sec:background}

We review the necessary facts about $RO(C_2)$-graded cohomology in this section. For a more thorough background on $RO(G)$-graded cohomology see \cite[Chapter IX]{Maybook} or \cite[Section 2]{HHR}.

The grading $RO(C_2)$ is the Grothendieck ring of finite-dimensional, real, orthogonal $C_2$-representations. This is a free abelian group of rank $2$ with generators given by the isomorphism classes of the one-dimensional trivial representation $[\R_{triv}]$ and of the sign representation $[\R_{sgn}]$. Every element in $RO(C_2)$ can be written as \[\R^{p,q}:=(p-q)[\R_{triv}]+q[\R_{sgn}].\] We will write $H^{p,q}(-;\underline{M})$ for $H^{\R^{p,q}}(-;\underline{M})$ and refer to the first grading as the {\bf{topological degree}} and the second grading as the {\bf{weight}}.

The coefficients of Bredon cohomology are a Mackey functor. For the group $C_2$, a {\bf{Mackey functor}} $\underline{M}$ is a choice of two abelian groups $\underline{M}(C_2/e)$ and  $\underline{M}(C_2/C_2)$ for the two orbits, and a choice of maps
\[p^*: \underline{M}(C_2/C_2) \to \underline{M}(C_2/C_2), \quad \quad p_*:\underline{M}(C_2/e)\to \underline{M}(C_2/C_2),\]
\[ t^*, t_*: \underline{M}(C_2/e)\to \underline{M}(C_2/e).\]
These maps also must satisfy some compatibility relations that we won't recall here. In Bredon cohomology with $\underline{M}$ coefficients, we have that 
\[H^{0,0}(C_2/C_2;\underline{M})=\underline{M}(C_2/C_2), \quad \quad H^{0,0}(C_2/e;\underline{M})=\underline{M}(C_2/e),\]
and the maps induced by the non-identity maps of orbits $p:C_2/e\to C_2/C_2$ and $t:C_2/e\to C_2/e$ are exactly the maps $p^*$ and $t^*$ from the Mackey functor. If $\underline{M}$ has the additional structure that it is a Green functor (also sometimes called a Mackey ring), then the cohomology is a bigraded ring.

We are interested in constant Mackey functors in this paper. For an abelian group $B$, the {\bf{constant Mackey functor}} $\underline{B}$ is defined so that 
\[\underline{B}(C_2/C_2)=\underline{B}(C_2/e)=B, \quad \quad  p^*=t^*=t_*=1, \text{ and} \quad \quad p_*=2.\] We will mostly be concerned with $\uZ$, but will also occasionally consider $\underline{\Z/2}$.\medskip 

Bredon cohomology is an ordinary $RO(C_2)$-graded cohomology theory, and thus for the constant Mackey functor $\underline{B}$,
\[
H^{p,0}(C_2;\underline{B}) =H^{p,0}(pt;\underline{B})= \begin{cases} B, \quad \quad &p=0\\ 0, \quad \quad &p\neq 0. \end{cases}
\] This vanishing does not have to persist when the weight is nonzero though, and indeed, the cohomology of both orbits is nonzero in infinitely many bidegrees. Since the cohomology is a bigraded ring, we use a grid to record information about the groups and other algebraic structures. Figure \ref{fig:pointz} shows the cohomology of a point, which we denote by $\M=H^{*,*}(pt;\uZ)$ and will describe shortly. This computation is often attributed to unpublished notes of Stong. Lewis also computed the cohomology of a point in coefficients given by any constant Mackey functor in \cite[Section 2]{L87}, and Dugger recomputed $H^{*,*}(pt;\uZ)$ in \cite[Appendix B]{Dug06}. The notation in Dugger's computation matches the notation given in this paper.
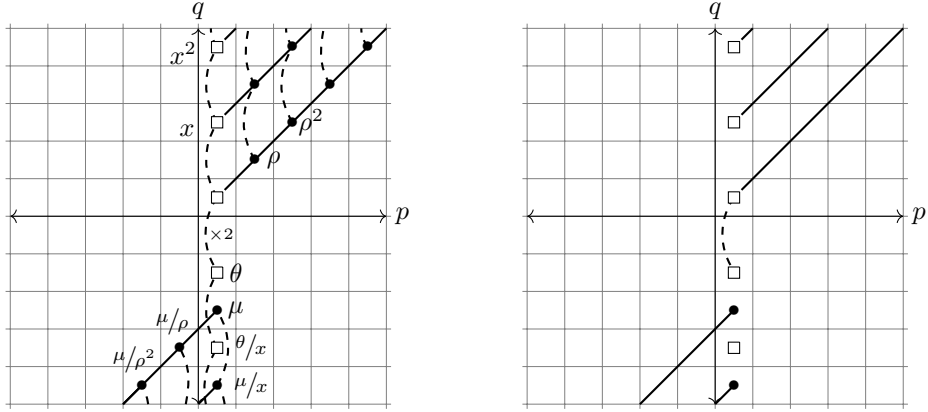
\begin{figure}[ht]
\begin{tikzpicture}[scale=.5]
\draw[help lines,gray] (-5.125,-5.125) grid (5.125, 5.125);
\draw[<->] (-5,0)--(5,0)node[right]{$p$};
\draw[<->] (0,-5)--(0,5)node[above]{$q$};
	\draw[thick,dashed] (1/2,1/2) to[out=120,in=240](1/2,2.5);
	\draw[thick,dashed] (1/2,2+1/2) to[out=120,in=240](1/2,4.5);
	\draw[thick,dashed] (1+1/2,1+1/2) to[out=120,in=240](1+1/2,3.5);
	\draw[thick,dashed] (1+1/2,3+1/2) to[out=120,in=260](1+1/3,5);
	\draw[thick,dashed] (2+1/2,2+1/2) to[out=120,in=240](2+1/2,4.5);
	\draw[thick,dashed] (3+1/2,3+1/2) to[out=120,in=260](3+1/3,5);
	\draw[thick,dashed] (4+1/2,4+1/2) to[out=120,in=270](4+1/3,5);
	\draw[thick,dashed] (2+1/2,4+1/2) to[out=120,in=270](2+1/3,5);
	\draw[thick,dashed] (1/2,-4+1/2) to[out=120,in=240](1/2,-1.5);
	\draw[thick,dashed] (1/2,-5+1/2) to[out=60,in=300](1/2,-2.5);
	\draw[thick,dashed] (-1/3,-5) to[out=80,in=300](-.5,-3.5);
	\draw[thick,dashed] (-4/3,-5) to[out=100,in=300](-1.5,-4.5);
	\draw[thick,dashed] (.7,-5) to[out=100,in=310](1/2,-4.5);
	\draw[thick,dashed] (1/5,-5) to[out=100,in=240](1/2,-3.5);
	\draw[thick,dashed] (1/2,4+1/2) to[out=120,in=270](1/3,5);
	\foreach \x in {1,2,3,4}
	\draw (\x+.5,\x+.5) node{\small{$\bullet$}};
	\foreach \x in {1,2}
	\draw (\x+.5,\x+2.5) node{\small{$\bullet$}};
	\zcone{0}{0}{black};
	\draw[thick] (.6,-.5) node{\tiny{$\times 2$}};
	\draw[thick] (-.5,-3.5) node{\small{$\bullet$}};
	\draw[thick](-1.5,-4.5)node{\small{$\bullet$}};
	\draw[thick](-1.5,-4.5)--(-2,-5);
	\draw (-.3,2.3) node{$x$};
	\draw (-.4,4.35) node{$x^2$};
	\draw (2,1.5) node{$\rho$};
	\draw (3,2.5) node{$\rho^2$};
	\draw (1,-1.5) node{$\theta$};
\draw (1,-2.5) node{$\mu$};
\draw (1.4,-3.5) node{$\sfrac{\theta}{x}$};
\draw (-.7,-2.8) node{$\sfrac{\mu}{\rho}$};
\draw (-1.7,-3.8) node{$\sfrac{\mu}{\rho^2}$};
\draw (1.4,-4.5) node{$\sfrac{\mu}{x}$};
\end{tikzpicture}\hspace{0.5in}
\begin{tikzpicture}[scale=.5]
\draw[help lines,gray] (-5.125,-5.125) grid (5.125, 5.125);
\draw[<->] (-5,0)--(5,0)node[right]{$p$};
\draw[<->] (0,-5)--(0,5)node[above]{$q$};
\zcone{0}{0}{black};
\end{tikzpicture}
\caption{The ring $\M=H^{*,*}(pt;\underline{\Z})$ and an abbreviated picture.}
\label{fig:pointz}
\end{figure}

In Figure \ref{fig:pointz} a box is used to denote a copy of $\Z$ and a dot is used to denote a copy of $\Z/2$. We plot information about the $(p,q)$ group up and to the right of the $(p,q)$ lattice point. For example, $H^{0,2}(pt;\uZ)=\Z$, $H^{1,1}(pt;\uZ)=\Z/2$, and $H^{0,1}(pt;\uZ)=0$. The portion of $\M$ in the first quadrant is isomorphic to $\Z[x,\rho]/(2\rho)$ where $x$ is in bidegree $(0,2)$ and $\rho$ is in bidegree $(1,1)$. In the bottom portion of the picture, there is an element $\theta$ in bidegree $(0,-2)$ that generates a copy of $\Z$ and is infinitely divisible by $x$. That is, there are elements $\theta/x^j$ for all $j\geq 1$ such that $x^j\cdot \theta/x^{j} = \theta$. There is also a $2$-torsion element $\mu$ in bidegree $(0,-3)$ that is infinitely divisible by both $\rho$ and $x$. The dashed vertical lines are used to indicate action by $x$, while the diagonal lines are used to indicate action by $\rho$. The action by $x$ is an isomorphism whenever possible, except in bidegree $(0,-2)$ where $x\theta =2$. 

The right-hand grid in Figure \ref{fig:pointz} shows an abbreviated picture of the free module that we will use in computations. We omit the $x$-connections that are isomorphisms, and just keep one dashed line to remember the action by $x$ is multiplication by $2$ in this bidegree.

\begin{rmk}\label{rempointz2} The ring $\M_2=H^{*,*}(pt;\underline{\Z/2})$ has similar structure. See Figure \ref{fig:pointz2}. We have a map of Green functors $\uZ\to \underline{\Z/2}$ given by the surjective map $\Z\to \Z/2$ at each orbit. This induces a ring map $\M\to \M_2$ where $\rho\mapsto \rho$, $x\mapsto \tau^2$, $\theta\mapsto \theta$, and $\mu\mapsto \theta/\tau$. This gives $\M_2$ an $\M$-module structure that is illustrated to the right in Figure \ref{fig:pointz2}. As usual, the $x$-actions that are isomorphisms are omitted for brevity, but we keep track of the action in bidegree $(0,-2)$ since $x$ acts trivially here. 
\end{rmk}

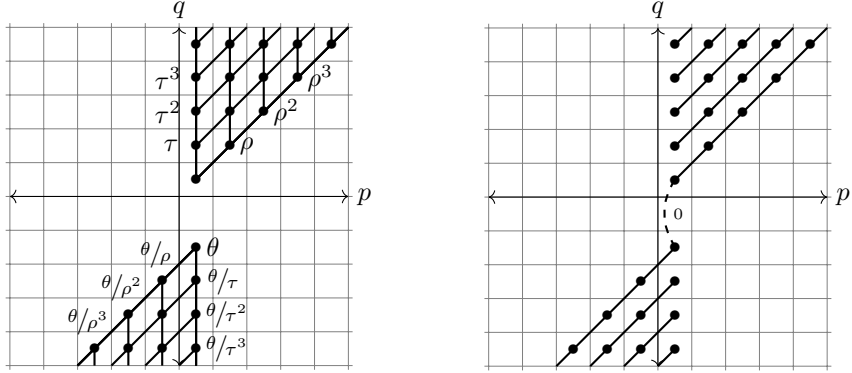
\begin{figure}[ht]
\begin{tikzpicture}[scale=.45]
\draw[help lines,gray] (-5.125,-5.125) grid (5.125, 5.125);
\draw[<->] (-5,0)--(5,0)node[right]{$p$};
\draw[<->] (0,-5)--(0,5)node[above]{$q$};
\foreach \y in {0,1,2,3,4}
\draw (0.5,\y+.5) node{\small{$\bullet$}};
\cone{0}{0}{black};
\foreach \y in {1,2,3,4}
\draw (1.5,\y+.5) node{\small{$\bullet$}};
\foreach \y in {2,3,4}
\draw (2.5,\y+.5) node{\small{$\bullet$}};
\foreach \y in {3,4}
\draw (3.5,\y+.5) node{\small{$\bullet$}};
\foreach \y in {4}
\draw (4.5,\y+.5) node{\small{$\bullet$}};
\foreach \y in {-1,-2,-3,-4}
\draw (.5,\y-.5) node{\small{$\bullet$}};
\cone{0}{0}{black};
\foreach \y in {-2,-3,-4}
\draw (-.5,\y-.5) node{\small{$\bullet$}};
\foreach \y in {-3,-4}
\draw (-1.5,\y-.5) node{\small{$\bullet$}};
\foreach \y in {-4}
\draw (-2.5,\y-.5) node{\small{$\bullet$}};
\cone{0}{0}{black};
\draw[thick](1.5,5)--(1.5,1.5)node[below, right]{$\rho$};
\draw[thick](2.5,5)--(2.5,2.5)node[below, right]{$\rho^2$};
\draw[thick](3.5,5)--(3.5,3.5)node[below, right]{$\rho^3$};
\draw[thick](4.5,4.5)--(4.5,5);
\draw[thick](.5,1.5)node[xshift=-2.2ex]{$\tau$}--(4,5);
\draw[thick](.5,2.5)node[xshift=-2.45ex]{$\tau^2$}--(3,5);
\draw[thick](.5,3.5)node[xshift=-2.45ex]{$\tau^3$}--(2,5);
\draw[thick](.5,4.5)--(1,5);
\draw[thick](-.5,-2.5)--(-.5,-5);
\draw[thick](-1.5,-3.5)--(-1.5,-5);
\draw[thick](-2.5,-4.5)--(-2.5,-5);
\draw[thick](.5,-2.5)--(-2,-5);
\draw[thick](.5,-3.5)--(-1,-5);
\draw[thick](.5,-4.5)--(0,-5);
\draw (1,-1.5) node{$\theta$};
\draw (1.3,-2.5) node{$\sfrac{\theta}{\tau}$};
\draw (1.4,-3.5) node{$\sfrac{\theta}{\tau^2}$};
\draw (1.4,-4.5) node{$\sfrac{\theta}{\tau^3}$};
\draw (-.7,-1.7) node{$\sfrac{\theta}{\rho}$};
\draw (-1.7,-2.7) node{$\sfrac{\theta}{\rho^2}$};
\draw (-2.7,-3.7) node{$\sfrac{\theta}{\rho^3}$};
\end{tikzpicture}\hspace{0.5 in}
\begin{tikzpicture}[scale=.45]
\draw[help lines,gray] (-5.125,-5.125) grid (5.125, 5.125);
\draw[<->] (-5,0)--(5,0)node[right]{$p$};
\draw[<->] (0,-5)--(0,5)node[above]{$q$};
\draw[thick,dashed] (1/2,1/2) arc (145:216:1.7);
\draw (.6,-.5) node{\tiny{$0$}};
\foreach\y in {0,2,4}
\draw[thick] (0.5,\y+.5)--(5-\y,5);
\foreach \y in {-4,-2}
\draw[thick] (.5,\y+.5)--(-5-\y,-5);
\foreach\y in {-4,-2,0,2,4}
\draw (.5,\y+.5) node{\dot};
\foreach\y in{1,3}
\draw[thick] (.5,\y+.5) node{\dot}--(5-\y,5);
\foreach\y in{-5,-3}
\draw[thick] (.5,\y+.5) node{\dot}--(-5-\y,-5);
\foreach \y in {1,2,3,4}
	\draw (1.5,\y+.5) node{\dot};
\foreach \y in {2,3,4}
	\draw (2.5,\y+.5) node{\dot};
\foreach \y in {3,4}
	\draw (3.5,\y+.5) node{\dot};
\foreach \y in {4}
	\draw (4.5,\y+.5) node{\dot};
\foreach \y in {-3,-4,-5}
	\draw (-.5,\y+.5) node{\dot};
\foreach \y in {-4,-5}
	\draw (-1.5,\y+.5) node{\dot};
\foreach \y in {-5}
	\draw (-2.5,\y+.5) node{\dot};
\end{tikzpicture}
\caption{$\M_2=H^{*,*}(pt; \underline{\Z/2})$ as a bigraded ring and as an $\M$-module.}
\label{fig:pointz2}
\end{figure}

Let $X$ be a $C_2$-space. We have an equivariant map $X\to pt$ and so $H^{*,*}(X;\underline{\Z})$ is a module over the ring $\M=H^{*,*}(pt;\underline{\Z})$. Our goal is to understand $H^{*,*}(X;\uZ)$ as an $\M$-module. We use the rest of this section to outline computational techniques. 

Given a finite-dimensional, real, orthogonal $C_2$-representation $V$, we can consider the unit sphere $S(V)$, the unit disk $D(V)$, and the one-point compactification $\hat{V}$. The space $\hat{V}$ is a $C_2$-sphere whose underlying space has dimension equal to the dimension of $V$. We refer to $\hat{V}$ as a {\bf{representation sphere}}. When $V=\R^{p,q}$, we write $S^{p,q}$ for $\hat{\R}^{p,q}$. Note the underlying space is just $S^p$. A few examples of representation spheres are shown in Figure \ref{fig:spheres} with the fixed set shown in blue. 

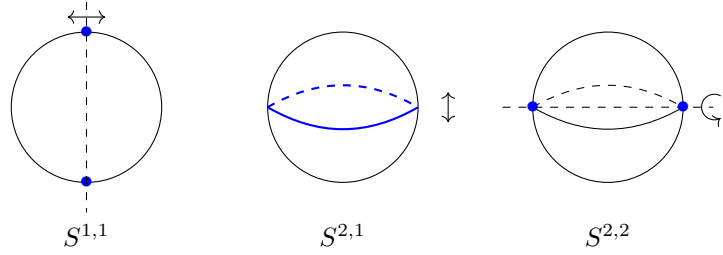
\begin{figure}[ht]
\begin{tikzpicture}
	\draw (0,0) circle (1cm);
	\draw[blue] (0,-1) node{$\bullet$};
	\draw[blue] (0,1) node{$\bullet$};
	\draw[dashed] (0,1.4)--(0,-1.4);
	\draw[<->] (-.25,1.2)--(.25,1.2);
	\draw (0,-1.7) node{$S^{1,1}$};
\end{tikzpicture}\hspace{.5in}
\begin{tikzpicture}
	\draw (0,0) circle (1cm);
	\draw[blue,thick] (-1,0) to[out=330,in=210](1,0);
	\draw[blue,thick, dashed] (-1,0) to[out=30,in=150](1,0);
	\draw (0,-1.7) node{$S^{2,1}$};
	\draw[<->] (1.4,-.2)--(1.4,.2);
\end{tikzpicture}\hspace{.2in}
\begin{tikzpicture}
	\draw (0,0) circle (1cm);
	\draw (-1,0) to[out=330,in=210](1,0);
	\draw[dashed] (-1,0) to[out=30,in=150](1,0);
	\draw[dashed] (-1.4,0)--(1.4,0);
	\draw[->] (1.5,.15) arc (60:300:.17cm);
	\draw (0,-1.7) node{$S^{2,2}$};
	\draw[blue] (1,0)node{$\bullet$};
	\draw[blue] (-1,0)node{$\bullet$};
\end{tikzpicture}
\caption{Some representation spheres.}
\label{fig:spheres}
\end{figure}

If $X$ is a based $C_2$-space, we can form the $(p,q)$-th suspension of $X$ as
\[
\Sigma^{p,q}X:=S^{p,q} \wedge X.
\]
Bredon cohomology is an $RO(C_2)$-graded cohomology theory, which means it has a suspension isomorphism
\[
\H^{*,*}(-;\underline{M}) \cong \H^{*+p,*+q}(\Sigma^{p,q}(-);\underline{M}).
\]
In particular, for any representation sphere $S^{p,q}$, we have that $\H^{*,*}(S^{p,q})=\Sigma^{p,q}\M$ where $\Sigma^{p,q}\M$ is given by shifting the module $\M$ by $(p,q)$. Visually, we just take the picture for $\M$ and shift it over $p$ units and up $q$ units.

Given a cofiber sequence of $C_2$-spaces
\[
A\overset{f}{\to} X \to C(f),
\]
we can form a $C_2$-equivariant version of the usual Puppe sequence:
\[
A \to X \to C(f) \to \Sigma^{1,0} A \to \Sigma^{1,0}X \to \Sigma^{1,0} C(f) \to \dots
\]
Applying $\H^{p,q}(-)$ gives an associated long exact sequence:
\[
\to \H^{p-1,q}(A)\to \H^{p,q}(C(f)) \to \H^{p,q}(X) \to \H^{p,q}(A) \to \H^{p+1,q}(C(f)) \to
\]
In order to form suspensions (and thus Puppe sequences), we need based spaces and based maps. The base point must be fixed, so given a free $C_2$-space $X$, we will often attach a disjoint fixed basepoint and denote this by $X_+=X\sqcup pt$.

In the associated long exact sequence, we have a differential \[d^{p,q}:\H^{p,q}(A) \to \H^{p+1,q}(C(f))\] for each $p,q$. These assemble to give a degree $(1,0)$ $\M$-module map \[d:\H^{*,*}(A) \to \H^{*+1,*}(C(f)).\] 

We finish this section with some helpful computational lemmas. The first is a standard fact relating the Bredon cohomology to the singular cohomology of the quotient space. This statement holds for any constant Mackey functor, not just $\uZ$. A simple proof in $\underline{\Z/2}$-coefficients can be found in \cite[Lemma 3.1]{Haz20}, and the same method will work here.

\begin{lemma}[The quotient lemma]\label{quotient} Let $X$ be a finite $C_2$-CW complex. Then $H^{p,0}(X;\uZ)\cong H^p_{sing}(X/C_2;\Z).$
\end{lemma}

Next recall there is an element $\rho \in \M$ in bidegree $(1,1)$. This element has a geometric representation given by the inclusion of fixed points $\rho:S^{0} \hookrightarrow S^{1,1}$. The cofiber of this map is $C_{2+}\wedge S^{1,0}$ (a wedge of two copies of $S^1$ where the $C_2$-action swaps the copies). Smashing this cofiber sequence with a $C_2$-space $X$ then yields the following long exact sequence, which was first described in \cite{AM}:

\begin{lemma}[The forgetful long exact sequence]\label{fles} Let $X$ be a $C_2$-space. Then we have a long exact sequence
\[
\to H^{p,q}(X;\uZ) \overset{\rho}{\to} H^{p+1,q+1}(X;\uZ) \overset{\psi}{\to} H^{p+1}_{sing}(X;\Z) \to H^{p+1,q}(X;\uZ) \overset{\rho}{\to} 
\]
where $\rho$ denotes the action of $\rho\in \M$ and $\psi$ denotes the forgetful map to singular cohomology.
\end{lemma}

There is another standard fact relating localization by $\rho$ and the singular cohomology of the fixed set. A proof with $\underline{\Z/2}$-coefficients using this notation can be found in \cite[Lemma 4.3]{clover}. The same method will work for $\uZ$-coefficients.

\begin{lemma}[$\rho$-localization]\label{rho_local}
Let $X$ be a finite $C_2$-CW complex. Then $$\rho^{-1}H^{*,*}(X;\underline{\Z})\cong \rho^{-1}\M \otimes_\Z H_{sing}^{*}(X^{C_2};\Z).$$
\end{lemma}

We will often encounter spaces of the form $C_2\times X$ in our computations. The cohomology of such spaces can be easily computed:

\begin{lemma}\label{product}
Let $X$ be a finite $C_2$-CW complex. Then \[H^{*,*}(C_2\times X;\uZ) \cong \A_0\otimes_\Z H^{*}_{sing}(X;\Z).\] If $X$ is pointed, then \[\H^{*,*}(C_{2+}\wedge X;\uZ)\cong \A_0\otimes_\Z \H^{*}_{sing}(X;\Z).\]
\end{lemma}

Lastly, we state a geometric fact that will be helpful in many computations. This was proven in \cite{Haz20} in the proof of Theorem 6.6, with the key step occurring in Figure 23.

\begin{lemma}\label{pinched}
Suppose $X$ is a nontrivial, nonfree, $C_2$-surface. Let $\tilde{X}$ denote the cofiber of the map $C_2\hookrightarrow X$. Then $\tilde{X}\simeq X\vee S^{1,1}$. In particular, $\H^{*,*}(\tilde{X})\cong \H^{*,*}(X) \oplus \Sigma^{1,1}\M$.
\end{lemma}

\subsection{Computational strategy}\label{compstat} Suppose we have a $C_2$-space $X$, and we would like to determine $H^{*,*}(X;\uZ)$ as an $\M$-module. Below is an outline for how we will approach such a computation:
\begin{itemize}
\item \emph{Find a cofiber sequence $A\to X\to Q$ where we know $H^{*,*}(A)$ and $H^{*,*}(Q)$.}
\item \emph{Determine the differential $d:\H^{*,*}(A) \to \H^{*+1,*}(Q)$.} This will be done by first determining what $d$ does to a set of generators for the module $\H^{*,*}(A)$. To do so, we might use the quotient lemma, the forgetful long exact sequence, data from another cofiber sequence with $X$, or facts about $H^{*,*}(X;\underline{\Z/2})$. Once we know how $d$ behaves on generators, we can use the module structure to determine the map.
\item \emph{Solve the extension problem $0 \to \coker(d) \to H^{*,*}(X) \to \ker(d)\to 0$.} This step is often nontrivial, and we will use similar techniques as in the previous bullet point. \end{itemize}

We first employ this strategy in full force in the proof of Lemma \ref{x2ncohom}. We end this section with an example of $\M$-modules that will appear later in the paper.

\begin{example} Let $S^n_a$ denote the $C_2$-sphere whose underlying space is $S^n$ and whose $C_2$-action is given by the antipodal map. Note $S^0_a$ is just the free orbit $C_2$. Write $\A_n$ for the module $H^{*,*}(S^n_a;\uZ)$. 

Observe we can include $S^{n-1}_a$ into $S^n_a$ as the equator. The quotient will be a wedge of two $n$-spheres where the $C_2$-action swaps the two copies. Thus we have a cofiber sequence
\[
S^{n-1}_a \hookrightarrow S^n_a \to C_{2+}\wedge S^{n}.
\] 
Using this cofiber sequence, a standard inductive computation shows that
\[
\A_n=H^{*,*}(S^n_a;\uZ) \cong x^{-1}\M/(\rho^{n+1}) \oplus \Sigma^{n,\epsilon} x^{-1}\M/(\rho), 
\]
where $\epsilon=0$ if $n$ is odd and $\epsilon=1$ if $n$ is even. See Figure \ref{fig:antisphere} for the abbreviated pictures of the modules $\A_0$, $\A_1$, $\A_2$, and $\A_3$. As usual, we omit the $x$-action from the picture since it is an isomorphism in all bidegrees.

To see why the answer should depend on the parity of $n$, observe the quotient $S^n_a/C_2$ is the space $\R P^n$. Hence the groups on the $p$-axis are $H^p_{sing}(\R P^n;\Z)$ by the quotient lemma. Recall the orientability of $\R P^n$ depends on the parity of $n$, which helps explain the role of $\epsilon$ in the above answer. 

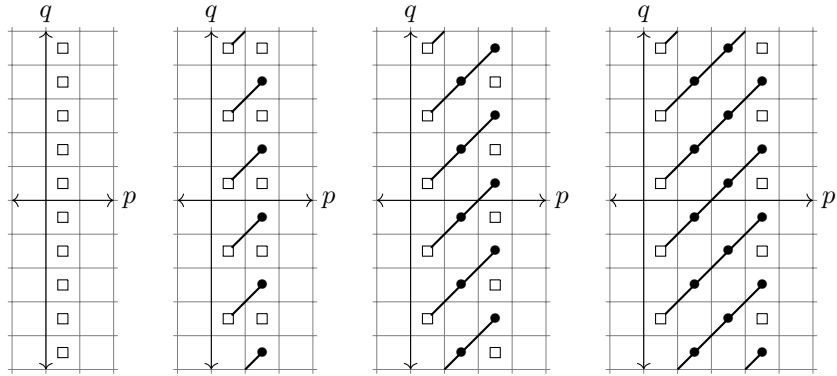
\begin{figure}[ht]
	\begin{tikzpicture}[scale=.45]
	\draw[help lines,gray] (-1.125,-5.125) grid (2.125, 5.125);
	\draw[<->] (-1,0)--(2,0)node[right]{$p$};
	\draw[<->] (0,-5)--(0,5)node[above]{$q$};
	\zline{0}{black};
	\end{tikzpicture}\hspace{0.1in}
	\begin{tikzpicture}[scale=.45]
	\draw[help lines,gray] (-1.125,-5.125) grid (3.125, 5.125);
	\draw[<->] (-1,0)--(3,0)node[right]{$p$};
	\draw[<->] (0,-5)--(0,5)node[above]{$q$};
	\zantizo{0}{black};
	\end{tikzpicture}\hspace{0.1in}
	\begin{tikzpicture}[scale=.45]
	\draw[help lines,gray] (-1.125,-5.125) grid (4.125, 5.125);
	\draw[<->] (-1,0)--(4,0)node[right]{$p$};
	\draw[<->] (0,-5)--(0,5)node[above]{$q$};
	\zanti{0}{0}{black}{2};
	\draw[thick] (0.5,4.5)--(1,5);
	\zbox{0}{4}{black};
	\draw[thick] (2.5,-3.5) node{\small{$\bullet$}}--(1,-5);
	\foreach \y in {3,1,-1,-3,-5}
		\draw (1.5,\y+.5) node{\small{$\bullet$}};
	\zbox{2}{-5}{black};
	\end{tikzpicture}\hspace{0.1in}
	\begin{tikzpicture}[scale=.45]
	\draw[help lines,gray] (-1.125,-5.125) grid (5.125, 5.125);
	\draw[<->] (-1,0)--(5,0)node[right]{$p$};
	\draw[<->] (0,-5)--(0,5)node[above]{$q$};
	\foreach \y in {-4,-2,0}
		{
		\draw[thick] (0.5, \y+.5)--(3.5,\y+3.5) node{\small{$\bullet$}};
		\zbox{0}{\y}{black};
		\draw (1.5,\y+1.5) node{\small{$\bullet$}};
		\draw (2.5,\y+2.5) node{\small{$\bullet$}};
		};
	\draw[thick] (0.5,2.5)--(3,5);
	\zbox{0}{2}{black};
	\draw (1.5,3.5) node{\small{$\bullet$}};
	\draw (2.5,4.5) node{\small{$\bullet$}};
	\draw[thick] (0.5,4.5)--(1,5);
	\zbox{0}{4}{black};
	\foreach \y in {4,2,0,-2,-4}
		\zbox{3}{\y}{black};
	\draw[thick] (1,-5)--(3.5,-2.5) node{\small{$\bullet$}};
	\draw[thick] (3,-5)--(3.5,-4.5) node{\small{$\bullet$}};
	\draw (1.5,-4.5) node{\small{$\bullet$}};
	\draw (2.5,-3.5) node{\small{$\bullet$}};
	\end{tikzpicture}
\caption{The modules $\A_0$, $\A_1$, $\A_2$ and $\A_3$.}
\label{fig:antisphere}
\end{figure}
\end{example}

%%%%%%%%%%%%%%%%%%%%%%%%%%%%%
\section{The family of $\M$-modules $\D_{4n}$}
\label{sec:somemods}
In this section, we introduce a family of $\M$-modules and prove there is a corresponding family of $C_2$-CW complexes whose cohomology is given by these modules. Only one of the spaces will be a $C_2$-surface, but it is no more difficult to compute the cohomology of a general member of this family than it is for the one surface example. This section serves to introduce some interesting modules, to illustrate the computational techniques described in the previous section, and to separate out one of the more complicated $C_2$-surface computations from Section \ref{sec:nonfreecomp}. \smallskip

Define the $\M$-module $\D_{4n}$ for $n>0$ using generators and relations as 
\[\D_{4n}=\M\langle \beta, \alpha_0,\alpha_1,\alpha_2,\dots\rangle/(x\alpha_{i}=\alpha_{i-1}, \rho\alpha_0=2n\beta, \rho\alpha_i=(\sfrac{\theta}{x^{i-1}})n\beta),\]
\[|\beta|=(1,1),\quad \quad |\alpha_i|=(0,-2i).\]
The subscript $4n$ is used because 
\[4n\cdot\beta=2(2n\cdot\beta)=2(\rho\cdot \alpha_0)=0,\]
and $\beta$ generates a cyclic group of order $4n$ in bidegrees $(1,2k+1)$ for $k\geq 0$. We depict the module $\D_{4n}$ in Figure \ref{fig:d4n} with some of the generator labels. The letter $k$ inside a circle is used to denote the group $\Z/k$. We again omit the $x$-connections that are isomorphisms from the picture. The one dashed line shown from bidegree $(1,-1)$ to $(1,1)$ denotes the inclusion $\Z/2n\hookrightarrow \Z/4n$. An abbreviated version of the module is shown on the right.  

\begin{figure}[ht]
\begin{tikzpicture}[scale=.5]
\draw[help lines,gray] (-5.125,-5.125) grid (5.125, 5.125);
\draw[<->] (-5,0)--(5,0)node[right]{$p$};
\draw[<->] (0,-5)--(0,5)node[above]{$q$};
	\draw[thick,dashed] (1+1/2,1+1/2) arc (145:216:1.7);
	\foreach\y in {0,2,4}
	\draw[thick] (0.5,\y+.5)--(5-\y,5);
	\foreach\y in {-2,-4}
	\draw[thick] (0.5,\y+.5)--(1.5,\y+1.5);
	\draw[thick] (1.5,-4.5)--(1,-5);
	\foreach \y in {-4,-2,0,2,4}
	{\zbox{0}{\y}{black};
	};
	\foreach \y in {1,3}
	{\draw[fill=white](1.5,\y+.5) circle (.27cm);
	\draw (1.52,\y+.51) node{\scalebox{.55}{$4n$}};
	};
	\foreach \y in {-1,-3,-5}
	{\draw[fill=white](1.5,\y+.5) circle (.27cm);
	\draw (1.52,\y+.51) node{\scalebox{.55}{$2n$}};
	};
	\draw[thick] (1.5,-1.5)node{\dot}--(-2,-5);
	\draw[thick] (1.5,-3.5)node{\dot}--(0,-5);
	\foreach \y in {2,4}
		\draw (2.5,\y+.5) node{\dot};
	\foreach \y in {3}
		\draw(3.5,\y+.5) node{\dot};
	\foreach \y in {4}
		\draw(4.5,\y+.5) node{\dot};
	\foreach \y in {-5,-3}
		\draw(.5,\y+.5) node{\dot};
	\draw(-.5,-3.5) node{\dot};
	\draw(-1.5,-4.5) node{\dot};
	\draw (1.85,1.15) node{\tiny{$\beta$}};
	\draw (1.9,3.1) node{\tiny{$x\beta$}};
	\draw (2.8,2.2) node{\tiny{$\rho\beta$}};
	\draw (3.85,3.2) node{\tiny{$\rho^2\beta$}};
	\draw (2.8,4.2) node{\tiny{$\rho x\beta$}};
	\draw (2,-.85) node{\tiny{$\theta\beta$}};
	\draw (2.5,-2.8) node{\tiny{$\sfrac{\theta}{x}\cdot \beta$}};
	\draw (2.3,-3.75) node{\tiny{$\sfrac{\mu}{x}\cdot \beta$}};
	\draw (2,-1.7) node{\tiny{$\mu\beta$}};
	\draw (0.5,0.15) node{\tiny{$\alpha_0$}};
	\draw (0.5,2.15) node{\tiny{$x\alpha_0$}};
	\draw (0.6,4) node{\tiny{$x^2\alpha_0$}};
	\draw (.5,-1.85) node{\tiny{$\alpha_1$}};
	\draw (.5,-3.85) node{\tiny{$\alpha_2$}};
\end{tikzpicture}
\begin{tikzpicture}[scale=.5]
\draw[help lines,gray] (-5.125,-5.125) grid (5.125, 5.125);
\draw[<->] (-5,0)--(5,0)node[right]{$p$};
\draw[<->] (0,-5)--(0,5)node[above]{$q$};
	\draw[thick,dashed] (1+1/2,1+1/2) arc (145:216:1.7);
	\foreach\y in {0,2,4}
	\draw[thick] (0.5,\y+.5)--(5-\y,5);
	\foreach\y in {-2,-4}
	\draw[thick] (0.5,\y+.5)--(1.5,\y+1.5);
	\draw[thick] (1.5,-4.5)--(1,-5);
	\foreach \y in {-4,-2,0,2,4}
	{\zbox{0}{\y}{black};
	};
	\foreach \y in {1,3}
	{\draw[fill=white](1.5,\y+.5) circle (.27cm);
	\draw (1.52,\y+.51) node{\scalebox{.55}{$4n$}};
	};
	\foreach \y in {-1,-3,-5}
	{\draw[fill=white](1.5,\y+.5) circle (.27cm);
	\draw (1.52,\y+.51) node{\scalebox{.55}{$2n$}};
	};
	\draw[thick] (1.5,-1.5)node{\dot}--(-2,-5);
	\draw[thick] (1.5,-3.5)node{\dot}--(0,-5);
\end{tikzpicture}
\caption{The $\M$-module $\D_{4n}$.}
\label{fig:d4n}
\end{figure}
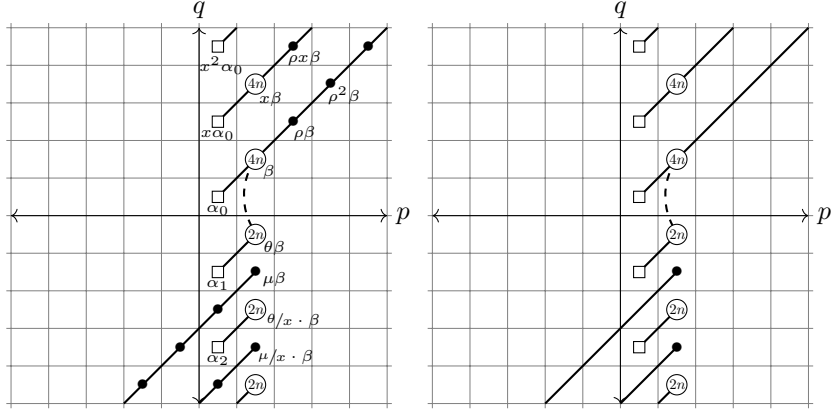

It is not obvious that the module $\D_{4n}$ defined using generators and relations is isomorphic to the illustrated module. But it is straightforward to see that $\D_{4n}$ surjects onto the module drawn, and one can check this gives an isomorphism by just considering all possible elements in each bidegree and checking that nothing nonzero is in the kernel. For example, the elements in bidegree $(1,-3)$ are $(\sfrac{\mu}{\rho})\beta$ and $\mu \alpha_0$. By construction, $\rho \alpha_0 = 2n\beta$ and so
\[
\mu \alpha_0 = (\sfrac{\mu}{\rho}\cdot  \rho) \alpha_0 = \sfrac{\mu}{\rho}\cdot 2n\beta = 0
\]
because $\sfrac{\mu}{\rho}$ is $2$-torsion. Thus $(\sfrac{\mu}{\rho})\beta$ is the only nonzero element in $(1,-3)$, and the map is injective in this bidegree. We leave the rest of the details to the reader.

We now construct a $C_2$-space $X_{2n}$ such that the cohomology of $X_{2n}$ is a shift of $\D_{4n}$. For any $k>1$ we construct a space $X_k$ as follows. Start with the cylinder $S^{1,1}\times D(\R^{1,1})$ where recall $D(\R^{1,1})$ is the closed unit disk in $\R^{1,1}$. Quotient the boundary circle $S^{1,1} \times \{-1\}$ using the identifications given by the degree $k$ map $S^{1}\to S^{1}$. Make the corresponding identifications in the conjugate circle $S^{1,1}\times \{1\}$ according to the $C_2$-action on the cylinder. Denote the resulting space by $X_k$. An illustration of $X_{3}$ is shown in Figure \ref{fig:x3}. The picture on the left is the underlying space whereas the picture on the right shows the $C_2$-action. The fixed set is $S^0$. Note $X_{2}$ is homeomorphic to a Klein bottle. (Using the surgery constructions described in the next section, $X_2$ can also be described as the double connected sum $S^{2,2}\#_2 \R P^2$.)

%%%%%%%The space X_{2n}
\begin{figure}[ht]
\begin{tikzpicture}[scale=0.8]
\draw[
        decoration={markings, mark=at position 0 with {\arrow{>}}, mark=at position 0.151 with {\arrow[scale=.5]{*}}, mark=at position 0.333 with {\arrow{>}}, mark=at position 0.484 with {\arrow[scale=.5]{*}}, mark=at position 0.666 with {\arrow{>}},mark=at position 0.817 with {\arrow[scale=.5]{*}}}, 
        postaction={decorate}
        ] (0,1) ellipse (.4cm and 1cm);
\draw[
        decoration={markings, mark=at position 0.166 with {\arrow{<<}}, mark=at position 0.317 with {\arrow[scale=.5]{*}}, mark=at position 0.5 with {\arrow{<<}}, mark=at position 0.651 with {\arrow[scale=.5]{*}}, mark=at position 0.833 with {\arrow{<<}}, mark=at position 0.984 with {\arrow[scale=.5]{*}}}, 
        postaction={decorate}
        ]
 (4,1) ellipse (.4cm and 1cm);
\draw[thick,fill=white, white] (3.5,0.03) rectangle (4,1.98);
\draw[dashed, 
        decoration={markings, mark=at position 0.166 with {\arrow{<<}}, mark=at position 0.317 with {\arrow[scale=.5]{*}}, mark=at position 0.5 with {\arrow{<<}}, mark=at position 0.651 with {\arrow[scale=.5]{*}}, mark=at position 0.833 with {\arrow{<<}}, mark=at position 0.984 with {\arrow[scale=.5]{*}}},
        postaction={decorate}
        ] (4,1) ellipse (.4cm and 1cm);
\draw(0,0)--(4,0);
\draw(0,2)--(4,2);
\draw[white] (2,3) node {$\bullet$};
\draw[white] (2,-.8) node {$\bullet$};
	\end{tikzpicture}\hspace{.5in}
	\begin{tikzpicture}[scale=0.8]
\draw[
        decoration={markings, mark=at position 0 with {\arrow{>}}, mark=at position 0.151 with {\arrow[scale=.5]{*}}, mark=at position 0.333 with {\arrow{>}}, mark=at position 0.484 with {\arrow[scale=.5]{*}}, mark=at position 0.666 with {\arrow{>}},mark=at position 0.817 with {\arrow[scale=.5]{*}}}, 
        postaction={decorate}
        ] (0,1) ellipse (.4cm and 1cm);
\draw[
        decoration={markings, mark=at position 0.166 with {\arrow{<<}}, mark=at position 0.317 with {\arrow[scale=.5]{*}}, mark=at position 0.5 with {\arrow{<<}}, mark=at position 0.651 with {\arrow[scale=.5]{*}}, mark=at position 0.833 with {\arrow{<<}}, mark=at position 0.984 with {\arrow[scale=.5]{*}}}, 
        postaction={decorate}
        ]
 (4,1) ellipse (.4cm and 1cm);
\draw[thick,fill=white, white] (3.5,0.03) rectangle (4,1.98);
\draw[dashed, 
        decoration={markings, mark=at position 0.166 with {\arrow{<<}}, mark=at position 0.317 with {\arrow[scale=.5]{*}}, mark=at position 0.5 with {\arrow{<<}}, mark=at position 0.651 with {\arrow[scale=.5]{*}}, mark=at position 0.833 with {\arrow{<<}}, mark=at position 0.984 with {\arrow[scale=.5]{*}}},
        postaction={decorate}
        ] (4,1) ellipse (.4cm and 1cm);
\draw(0,0)--(4,0);
\draw(0,2)--(4,2);
\draw[dashed] (2,3)--(2,2);
\draw[dashed] (2,2)--(2,0);
\draw[dashed] (2,0)--(2,-1);
\draw[blue](2,0) node {\dot};
\draw[blue](2,2) node {\dot};
\draw[->] (2.13,2.5) arc (-40:220:.17cm);
\draw (2.6,2.6) node{\tiny{$180^\circ$}};
\end{tikzpicture}
\caption{The space $X_{3}$.}
\label{fig:x3}
\end{figure}
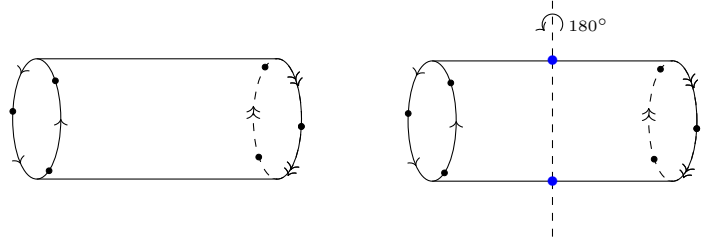

\begin{lemma}\label{x2ncohom} Let $X_{2n}$ be defined as above. As an $\M$-module $\H(X_{2n})\cong \Sigma^{1,1}\D_{4n}$.
\end{lemma}
\begin{proof} We make use of two cofiber sequences. We use the first to see $\H^{2,2k-1}(X_{2n})=0$ and $\H^{2,2k}(X_{2n})\cong \Z/4n$ when $k\geq 1$, and then use the second to finish the computation.

Denote by $Y_{k}$ the quotient of the disk $D^2$ where identifications in $\partial D^2$ have been made using the degree $k$ map, i.e. $Y_k\cong \cof(S^1 \overset{k}{\to}S^1)$. Let $Y_k'$ denote $Y_k$ with a small disk from the interior of $D^2$ removed. To form this first cofiber sequence, consider a closed neighborhood of the boundary circles in $ S^{1,1}\times D(\R^{1,1})$ that is homotopic to $C_2\times Y_{2n}'$ in $X_{2n}$. This neighborhood is shaded on $X_{2}$ in Figure \ref{fig:nbhdx2n}.

\begin{figure}[ht]
\begin{tikzpicture}[scale=0.8]
	\draw[fill, light-gray] (0.7,1) ellipse (.4cm and 1cm);
	\draw[fill, light-gray] (0,0)--(0.7,0)--(.7,2)--(0,2)--(0,0);
	\draw[fill, light-gray] (4,0)--(3,0)--(3,2)--(4,2)--(4,0);
	\draw[fill, white] (0,1) ellipse (.4cm and 1cm);
	\draw[fill, light-gray] (4,1) ellipse (.4cm and 1cm);
	\draw[fill, white] (3,1) ellipse (.4cm and 1cm);
\draw[
        decoration={markings, mark=at position 0 with {\arrow{>}}, mark=at position 0.25 with {\arrow[scale=.5]{*}}, mark=at position 0.5 with {\arrow{>}}, mark=at position 0.75 with {\arrow[scale=.5]{*}}}, 
        postaction={decorate}
        ] (0,1) ellipse (.4cm and 1cm);
\draw[
        decoration={markings, mark=at position 0 with {\arrow{<<}}, mark=at position 0.25 with {\arrow[scale=.5]{*}}, mark=at position 0.5 with {\arrow{<<}}, mark=at position 0.75 with {\arrow[scale=.5]{*}}}, 
        postaction={decorate}
        ]
 (4,1) ellipse (.4cm and 1cm);
\draw[thick,fill=light-gray, light-gray] (3.5,0.03) rectangle (4,1.98);
\draw[dashed, 
          decoration={markings, mark=at position 0 with {\arrow{<<}}, mark=at position 0.25 with {\arrow[scale=.5]{*}}, mark=at position 0.5 with {\arrow{<<}}, mark=at position 0.75 with {\arrow[scale=.5]{*}}}, 
        postaction={decorate}
        ] (4,1) ellipse (.4cm and 1cm);
\draw(0,0)--(4,0);
\draw(0,2)--(4,2);
\draw[dashed] (2,3)--(2,2);
\draw[dashed] (2,2)--(2,0);
\draw[dashed] (2,0)--(2,-1);
\draw[blue](2,0) node {\dot};
\draw[blue](2,2) node {\dot};
\draw[->] (2.13,2.5) arc (-40:220:.17cm);
\draw (2.6,2.6) node{\tiny{$180^\circ$}};
\end{tikzpicture}
\caption{The neighborhood in $X_{2}$.}
\label{fig:nbhdx2n}
\end{figure}
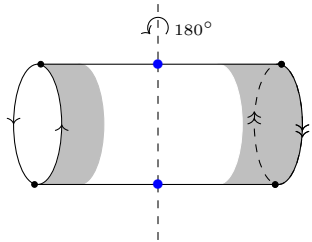
When we quotient by this neighborhood, the resulting $C_2$-space is a rotating sphere with two conjugate points identified. That is, we get the space $\tilde{S}^{2,2}=\cof(C_2\hookrightarrow S^{2,2})$. To summarize, we have the cofiber sequence
\begin{equation}\label{eq:x2ncofib1}
(C_2\times Y_{2n}')_+ \hookrightarrow X_{2n+} \to \tilde{S}^{2,2}.
\end{equation}

From Lemma \ref{pinched}, $\tilde{S}^{2,2}\simeq S^{2,2}\vee S^{1,1}$ and from Lemma \ref{product},
\[
H^{*,*}(C_2\times Y_{2n}')\cong \A_0\otimes H^{*}_{sing}(Y_{2n}') \cong \A_0 \oplus \Sigma^{1,0}\A_0.
\]
 The differential $d:\H^{*,*}((C_2\times Y_{2n}')_+)\to \H^{*+1,*}(\tilde{S}^{2,2})$ is shown in Figure \ref{fig:d4n_cof1}. We are interested in $d^{1,2k}$ for $k\geq 0$ in this sequence. 

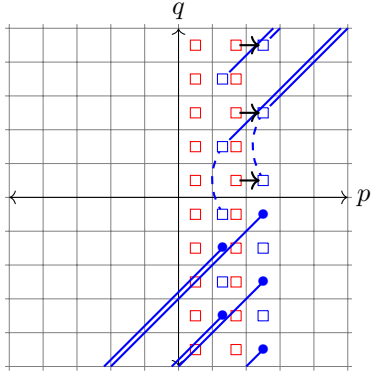
\begin{figure}[ht]
\begin{tikzpicture}[scale=.45]
\draw[help lines,gray] (-5.125,-5.125) grid (5.125, 5.125);
\draw[<->] (-5,0)--(5,0)node[right]{$p$};
\draw[<->] (0,-5)--(0,5)node[above]{$q$};
	\zline{1.2}{red};
	\zline{0}{red};
	\zcone{2}{2}{blue};
	\zcone{.8}{1}{blue};
	\zbox{2}{-4}{blue};
	\draw[blue,thick] (2.5,-4.5)node{\dot}--(2,-5);
	\draw[thick, ->](1.8,.5)--(2.4,.5);
	\draw[thick, ->](1.8,2.5)--(2.4,2.5);
	\draw[thick, ->](1.8,4.5)--(2.4,4.5);
\end{tikzpicture}
\caption{$d:\H^{*,*}((C_2\times Y_{2n}')_+) \to \H^{*+1,*}(\tilde{S}^{2,2})$. }
\label{fig:d4n_cof1}
\end{figure}

Observe $X_{2n}/C_2\cong Y_{2n}$ and so $H^{2,0}(X_{2n})\cong H^{2}_{sing}(Y_{2n}) \cong \Z/2n$ by the quotient lemma. Thus $d^{1,0}$ must be multiplication by $2n$. It follows from the module structure that $d^{1,2k}$ for $k>0$ must be multiplication by $4n$. We conclude $\H^{2,2k}(X_{2n})\cong \Z/4n$ for $k\geq 1$. Also note $\H^{2,2k-1}(X_{2n})=0$.

We use a different cofiber sequence to finish the computation. The equator of the cylinder is isomorphic to $S^{1,1}$. Including the equator gives the cofiber sequence
\begin{equation}\label{eq:x2ncofib2}
S^{1,1}\hookrightarrow X_{2n} \to C_{2+}\wedge Y_{2n}.
\end{equation}
By Lemma \ref{product}, $\H^{*,*}(C_{2+}\wedge Y_{2n})\cong \A_0 \otimes \H^{*}_{sing}(Y_{2n})$.

The associated differential is shown in Figure \ref{fig:d4n_cof2}. We only need to determine $d^{1,1}$ since $\H^{*,*}(S^{1,1})$ is a free module. From the previous cofiber sequence, we know $\H^{2,1}(X_{2n})=0$ and so $d^{1,1}$ must be a surjection. We now solve the extension problem shown on the right in Figure \ref{fig:d4n_cof2}.
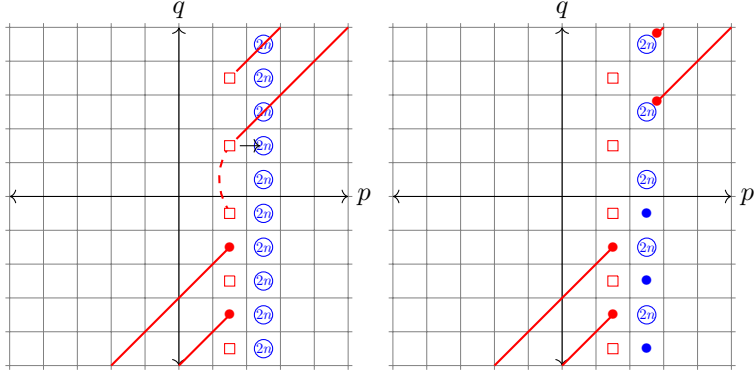
\begin{figure}[ht]
\begin{tikzpicture}[scale=.45]
\draw[help lines,gray] (-5.125,-5.125) grid (5.125, 5.125);
\draw[<->] (-5,0)--(5,0)node[right]{$p$};
\draw[<->] (0,-5)--(0,5)node[above]{$q$};
	\foreach \y in {-5,-4,-3,-2,-1,0,1,2,3,4}
	{\draw[blue,fill=white](2.5,\y+.5) circle (.27cm);
	\draw[blue] (2.52,\y+.51) node{\scalebox{.55}{$2n$}};};
	\zcone{1}{1}{red};
	\zbox{1}{-5}{red};
	\draw[->](1.8,1.5)--(2.4,1.5);
\end{tikzpicture}
\begin{tikzpicture}[scale=.45]
\draw[help lines,gray] (-5.125,-5.125) grid (5.125, 5.125);
\draw[<->] (-5,0)--(5,0)node[right]{$p$};
\draw[<->] (0,-5)--(0,5)node[above]{$q$};
	\foreach \y in {-4,-2,0,2,4}
	{\draw[blue,fill=white](2.5,\y+.5) circle (.27cm);
	\draw[blue] (2.52,\y+.51) node{\scalebox{.55}{$2n$}};};
	\foreach \y in{-1,-3,-5}
	\draw[blue] (2.5,\y+.5)node{\dot};
	\foreach \y in {-5,-3,-1,1,3}
	\zbox{1}{\y}{red};
	\draw[red,thick] (1.5,-1.5)node{\dot}--(-2,-5);
	\draw[red,thick] (1.5,-3.5)node{\dot}--(0,-5);
	\draw[red,thick] (2.8,2.8)node{\dot}--(5,5);
	\draw[red,thick] (2.8,4.8)node{\dot}--(3,5);
\end{tikzpicture}
\caption{The differential $d:\H^{*,*}(S^{1,1})\to \H^{*+1,*}(C_{2+} \wedge Y_{2n})$.}
\label{fig:d4n_cof2}
\end{figure}

From the previous cofiber sequence, we have that $\H^{2,2k+1}(X_{2n})\cong \Z/4n$. This solves the extension problem when regarded as bigraded abelian groups. The action of $x^j$ and $\sfrac{\theta}{x^j}$ are then determined from the actions in the cokernel and kernel. We next consider possible $\rho$-extensions.

To determine how $\rho$ acts when $p=1$, consider the following portion of the forgetful long exact sequence:
\[ H^{1,q}(X) \overset{\rho}{\to} H^{2,q+1}(X) \overset{\psi}{\to} H^{2}_{sing}(X) \to H^{2,q}(X) \to H^{3,q+1}(X).\]
Observe $H^{2}_{sing}(X)\cong \Z/2n$. When $q=2k+1$ for $k\geq 0$, this sequence is
\[
\Z \overset{\rho}{\to} \Z/4n \overset{\psi}{\to}  \Z/2n \to 0 \to 0.
\]
Thus $\rho$ must act nontrivially. When $q=2k+1$ for $k\leq -1$, this gives us
\[
\Z \overset{\rho}{\to} \Z/2n \overset{\psi}{\to}  \Z/2n \to \Z/2 \to 0.
\]
Thus $\rho$ must act nontrivially. Lastly, when $q=2k$ for $k\leq -1$, we have
\[
\Z/2 \overset{\rho}{\to} \Z/2 \overset{\psi}{\to}  \Z/2n \to \Z/2n \to 0,
\]
which again implies $\rho$ acts nontrivially. 

We now know how $\rho$, $x$, and $\sfrac{\theta}{x^j}$ act. It remains to consider $\sfrac{\mu}{(\rho^ix^j)}$. We explain how to do this for $\mu$, noting the argument for the other classes is similar. Recall $\mu x = 0$, so $\mu$ acts trivially on anything in the image of $x$. All elements are in the image of $x$ except the generator of $H^{2,2}(X)$. Call this generator $\beta$. It follows from relations in the kernel that $(\sfrac{\mu}{\rho})\cdot \beta$ is the generator in bidegree $(1, -2)$. Thus 
\[\mu\cdot \beta = \rho \cdot (\sfrac{\mu}{\rho} \cdot \beta),\]
which is the nonzero class in $(2,-1)$ from our analysis of the $\rho$-action. We can do a similar analysis to determine $\sfrac{\mu}{(\rho^ix^j)}$ acts nontrivially on $\beta$. We conclude the extension problem is solved by the module $\Sigma^{1,1}\D_4$.
\end{proof}

\begin{rmk} We can also compute the cohomology of $X_{2n+1}$. It is straightforward to determine $\H^{*,*}(X_{2n+1})\cong \Sigma^{1,1}\M \oplus \Sigma^{2,0}\M/(2n+1)$.
\end{rmk}

%%%%%%%%%%%%%%%%%%%%%%%
\section{A top cohomology class for general $C_2$-manifolds}
\label{sec:sphereimage}
We now consider the cohomology of $C_2$-manifolds of any dimension and prove a fact that will be useful in the surface computations. Non-equivariantly, recall if $M$ is an $n$-manifold and $x\in M$ is a point contained in an open $n$-disk $D\subset M$, then we get a quotient map $M\to M/(M-D)\cong S^n$. If $M$ is orientable, then this induces an isomorphism 
\[\Z\cong \H^n_{sing}(S^n)\to \H^n_{sing}(M).\] If $M$ is nonorientable, then $H^{n}_{sing}(M)\cong \Z/2$ and the induced map is surjective. 

Let's try to play the same game for closed manifolds with a $C_2$-action. If $M$ is a nonfree $C_2$-manifold, then its fixed set $M^{C_2}$ is a disjoint union of finitely many submanifolds, possibly of varying dimensions. If we pick a fixed point in a component of $M^{C_2}$ of codimension $k$, then there is an $(n,k)$-disk containing $x$, and we get a quotient map $M\to M/(M-D(\R^{n,k}))\cong S^{n,k}$. On cohomology this gives a map
\[
\Sigma^{n,k}\M\cong \H^{*,*}(S^{n,k})\to \H^{*,*}(M),
\]
and we'd like to understand the image. We begin with the orientable case.

If $M$ is orientable, then we might expect the map to be injective and thus give a free submodule in topological degree $n$. This is similar to how, non-equivariantly, the map $H^*_{sing}(S^n)\to H^*_{sing}(M)$ gives a free abelian group in dimension $n$. If we are careful about how we pick $k$ when the dimension of $M^{C_2}$ is not constant, then the desired statement is true for any orientable $C_2$-manifold $M$. That is, we prove there is a predictable copy of $\M$ generated in topological degree $n=\dim(M)$ in the cohomology of $M$, and the weight is given by the \emph{minimum} codimension of the components of the fixed set. An analogous statement was proven with $\underline{\Z/2}$ coefficients in \cite[Theorem A.1]{Haz20}. We give an example before stating the theorem.

\begin{example}
Consider the $3$-dimensional $C_2$-manifold $P(\R^{4,1})$, the projective space given by taking lines in $\R^{4,1}$. The underlying space is the orientable manifold $\R P^3$. We can denote points in $P(\R^{4,1})$ by $[x_1\colon x_2\colon x_3\colon x_4]$ where not all $x_i$ are zero and $[x_1\colon x_2\colon x_3\colon x_4]= [y_1\colon y_2 \colon y_3 \colon y_4]$ if there exists $\lambda \in \R$ such that   $(x_1,x_2,x_3,x_4) =\lambda (y_1,y_2,y_3,y_4)$. We can assume the action on $\R^{4,1}$ is given by $(x_1,x_2,x_3,x_4) \mapsto (-x_1,x_2,x_3,x_4)$. Then the fixed set of $P(\R^{4,1})$ has two components given by
\[
\{[x_1\colon0\colon0\colon0] \}\cong pt, \quad \quad \{[0\colon x_2 \colon x_3 \colon x_4]\} \cong \R P^2.
\]
The codimensions of these two components are $3$ and $1$, respectively. Thus the minimum codimension is $1$ and so $\Sigma^{3,1}\M$ should be a submodule of $\H^{*,*}(P(\R^{4,1}))$. Indeed, one can compute $\H^{*,*}(P(\R^{4,1})) \cong \Sigma^{2,1}\M_2 \oplus \Sigma^{3,1}\M$.
\end{example}

\begin{theorem}\label{topm}
Let $M$ be a nonfree, $n$-dimensional $C_2$-manifold such that the underlying manifold is orientable. Suppose $k$ is the minimum codimension of any component of $M^{C_2}$. Let $x\in M^{C_2}$ be a fixed point in a component of codimension $k$ and $D$ be a neighborhood of $x$ such that $D\cong D(\R^{n,k})$. Then the quotient map $q:M\to M/(M-D)\cong S^{n,k}$ induces an injection $q^*:\H^{*,*}(S^{n,k})\hookrightarrow \H^{*,*}(M)$.  
\end{theorem}

\begin{proof}
Let $\alpha$ be a generator of the free module $\H^{*,*}(S^{n,k})$. We will use the following commutative diagrams where $\pi$ denotes the map induced by the quotient map $\uZ\to \underline{\Z/2}$ and $\psi$ denotes the forgetful map:
\begin{equation}\label{splitting_diagram_1}
\begin{tikzcd}
\H^{r,s}(S^{n,k};\uZ) \arrow[r,"q^*"]\arrow[d,"\pi"]& \H^{r,s}(M;\uZ)\arrow[d,"\pi"]\\
\H^{r,s}(S^{n,k};\underline{\Z/2})\arrow[r,hook,"q^*"] & \H^{r,s}(M;\underline{\Z/2})
\end{tikzcd}
\end{equation}
\begin{equation}\label{splitting_diagram_2}
\begin{tikzcd}
\H^{r,s}(S^{n,k};\uZ) \arrow[r,"q^*"]\arrow[d,"\psi"]& \H^{r,s}(M;\uZ)\arrow[d,"\psi"]\\
\H^{r}_{sing}(S^{n,k};\Z)\arrow[r,"q^*"] & \H^{r}_{sing}(M;\Z)
\end{tikzcd}
\end{equation}
The map $\pi$ is described in Remark \ref{rempointz2}. The bottom horizontal map in \ref{splitting_diagram_1} is a split injection by \cite[Theorem A.1]{Haz20}. 

All elements in $\H^{r,s}(S^{n,k};\uZ)$ are of the form $m\alpha$ for $m\in \M$. We break into three cases based on the value of $m$.

\emph{Case 1: $m=\rho^ix^j$ or $m=\sfrac{\mu}{(\rho^{i-1}x^j)}$ for $i\geq 1$, $j\geq 0$.} In these bidegrees, $\H^{r,s}(S^{n,k};\uZ)\cong \Z/2$ and the left vertical map $\pi$ in \ref{splitting_diagram_1} is an isomorphism. Thus the injectivity of the lower map $q^*$ implies the injectivity of the top map $q^*$.

\emph{Case 2: $m=x^j$ for $j\geq 0$.} In these bidegrees $\H^{n,s}(S^{n,k})\cong \Z$. Consider the square in \ref{splitting_diagram_2}. The bottom horizontal map is an isomorphism because $r=n$ and $M$ is orientable, and the left vertical map is an isomorphism because $\psi(x^j)=1$ and $\psi(\alpha)$ is the generator of $\H^{n}_{sing}(S^{n,k};\Z)$. Hence the top map $q^*$ is injective.

\emph{Case 3: $m=\sfrac{\theta}{x^j}$ for $j\geq 0$.} In these bidegrees $\H^{n,s}(S^{n,k})\cong \Z$. This will be similar to Case 2, except now $\psi(\sfrac{\theta}{x^j})=2$ and so the left map isn't an isomorphism, but it is still injective. Thus $q^*\psi$ is injective, which implies the top map $q^*$ is injective. 

We have shown $q^*$ is injective in all bidegrees where the domain is nonzero. We conclude $q^*$ is injective. 
\end{proof}

\begin{rmk}\label{sphereimagermk}
Suppose the dimension of $M^{C_2}$ is not constant, and we have a component of codimension $j$ where $j>k$. We can also determine the image of the map induced by $q_j:M\to S^{n,j}$. Since $M$ is orientable, the action on $M$ is either orientation preserving or orientation reversing. If it is orientation preserving, then $k$ and $j$ must be even integers. If it is orientation reversing, then $k$ and $j$ must be odd integers. In either case, $k$ and $j$ have the same parity. Using similar methods to case when $j=k$, one can show $\im(q_j^*)$ is the submodule generated by $x^{j-k}\cdot q^*(\alpha)$ where $q$ is the map from the theorem above and $\alpha$ is the generator of $\H^{*,*}(S^{n,k})$. Note the map will not be injective because $\mu x^{j-k}=0$ if $j>k$.
\end{rmk}

We now consider the abelian group generated by $q^*(\alpha)$ when the underlying space is nonorientable. One could also investigate the submodule $\im(q^*)$, but the answer is not yet well understood. We only need the following fact in the surface computations.

\begin{prop}
\label{nonortopm}
 Let $M$ be a nonfree $C_2$-manifold such that the underlying manifold is nonorientable. Suppose $k$ is the minimum codimension of any connected component of $M^{C_2}$. Let $q:M\to S^{n,k}$ be defined as above and let $\alpha \in \H^{n,k}(S^{n,k})$ be the generator. Then the abelian group generated by $q^*\alpha$  is either $\Z/4$ or $\Z/2$.
\end{prop} 

\begin{proof}
Since $M$ is nonorientable, $\H^{n}_{sing}(M)\cong \Z/2$ and the induced map $q^*$ on singular cohomology is surjective. Consider the map of forgetful long exact sequences:
\begin{center}
\begin{tikzcd}
	\H^{n-1,k-1}(S^{n,k}) \arrow[r,"\rho"]\arrow[d,"q^*"]&\H^{n,k}(S^{n,k})\arrow[d,"q^*"]\arrow[r,"\psi","\cong" below]& \H^{n}_{sing}(S^{n,k}) \arrow[d,"q^*",twoheadrightarrow]\\
	\H^{n-1,k-1}(M) \arrow[r,"\rho"]&\H^{n,k}(M)\arrow[r,"\psi"]& \H^{n}_{sing}(M) 
\end{tikzcd}
\end{center}
The right-hand square implies $q^*\alpha$ is nonzero, and since $\H^n_{sing}(M)=\Z/2$, it must be that $2q^*\alpha \in \ker(\psi)=\im(\rho)$. Everything in the image of $\rho$ is $2$-torsion, so $2q^*\alpha$ must be $2$-torsion. Thus $\langle q^*\alpha \rangle \cong \Z/4$ or $\langle q^*\alpha \rangle \cong \Z/2$. 
\end{proof}

\begin{rmk}
It will follow from Theorem \ref{nonfreenonoranswer} that both of these possibilities can be realized. The two possibilities hint at different levels of nonorientability. In the $C_2$-surface case, the image is $\Z/4$ when the components of the fixed set have constant dimension and orientable normal bundles. The image is $\Z/2$ when the fixed set either has components of different dimensions, or when at least one of the normal bundles is nonorientable. This suggests the two cases depend on orientability conditions on the fixed set. Exactly how the two cases arise in higher dimensions is unclear. 
\end{rmk}

We record one last fact about $q^*\alpha$ and $\H^{*,*}(M)$ that will be helpful later on. 

\begin{lemma}\label{miscmani}
 Let $M$ be a nonfree $n$-dimensional $C_2$-manifold and let $q$, $k$, and $\alpha$ be as above. Then $q^*\alpha$ is not in the image of $\rho$ and $\H^{n,k-1}(M)=0$.
\end{lemma}
\begin{proof}
We showed $\psi(q^*\alpha)$ is a generator of $\H^{n}_{sing}(M)$ in both the orientable case and nonorientable case in the previous proofs, so it follows from the forgetful long exact sequence that $q^*\alpha$ is not in the image of $\rho$. 

Now assume to the contrary there is a nonzero class $y\in \H^{n,k-1}(M)$. The forgetful map $\psi:\H^{n,k}(M)\to \H^n_{sing}(M)$ is surjective because $q^*\alpha$ maps to the generator, so the forgetful long exact sequence
\[\H^{n,k}(M)\overset{\psi}{\twoheadrightarrow} \H^n_{sing}(M)\to \H^{n,k-1}(M) \overset{\rho}{\hookrightarrow} \H^{n+1,k}(M) \]
 implies $\rho y\neq 0$. But then $\rho^j y\neq 0 $ for all $j\geq 1$ because $\H^{n+j}_{sing}(M)=0$. Thus $y$ has no $\rho$-torsion and will give something nonzero in $\rho^{-1}\H^{*,*}(M)$ in bidegree $(n,k-1)$. From the $\rho$-localization given in Lemma \ref{rho_local} this would imply that $\H^{i}_{sing}(M^{C_2})$ is nonzero for some $i\geq n-(k-1)$. Recall $n-k$ is the maximum dimension of $M^{C_2}$, so this is a contradiction.
\end{proof}

%%%%%%%%%%%%%%%%%%%%%%%%%%%%%%%%%%%%%
\section{Background on equivariant surgery and $C_2$-surfaces}\label{sec:backgroundsurg}
We focus exclusively on $C_2$-surfaces for the remainder of the paper. In Dugger's classification of $C_2$-surfaces \cite{Dug19}, various types of equivariant surgery were introduced. We review necessary constructions and pieces of the classification here.

Let $X$ be a nontrivial $C_2$-surface and let $Y$ be a nonequivariant surface. We can form the {\bf{double connected sum}} of $X$ with $Y$ as follows. Pick a non-fixed point $x\in X$ and let $D$ be an open disk containing $x$ such that $D$ does not intersect its conjugate disk $\sigma D$. Let $X''=X-(D \sqcup \sigma D)$. Pick an open disk in $Y$ and let $Y'$ denote $Y$ with the open disk removed. Let $f$ be a homeomorphism from $\partial D$ to $\partial Y'$. Define
\[
X\#_2 Y := [X'' \sqcup (C_2\times Y')]/\sim
\]
where $x\sim (0,f(x))$ and $\sigma x \sim (1,f(x))$ for all $x\in \partial D$. Note the underlying space is just $Y \# X \# Y$. 

Let $T_1^{rot}$ denote the free $C_2$-torus whose underlying space is the genus one torus and whose action is given by rotating 180$^\circ$ around an axis that goes through the torus hole. Let $T_1^{anti}$ denote the free $C_2$-torus whose underlying space is the genus one torus and whose action is given by embedding the torus in $\R^3$ centered at the origin, and then restricting the antipodal action on $\R^3$. The theorem below follows from the classification in \cite[Theorem 4.1]{Dug19}.

\begin{theorem}\label{freeclass}
Let $X$ be a free $C_2$-surface. Then there exists a nonequivariant surface $Y$ such that $X\cong T_1^{anti}\#_2 Y$, $X\cong T_1^{rot}\#_2 Y$, or $X\cong S^2_a \#_2 Y$.
\end{theorem}

\begin{rmk}
If $Y$ is the nonorientable surface $N_r=(\R P^2)^{\# r}$, then $T_{1}^{rot}\#_2 Y \cong T_{1}^{anti}\#_2 Y$ \cite[Proposition 4.25]{Dug19}. This is not true if $Y=T_g$. For computational reasons, it is easier to keep the classification split into the three cases above even if one is redundant in the nonorientable case. 
\end{rmk}

We now consider nonfree surfaces and describe two other forms of equivariant surgery. The first is an equivariant version of attaching a handle. 

Let $X$ be a nontrivial $C_2$-surface, and as before let $X''=X-(D\sqcup \sigma D)$ where $D$ is a disk that does not intersect its conjugate. First consider the cylinder $S^{1,0}\times D(\R^{1,1})$ where recall $D(\R^{1,1})$ is the unit disk in the sign representation $\R^{1,1}$. The cylinder looks like a tube with a reflection action where the fixed set is the equatorial circle. We can attach this cylinder to $X''$ by gluing one end of the tube to $\partial D$, and then gluing the other end to $\partial(\sigma D)$, where the gluing to the conjugate boundary is determined using the $C_2$-action. We call this {\bf{$\mathbf{S^{1,0}}$-surgery}} and denote the resulting space by $X+[S^{1,0}-AT]$. We can similarly preform {\bf{$\mathbf{S^{1,1}}$-surgery}} by instead attaching a cylinder of the form $S^{1,1}\times D(\R^{1,1})$ (a rotating tube with two fixed points) to form a new space $X+[S^{1,1}-AT]$. In both cases we refer to the handle as an {\bf{antitube}}.

We can now describe all nonfree $C_2$-surfaces up to equivariant homeomorphism when the underlying space is orientable, and when the underlying space is nonorientable and the fixed set only consists of isolated fixed points; these follow from \cite[Theorems 5.7 and 7.7]{Dug19}.

\begin{theorem}\label{nonfreeclasstorus} Let $X$ be a nonfree, nontrivial $C_2$-surface whose underlying space is orientable. If $X^{C_2}$ consists only of isolated fixed points, then there exists $g \geq 0$, $n\geq0$ such that \[X \cong \left(S^{2,2} \#_2 T_g \right) + n[S^{1,1}-AT].\] If $X^{C_2}$ consists only of fixed circles, then there exists $g\geq 0$, $n\geq 0$ such that $X$ is isomorphic to 
\[\left(S^2_a \#_2 T_g\right) + n[S^{1,0}-AT], ~\left(T_1^{anti} \#_2 T_g\right) + n[S^{1,0}-AT], ~\text{or}~  \left( S^{2,1}\#_2 T_g\right) + n[S^{1,0}-AT].\]
\end{theorem}

\begin{theorem}\label{nonfreeclassnonorf} Let $X$ be a nonfree, nontrivial $C_2$-surface whose underlying space is nonorientable. Suppose $X^{C_2}$ contains only isolated fixed points. Then \[X \cong (S^{2,2} \#_2 N_r) + m[S^{1,1}-AT].\]
\end{theorem}

We need one more type of surgery to do computations in the case where the surface is nonorientable. This surgery method replaces fixed points with fixed circles by attaching M\"obius bands and can be described as follows.

Suppose $X$ is a $C_2$-surface with an isolated fixed point $x\in X^{C_2}$. Let $D$ be an equivariant disk containing the fixed point. Then $D\cong \R^{2,2}$ because the fixed point is isolated, and thus $\partial D \cong S^1_a$. Let $X'=X-D$. Let $M$ denote the closed M\"obius band whose $C_2$-action is shown below. This is the disk bundle of the nontrivial bundle over $S^{1,0}$. The fibers over each point in the circle are being reflected. Observe $\partial M \cong S^1_a$ so we can attach $M$ to $X'$ by equivariantly identifying their boundary circles. We denote the new space by $X+[FM]$. The ``FM'' stands for ``fixed point to M\"obius band''. Note the underlying space $X+[FM]$ is homeomorphic to $X\# \R P^2$.
\begin{center}
\includegraphics[scale=0.6]{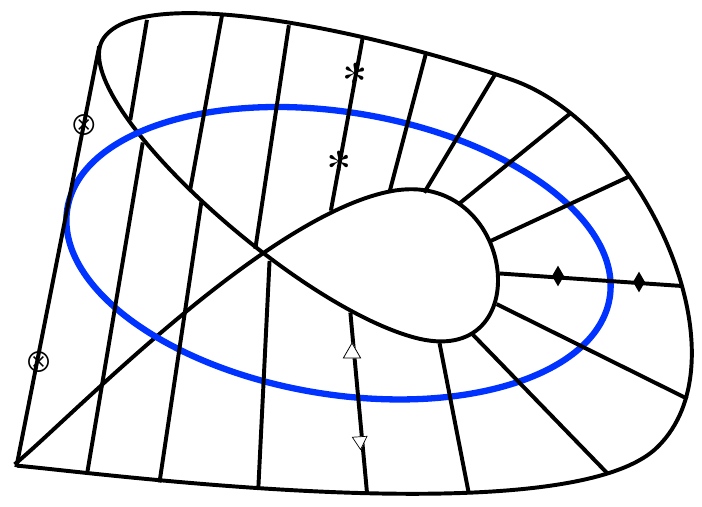}
\end{center}

We won't need any remaining other parts of the classification to complete the computations. We instead make use of various surgery operations in the case that the fixed set contains a fixed circle. Recall there are two isomorphism classes of one-dimensional bundles over $S^1$, the trivial bundle and the M\"obius bundle. When the normal bundle of a fixed circle is trivial, we say that circle is {\bf{two-sided}}. When the normal bundle of a fixed circle is nontrivial, we say that circle is {\bf{one-sided}}. 

We will occasionally need to do the above surgery operations backwards. For example, suppose $X$ is a $C_2$-surface with a two-sided fixed circle $S$ such that $X-S$ is connected. Observe we can remove an equivariant tubular neighborhood around $S$ (this will be an $S^{1,0}$-antitube), and the resulting space will have boundary $C_2\times S^1$. We can then attach $C_2\times D^2$ to get a new surface $Y$. It was proven in \cite{Dug19} that we can recover $X$ from $Y$ via
\[X\cong Y+[S^{1,0}-AT].\] We can do something similar when the circle is one-sided and get that 
\[X\cong Y+[FM]\] for some $C_2$-surface $Y$. \smallskip

We will need to keep track of how surgery operations affect the fixed set and the genus. Recall the genus of a surface can be recovered from the dimension of its first singular cohomology group in $\Z/2$-coefficients. For a torus the usual genus is half of the dimension of its first singular cohomology group, while for a nonorientable surface $N_r=(\R P ^2)^{\# r}$, the dimension is exactly $r$, which is the usual notion of ``genus'' for a nonorientable surface. For a $C_2$-surface $X$ write
\begin{align*}
&F(X) = \# \text{ isolated fixed points} , \quad\quad\quad ~~ C(X) = \# \text{ fixed circles},\\
&C_+(X)  = \# \text{ two-sided fixed circles}, \quad C_-(X)  = \# \text{ one-sided fixed circles}, \\
&\beta(X)  = \dim_{\Z/2}(H^1_{sing}(X;\Z/2)).
\end{align*}

\begin{lemma}\label{surgeffects}
Let $X, Y$ be $C_2$-surfaces. 
\begin{enumerate}
\item[(i)] Suppose $X\cong Y+[S^{1,0}-AT]$. Then $\beta(X)=\beta(Y)+2$, $F(X)=F(Y)$, $C_+(X)=C_+(Y)+1$, and $C_-(X) = C_-(Y)$.
\item[(ii)] Suppose $X\cong Y+[S^{1,1}-AT]$. Then $\beta(X)=\beta(Y)+2$, $F(X)=F(Y)+2$, $C_+(X)=C_+(Y)$, and $C_-(X) = C_-(Y)$.
\item[(iii)] Suppose $X\cong Y+[FM]$. Then $\beta(X)=\beta(Y)+1$, $F(X)=F(Y)-1$, $C_+(X)=C_+(Y)$, and $C_-(X) = C_-(Y)+1$.
\end{enumerate}
\end{lemma}

%%%%%%%%%%%%%%%%%%%%%%%%%%%%%%%%%
\section{Computations for free $C_2$-surfaces}
\label{sec:freecomp}
We now compute the cohomology of all $C_2$-surfaces with a free $C_2$-action. We start by computing the cohomology of the free tori $T_1^{anti}$ and $T_1^{rot}$. We then consider equivariant connected sums, and finally we make use of the classification in Theorem \ref{freeclass} to finish the computation.

\begin{lemma}\label{freetorusanswerz} For the two free genus one tori, we have that
\[H^{*,*}(T_1^{rot}) \cong \Sigma^{1,0}\A_1 \oplus \Sigma^{1,0}\A_1 \quad \quad \text{and} \quad \quad H^{*,*}(T_1^{anti}) \cong \Sigma^{1,0}\A_1 \oplus \Sigma^{1,1}\A_1.\]
\end{lemma}

\begin{proof}
Consider first $T=T_1^{rot}$. This $C_2$-torus can be realized as the product $S^{1,0}\times S^1_a$ and so we have a cofiber sequence
\[S^{1}_{a+}\hookrightarrow T_+ \to S^{1,0}\wedge S^{1}_{a+}.\]
The associated long exact sequence is shown in Figure \ref{fig:zfreetoricomp}. Observe all differentials can be determined using $d^{0,0}$, $d^{1,0}$, and the module structure.
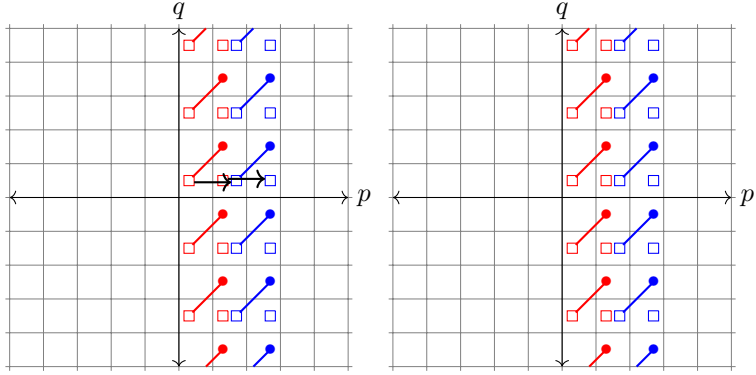
\begin{figure}[ht]
\begin{tikzpicture}[scale=.45]
\draw[help lines,gray] (-5.125,-5.125) grid (5.125, 5.125);
\draw[<->] (-5,0)--(5,0)node[right]{$p$};
\draw[<->] (0,-5)--(0,5)node[above]{$q$};
\zantizo{-.2}{red};
\zantizo{1.2}{blue};
\draw[thick,->] (.45,.45)--(1.55,.45);
\draw[thick,->] (1.45,.55)--(2.55,.55);
\end{tikzpicture}
\begin{tikzpicture}[scale=.45]
\draw[help lines,gray] (-5.125,-5.125) grid (5.125, 5.125);
\draw[<->] (-5,0)--(5,0)node[right]{$p$};
\draw[<->] (0,-5)--(0,5)node[above]{$q$};
\zantizo{-.2}{red};
\zantizo{1.2}{blue};
\end{tikzpicture}
\caption{The differential $\tilde{H}^{*,*}(S^1_{a+})\to \tilde{H}^{*+1,*}(S^{1,0}\wedge S^1_{a+})$.}
\label{fig:zfreetoricomp}
\end{figure}

The orbit space $T/C_2$ is again a torus, so by the quotient lemma 
\[H^{0,0}(T) \cong H^{0}_{sing}(T/C_2)\cong \Z\quad \text{ and } \quad H^{1,0}(T)\cong H^{1}_{sing}(T/C_2) \cong \Z \oplus \Z.\] Thus the two differentials must be zero, which implies all differentials are zero. 

It remains to solve the extension problem
\[0\to \H^{*,*}(S^{1,0}\wedge S^1_{a+})\to H^{*,*}(T_1^{rot}) \to H^{*,*}(S^1_a) \to 0.\]
Since $T^{rot}_1=S^{1,0}\times S^{1}_a$, we can consider the composition \[\{0\}\times S^1_a\hookrightarrow S^{1,0} \times S^1_a \to S^1_a\] that is equal to the identity. This gives a splitting of the short exact sequence above, so $H^{*,*}(T^{rot}_1)\cong \A_1 \oplus \Sigma^{1,0}\A_1$. The computation for the antipodal torus is done similarly by noting $T^{anti}_1=S^{1,1}\times S^1_a$ and $T^{anti}_1/C_2$ is a Klein bottle.
\end{proof}

We now consider connected sums. Let $Y$ be a surface and $T$ be one of the two free tori. In order to compute the cohomology of $T\#_2 Y$, we will use a cofiber sequence that crushes the two copies of $Y$ to a single point. The resulting space is the cofiber of $C_2\hookrightarrow T$. We compute the cohomology of these ``pinched'' tori below and then consider $T\#_2 Y$.

\begin{lemma}\label{pinchedtor} 
For a nontrivial $C_2$-space $X$, denote by $\tilde{X}$ the cofiber of $C_2\hookrightarrow X$. For the two free tori,
		\[
		\tilde{H}^{*,*}(\tilde{T}_{1}^{rot}) \cong \left(\Sigma^{1,0}\A_0\right)\oplus \left(\Sigma^{1,0}\A_1\right)  \quad \text{ and} \quad  \tilde{H}^{*,*}(\tilde{T}_{1}^{anti}) \cong \left(\Sigma^{1,0}\A_0\right)\oplus \left(\Sigma^{1,1}\A_1\right).
		\]
\end{lemma}
\begin{proof}
We provide the proof for $T=T_1^{anti}$ noting the computation for $T_1^{rot}$ follows similarly. Extend the defining cofiber sequence to
	\[
	T_+\to \tilde{T}\to\Sigma^{1,0}C_{2+}.
	\]
The differential $d: \tilde{H}^{*,*}(T_+) \to\tilde{H}^{*,*+1}(\Sigma^{1,0}C_{2+})$ is illustrated on the left in Figure \ref{fig:ztortilde} below.
\begin{figure}[ht]
	\begin{tikzpicture}[scale=.45]
	\draw[help lines,gray] (-1.125,-5.125) grid (5.125, 5.125);
	\draw[<->] (-1,0)--(5,0)node[right]{$p$};
	\draw[<->] (0,-5)--(0,5)node[above]{$q$};
	\zantizo{-.3}{red};
	\zantioo{1.3}{red};
	\zline{1}{blue};
	\draw[thick,->] (.2,.5)--(1.5,.5);
	\end{tikzpicture}
	\begin{tikzpicture}[scale=.45]
	\draw[help lines,gray] (-1.125,-5.125) grid (5.125, 5.125);
	\draw[<->] (-1,0)--(5,0)node[right]{$p$};
	\draw[<->] (0,-5)--(0,5)node[above]{$q$};
	\zantioo{1.3}{red};
	\foreach \y in {-4,-2,0,2,4}
		\zbox{.8}{\y}{red};
	\foreach \y in {-5,-3,-1,1,3}
		\draw[red] (1.2,\y+.5) node{\small{$\bullet$}};
	\foreach \y in {-5,-3,-1,1,3}
		\zbox{1}{\y}{blue};
	\end{tikzpicture}
	\begin{tikzpicture}[scale=.45]
	\draw[help lines,gray] (-1.125,-5.125) grid (5.125, 5.125);
	\draw[<->] (-1,0)--(5,0)node[right]{$p$};
	\draw[<->] (0,-5)--(0,5)node[above]{$q$};
	\zantioo{1.3}{black};
	\zline{0.8}{black};
	\end{tikzpicture}
\caption{The differential $\tilde{H}^{*,*}(T_+)\to \tilde{H}^{*+1,*}(\Sigma^{1,0}C_{2+})$.}
\label{fig:ztortilde}
\end{figure}
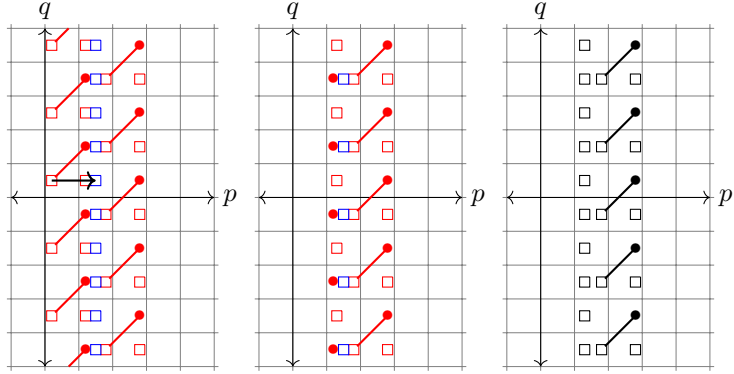

Many of the differentials are forced to be zero, and all others can be determined from $d^{0,0}$ and the module structure. Observe $\tilde{T}/C_2\cong T/C_2\cong \R P^2\# \R P^2$ so by the quotient lemma $\tilde{H}^{0,0}(\tilde{T}) = 0$ and $\tilde{H}^{1,0}(\tilde{T}) \cong \Z$. Thus $\ker(d^{0,0})=0$ and $d^{0,0}$ is injective. There is a short exact sequence
	\[
	0 \to \coker(d^{0,0}) \to \H^{1,0}(\tilde{T}) \to \ker(d^{1,0}) \to 0.
	\]
Observe $\ker(d^{1,0})\cong \Z$ so it must be that $\coker(d^{0,0})=0$, which implies $d^{0,0}$ is surjective as well. It remains to solve the extension problem of $\M$-modules shown in the middle grid in Figure \ref{fig:ztortilde}.

For each $(p,q)$ we must first solve the extension problem of abelian groups
	\[
	0 \to \coker(d^{p-1,q}) \to \tilde{H}^{p,q}(\tilde{T}) \to \ker(d^{p,q}) \to 0
	\]
The only possible nontrivial extensions occur in bidegrees $(1,2k+1)$ where we have
\begin{equation}\label{eq:exttildex}
	0 \to \Z \to \tilde{H}^{1,2k+1}(\tilde{T}) \to \Z \oplus \Z/2 \to 0.
\end{equation}
To solve this, consider a portion of the forgetful long exact sequence:
\begin{center}
	\begin{tikzcd}[column sep=small]
	\H^{0,2k}(\tilde{T})\arrow[r,"\rho"]&\H^{1,2k+1}(\tilde{T})	\arrow[r,"\psi"] & \H^{1}_{sing}(\tilde{T}).
\end{tikzcd}
\end{center}
Non-equivariantly $\tilde{T} \simeq T \vee S^{1}$ so $\H^{1}_{sing}(\tilde{T})\cong\Z^{\oplus 3}$. The above portion is thus
\begin{center}
	\begin{tikzcd}[column sep=small]
	0\arrow[r,"\rho"]&\H^{1,2k+1}(\tilde{T}) \arrow[r,"\psi"] & \Z^{3} 
\end{tikzcd}
\end{center}
Now the two options are $\H^{1,1}(\tilde{T})\cong \Z^{\oplus 2}$ if the extension in \ref{eq:exttildex} is nontrivial, or $\H^{1,1}(\tilde{T})\cong \Z^{\oplus 2}\oplus \Z/2$ if the extension is trivial. The above shows $\H^{1,1}(\tilde{T})$ must inject into $\Z^{3}$, so it must be that the extension is nontrivial, and the cohomology of $\tilde{T}$ is given by the far right picture in Figure \ref{fig:ztortilde}. In this case, the $\M$-module structure can be entirely determined from the middle picture once we know the extension as a bigraded abelian group. As usual, we have suppressed the action of $x$ from our picture, noting it is an isomorphism in all bidegrees.
\end{proof}

\begin{lemma}\label{ztorconnsumor} Let $Y$ be a non-equivariant surface and let $T$ be one of the two free tori. If $Y$ is orientable, then
	\[
	H^{*,*}(T\#_2 Y) \cong H^{*,*}(T) \oplus \left( \Sigma^{1,0}\A_0 \right)^{\oplus \beta(Y)}.
	\]
If $Y$ is nonorientable, then
	\[
	H^{*,*}(T\#_2 Y) \cong \A_1 \oplus \left( \Sigma^{1,0}\A_0 \right)^{\oplus \beta(Y)-1}\oplus  \Sigma^{1,0} F_1 \oplus  \Sigma^{1,1}F_1,
	\]
where $F_i=x^{-1}\M_2/(\rho^{i+1})$.
\end{lemma}
\begin{proof} 
Assume $T= T_1^{anti}$. The proof for $T=T_1^{rot}$ will follow similarly. Let $Y'$ be the non-equivariant space obtained by removing a disk from $Y$ in order to form $T\#_2 Y$. Consider the cofiber sequence
	\[
	\left(C_2\times Y'\right)_+ \hookrightarrow \left(T\#_2 Y\right)_+ \to \tilde{T}
	\]
where $\tilde{T}$ is the pinched space whose cohomology was found in Lemma \ref{pinchedtor}. By Lemma \ref{product}, the cohomology of $C_2\times Y'$ is given by
	\[
	H^{*,*}(C_2\times Y')\cong \A_0\otimes_{\Z} H^{*}_{sing}(Y') \cong \A_0 \oplus (\Sigma^{1,0} \A_0)^{\beta(Y)}.
	\]
Before illustrating the long exact sequence, we make the following observation. 

Consider the map $q: C_2\times Y' \to C_2$ that collapses each component to a point. This map induces an isomorphism in bidegrees $(0,j)$ for all $j$.  We have a map of cofiber sequences 
\begin{center}
\begin{tikzcd}[column sep=small, row sep=small]
	C_{2}\times Y' \arrow[r,hook] \arrow[d,"q"]& T\#_2Y \arrow[r] \arrow[d,"q"] & \tilde{T}\arrow[d,"id"]\\
	C_{2} \arrow[r,hook] & T \arrow[r] & \tilde{T}
\end{tikzcd}
\end{center}
which induces the map of long exact sequences
\begin{center}
\begin{tikzcd}[column sep=small]
	H^{0,j}\left(C_{2}\times Y'\right) \arrow[r] &\H^{1,j}(\tilde{T}) \arrow[r]& H^{1,j}(T\#_2 Y) \arrow[r]& H^{1,j}\left(C_{2}\times Y'\right) \\
	H^{0,j}(C_{2}) \arrow[r] \arrow[u,"q^*" left, "\cong" right]&\H^{1,j}(\tilde{T}) \arrow[r]\arrow[u,"id", "\cong" right]& H^{1,j}(T) \arrow[r]\arrow[u,"q^*" left]& H^{1,j}(C_{2})  \arrow[u,"q^*"left, hook]
\end{tikzcd}
\end{center}
A diagram chase shows $q^{*}:H^{1,j}(T)\to H^{1,j}(T\#_2 Y)$ must be injective. Similarly $q^{*}:H^{0,j}(T) \to H^{0,j}(T\#_2 Y)$ is injective.

Now let's consider the long exact sequence associated to the cofiber sequence. The differential $d: H^{*,*}\left(C_{2}\times Y'\right) \to \H^{*+1,*}(\tilde{T})$ is shown in Figure \ref{fig:eqtorconnsumz} in the orientable case and in Figure \ref{fig:nonorientcon} in the nonorientable case. The label $\beta(Y)$ is used to note there are $\beta(Y)$ copies of the red summand $\Sigma^{1,0}\A_0$.

\begin{figure}[ht]
	\begin{tikzpicture}[scale=0.45]
	\draw[help lines,gray] (-1.125,-5.125) grid (5.125, 5.125);
	\draw[<->] (-1,0)--(5,0)node[right]{$p$};
	\draw[<->] (0,-5)--(0,5)node[above]{$q$};
	\zline{0}{red};
	\zline{.7}{red};
	\zantioo{1.3}{blue};
	\foreach \y in {-4,-2,0,2,4}
		\zbox{1}{\y}{blue};
	\foreach \y in {-5,-3,-1,1,3}
		\zbox{1}{\y}{blue};
	\lab{.7}{\beta(Y)}{red};
	\draw[->, thick] (.6,.37)--(1.35,.37);
	\draw[->,thick] (1.35,.6)--(2.7,.6);
	\draw[->, thick] (.6,1.37)--(1.35,1.37);
	\draw[->,thick] (1.35,1.6)--(2.7,1.6);
	\end{tikzpicture}\hspace{0.2in}
	\begin{tikzpicture}[scale=0.45]
	\draw[help lines,gray] (-1.125,-5.125) grid (5.125, 5.125);
	\draw[<->] (-1,0)--(5,0)node[right]{$p$};
	\draw[<->] (0,-5)--(0,5)node[above]{$q$};
	\zline{.7}{red};
	\zantioo{1.3}{blue};
	\foreach \y in {-4,-2,0,2,4}
		{
		\zbox{1}{\y}{blue};
		\zbox{0}{\y}{red};
		};
	\foreach \y in {-5,-3,-1,1,3}
		\draw[blue] (1.5,\y+.5) node{\small{$\bullet$}};
	\lab{.7}{\beta(Y)}{red};
	\end{tikzpicture}\hspace{0.2in}
	\begin{tikzpicture}[scale=0.45]
	\draw[help lines,gray] (-1.125,-5.125) grid (5.125, 5.125);
	\draw[<->] (-1,0)--(5,0)node[right]{$p$};
	\draw[<->] (0,-5)--(0,5)node[above]{$q$};
	\zantioo{1.3}{black};
	\zantizo{-.3}{black};
	\zline{1}{purple};
	\lab{1}{\beta(Y)}{purple};
	\end{tikzpicture}
\caption{The differential when $Y$ is orientable.}
\label{fig:eqtorconnsumz}
\end{figure}
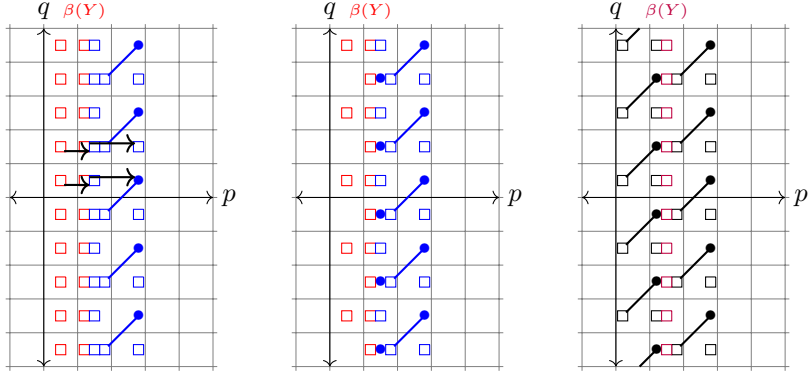
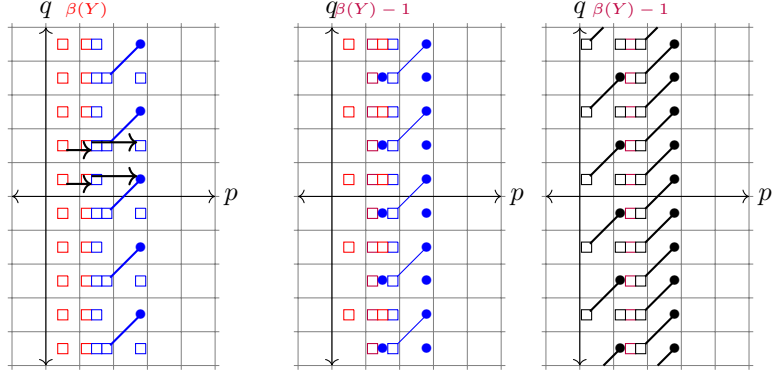
\begin{figure}[ht]
	\begin{tikzpicture}[scale=0.45]
	\draw[help lines,gray] (-1.125,-5.125) grid (5.125, 5.125);
	\draw[<->] (-1,0)--(5,0)node[right]{$p$};
	\draw[<->] (0,-5)--(0,5)node[above]{$q$};
	\zline{0}{red};
	\zline{.7}{red};
	\zantioo{1.3}{blue};
	\foreach \y in {-4,-2,0,2,4}
		\zbox{1}{\y}{blue};
	\foreach \y in {-5,-3,-1,1,3}
		\zbox{1}{\y}{blue};
	\lab{.7}{\beta(Y)}{red};
	\draw[->, thick] (.6,.37)--(1.35,.37);
	\draw[->,thick] (1.35,.6)--(2.7,.6);
	\draw[->, thick] (.6,1.37)--(1.35,1.37);
	\draw[->,thick] (1.35,1.6)--(2.7,1.6);
	\end{tikzpicture}\hspace{0.2in}
	\begin{tikzpicture}[scale=0.45]
	\draw[help lines,gray] (-1.125,-5.125) grid (5.125, 5.125);
	\draw[<->] (-1,0)--(5,0)node[right]{$p$};
	\draw[<->] (0,-5)--(0,5)node[above]{$q$};
	\zline{.7}{purple};
	\foreach\y in {1,3,-1,-3,-5}
		\draw[blue] (1.8,\y+.5)--(2.8,\y+1.5) node{\small{$\bullet$}};
	\foreach\y in {1, 3, -1, -3, -5}
		\zbox{1.3}{\y}{blue};
	\foreach \y in {-4,-2,0,2,4}
		{
		\zbox{1.3}{\y}{blue};
		\zbox{1}{\y}{red};
		\zbox{0}{\y}{red};
		};
	\foreach \y in {-5,-3,-1,1,3}
		{
		\draw[blue] (1.5,\y+.5) node{\small{$\bullet$}};
		\draw[blue] (2.8,\y+.5) node{\small{$\bullet$}};
		};
	\lab{.7}{\beta(Y)-1}{purple};
	\end{tikzpicture}
	\begin{tikzpicture}[scale=0.45]
	\draw[help lines,gray] (-1.125,-5.125) grid (5.125, 5.125);
	\draw[<->] (-1,0)--(5,0)node[right]{$p$};
	\draw[<->] (0,-5)--(0,5)node[above]{$q$};
	\zline{1}{purple};
	\zantizo{-.3}{black};
	\draw[thick] (1.8,4.5)--(2.3,5);
	\foreach\y in {-5,-4,-3,-2,-1,0,1,2,3}
		\draw[thick] (1.8,\y+.5)--(2.8,\y+1.5) node{\small{$\bullet$}};
	\foreach\y in {-5,-4,-3,-2,-1,0,1,2,3,4}
		\zbox{1.3}{\y}{black};
	\lab{1}{\beta(Y)-1}{purple};
	\draw[thick] (2.3,-5)--(2.8,-4.5) node{\small{$\bullet$}};
	\end{tikzpicture}
\caption{The differential when $Y$ is nonorientable.}
\label{fig:nonorientcon}
\end{figure}

We need to determine the differentials in bidegrees $(0,0)$, $(1,0)$, $(0,1)$, and $(1,1)$. Recall $T\cong T^{anti}_1$ and so the quotient $(T\#_2Y)/C_2$ is a nonorientable surface. Thus
	\[
	H^{2}_{sing}((T\#_2Y)/C_2)\cong\Z/2 \quad \text{ and }\quad H^{0}_{sing}((T\#_2Y)/C_2)\cong \Z.
	\] 
Hence $d^{1,0}=d^{0,0}=0$ by the quotient lemma. 

For $d^{0,1}$ recall we have an injective map $q^*: H^{1,1}(T)\to H^{1,1}(T\#_2Y)$ and $H^{1,1}(T)\cong \Z\oplus \Z/2$. The differential cannot be zero because if it were, then $H^{1,1}(T\#_2Y)\cong \Z^{\beta(Y) +2}$ which has no $2$-torsion. Now $d^{0,1}$ is a map $\Z\to \Z^2$, so we can conclude $\coker(d^{0,1})\cong \Z \oplus \Z/2n$ for some $n\geq 1$. To determine $n$, consider the portion of the forgetful long exact sequence: 
	\[
	H^{0,0}(T\#_2Y)\overset{\rho}{\longrightarrow} H^{1,1}(T\#_2Y)\overset{\psi}{\longrightarrow} H^{1}_{sing}(T\#_2Y),
	\]
which is given by
	\[
	\Z \overset{\rho}{\longrightarrow} \Z^{\oplus \beta(Y)+1}\oplus \Z/2n\overset{\psi}{\longrightarrow} H^{1}_{sing}(T\#_2Y).
	\]
The group $H^{1}_{sing}(T\#_2Y)$ has no torsion regardless of the orientability of $Y$, so exactness implies $\Z/2n \subset \ker(\psi)=\text{im}(\rho)$. All elements in the image of $\rho$ are $2$-torsion, so $n=1$ and $\coker(d^{0,1})\cong \Z \oplus \Z/2$. 

To compute the differential in bidegree $(1,1)$, we consider another portion of the forgetful long exact sequence:
\begin{center}
\begin{tikzcd}[column sep=small]
	\arrow[r,"\rho"]&\arrow[r] H^{2,1}(T\#_2Y)\arrow[r,"\psi"]&H^{2}_{sing}(T\#_2Y)\arrow[r]&H^{2,0}(T\#_2Y)\arrow[r,"\rho"]& H^{3,1}(T\#_2 Y)
\end{tikzcd}
\end{center}
which gives the exact sequence
	\[
	0 \to \coker(\rho) \to H^{2}_{sing}(T\#_2Y)\to H^{2,0}(T\#_2Y)\to 0.
	\]
There are two cases for the above based on the orientability of $Y$. If $Y$ is orientable, then we have
\[
0 \to \coker(\rho) \to \Z\to \Z/2\to0,
\]
so $\coker(\rho)\cong \Z$. From Figure \ref{fig:eqtorconnsumz} this is only possible if $H^{2,1}(T\#_2Y)\cong \Z$ and $d^{1,1}$ is zero. Thus we know the entire differential when $Y$ is orientable. If $Y$ is nonorientable, then we have
\[
0 \to \coker(\rho) \to \Z/2\to \Z/2\to0,
\]
so $\coker(\rho)=0$ and $\rho$ is surjective. Observe $H^{2,1}(X)\cong \coker(d^{1,1})$ and thus isomorphic to a quotient of $\Z$. Given that the image of $\rho$ is $2$-torsion, the only possibilities are $H^{2,1}(X\#_2 Y)=0$ or $H^{2,1}(X \#_2 Y)\cong \Z/2$. We can eliminate the former case by considering the following long exact sequence which is induced by the short exact sequence of Mackey functors $\uZ \to \uZ \to \underline{\Z/2}$:
	\[
	\dots \to H^{2,1}(X;\underline{\Z}) \overset{2}{\to} H^{2,1}(X;\underline{\Z}) \to H^{2,1}(X;\underline{\Z/2})\to H^{3,1}(X;\underline{\Z})\to\dots
	\]
From the cofiber sequence pictured in Figure \ref{fig:eqtorconnsumz}, we know $H^{3,1}(X;\underline{\Z})=0$. From the computations in \cite[Theorem 5.5]{Haz20}, we know $H^{2,1}(X;\underline{\Z/2})\cong \Z/2$. Hence $H^{2,1}(X;\underline{\Z}) \neq 0$ and so $\coker(d^{1,1})=\Z/2$. 

We now know all pieces of the differential in both cases, and it remains to solve the extension problems shown in the middle graphs in Figures \ref{fig:eqtorconnsumz} and \ref{fig:nonorientcon}. Beginning with the case when $Y$ is orientable, note the only possibility for a nontrivial extension is for $\rho$ to act on $H^{0,2k}(T\#_2Y)$ nontrivially. We already know this to be the case because $H^{\epsilon,*}(T)$ injects into $H^{\epsilon,*}(T\#_2Y)$ for $\epsilon = 0,1$. We conclude the cohomology of $T\#_2Y$ is isomorphic to module shown in the far right in Figure \ref{fig:eqtorconnsumz}, which is exactly $H^{*,*}(T) \oplus \left(\Sigma^{1,0}\A_0 \right)^{\beta(Y)}$. In the case when $Y$ isn't orientable, we similarly get the $\rho$-connections from $(0,2k)$ to $(1, 2k+1)$ using the injection from $H^{\epsilon,*}(T)$. The other $\rho$ connections from topological degree $1$ to topological degree $2$ follow from the previous paragraph. 
\end{proof}

We also need the cohomology of $S^2_a\#_2 Y$. The computation follows the same outline as the one given above, so we leave the details to the reader. 

\begin{lemma}\label{freesphereconnsum}
Let $Y$ be surface. If $Y$ is orientable, then 
\[H^{*,*}(S^2_a \#_2 Y) \cong \A_2 \oplus \left( \Sigma^{1,0} \A_0 \right)^{\beta(Y)}.\] If $Y$ is nonorientable, then 
\[H^{*,*}(S^2_a \#_2 Y) \cong F_2\oplus \Sigma^{1,0} F_1 \oplus \left( \Sigma^{1,0} \A_0 \right)^{\beta(Y)-1}.\]
\end{lemma}

We now state the main theorem.

\begin{theorem}\label{freeanswer} Let $X$ be a free $C_2$-surface and $\beta=\dim(H^1_{sing}(X;\Z/2))$. Denote by $F_i$ the $\M$-module $x^{-1}\M/(\rho^{i+1})$. There are four cases for the cohomology of $X$.
\begin{enumerate}
\item[(i)] Suppose the underlying space is an odd genus orientable surface. Let $\epsilon=0$ if the action is orientation preserving and $\epsilon=1$ if the action is orientation reversing. Then 
\[H^{*,*}(X) \cong \A_1 \oplus \Sigma^{1,\epsilon}\A_1\oplus \left(\Sigma^{1,0}\A_0\right)^{\oplus \beta/2-1}.\]
\item[(ii)] Suppose the underlying space is an even genus orientable surface. Then 
\[H^{*,*}(X) \cong \A_2 \oplus \left(\Sigma^{1,0}\A_0\right)^{\oplus \beta/2}.\]
\item[(iii)] Suppose the underlying space is nonorientable and $X\cong T_1^{anti}\#_2Y$ for some surface $Y$. Then 
\[H^{*,*}(X) \cong \A_1 \oplus \Sigma^{1,0}F_1\oplus \Sigma^{1,1}F_1 \oplus \left(\Sigma^{1,0}\A_0\right)^{\oplus \beta/2-2}.\]
\item[(iv)] Suppose the underlying space is nonorientable and $X\cong S^2_{a}\#_2Y$ for some surface $Y$. Then 
\[H^{*,*}(X) \cong F_2 \oplus \Sigma^{1,0}F_1 \oplus  \left(\Sigma^{1,0}\A_0\right)^{\oplus \beta/2-1}. \]
\end{enumerate}
\end{theorem}
\begin{proof} For (i) the classification in Theorem \ref{freeclass} implies $X\cong T_1^{rot}\#_2 T_g$ or $X\cong T_1^{anti} \#_2 T_g$. Then the computation follows from Lemma \ref{ztorconnsumor} after noting $\beta(X)=2\beta(T_g)+2$. In (ii), the classification implies $X\cong S^2_a \#_2 T_g$ and so this follows from Lemma \ref{freesphereconnsum} and the fact that $\beta(X)=2\beta(T_g)$. Lastly (iii) and (iv) follow from Lemmas \ref{ztorconnsumor} and \ref{freesphereconnsum} after noting $\beta(X)=2\beta(Y)+2$ in (iii) and $\beta(X)=2\beta(Y)$ in (iv).
\end{proof}

We have the following corollary based on the classification of free $C_2$-surfaces and the above computation. This does not hold for $C_2$-surfaces with nonfree actions. See Remark \ref{nonfreenoninv}.

\begin{cor}
If $X$ and $Y$ are two free $C_2$-surfaces, then $H^{*,*}(X;\underline{\Z}) \cong H^{*,*}(Y;\underline{\Z})$ as $\M$-modules if and only if $X\cong Y$ as $C_2$-spaces.
\end{cor}
\begin{proof} This follows from the complete classification in \cite{Dug19} and the above theorem. 
\end{proof}

%%%%%%%%%%%%%%%%%

\section{Computations for nonfree $C_2$-surfaces}
\label{sec:nonfreecomp}
We compute the cohomology of all $C_2$-surfaces whose $C_2$-action is nontrivial and nonfree. Recall from Section \ref{sec:backgroundsurg} that we can construct new $C_2$-surfaces from old by attaching $S^{1,0}$- or $S^{1,1}$-antitubes, or by removing a neighborhood of an isolated fixed point and sewing in a M\"obius band, which we refer to as $FM$-surgery. The following lemmas explain how such surgery affects the cohomology in some cases.

\begin{lemma} \label{attach_antitube10}
Suppose $X$ is a nontrivial $C_2$-surface whose fixed set contains at least one fixed circle. Then 
$ \H^{*,*}(X+n[S^{1,0}-AT])\cong \H^{*,*}(X) \oplus \left( \Sigma^{1,0}\M \right)^{\oplus n} \oplus \left(\Sigma^{1,1} \M \right)^{\oplus n}.$
\end{lemma}
\begin{proof} We induct on the number of antitubes attached. The base case $n=0$ follows immediately. Thus assume for some $n\geq 1$, 
\[
\H^{*,*}(X+(n-1)[S^{1,0}-AT])\cong \H^{*,*}(X) \oplus \left(\Sigma^{1,0}\M \right)^{\oplus n-1} \oplus \left(\Sigma^{1,1}\M \right)^{\oplus n-1}
\]
and consider $X+ n[S^{1,0}-AT]$. We can include a copy of $S^{1,0}$ into one of the attached antitiubes to get the cofiber sequence
\[
S^{1,0} \hookrightarrow X+ n[S^{1,0}-AT] \to (X+(n-1)[S^{1,0}-AT])^\sim
\]
where $(X+(n-1)[S^{1,0}-AT])^\sim=\cof(C_2\hookrightarrow X+(n-1)[S^{1,0}-AT])$ is the usual pinched space. By assumption, the fixed set $X^{C_2}$ is nonempty, so it follows from Lemma \ref{pinched} that
\[
(X+(n-1)[S^{1,0}-AT])^\sim\simeq (X+(n-1)[S^{1,0}-AT]) \vee S^{1,1}.
\]
From Lemma \ref{miscmani}, we know $H^{2,0}(X+(n-1)[S^{1,0}-AT])=0$. 

Let's return to the cofiber sequence. The domain of the differential is the free module $\H^{*,*}(S^{1,0})$. But the codomain is $0$ in bidegree $(2,0)$, so the differential must be zero. Hence the kernel of the differential is free, and so the short exact sequence
\[
0 \to \coker(d) \to \H^{*,*}(X+ n[S^{1,0}-AT]) \to \ker(d) \to 0
\]
splits. We have that
\begin{align*}
\H^{*,*}(X+ n[S^{1,0}-AT]) 
&\cong \coker(d) \oplus \ker(d) \\
&\cong \H^{*,*}((X+(n-1)[S^{1,0}-AT]) \vee S^{1,1}) \oplus \H^{*,*}(S^{1,0})\\
&\cong \left(\H^{*,*}(X) \oplus  \left( \Sigma^{1,0}\M \right)^{\oplus n-1} \oplus \left(\Sigma^{1,1} \M \right)^{\oplus n-1}\oplus \Sigma^{1,1}\M\right)\\
&\quad \oplus \H^{*,*}(S^{1,0})
\\
&\cong \H^{*,*}(X) \oplus \left( \Sigma^{1,0}\M \right)^{\oplus n} \oplus \left(\Sigma^{1,1} \M \right)^{\oplus n},
\end{align*}
where the third isomorphism uses the induction hypothesis. This completes the proof.
\end{proof}
\begin{lemma}\label{attach_antitube11}
Suppose $X$ is a nonfree $C_2$-surface whose fixed set contains only isolated fixed points. Then 
$ \H^{*,*}(X+n[S^{1,1}-AT])\cong \H^{*,*}(X) \oplus \left(\Sigma^{1,1} \M \right)^{\oplus 2n}.$
\end{lemma}
\begin{proof}
The proof is analogous to the proof of Lemma \ref{attach_antitube10}, except we include $S^{1,1}$ instead of $S^{1,0}$ into one of the attached antitubes.
\end{proof}

\begin{lemma}\label{attach_fm}
Suppose $X$ is a $C_2$-surface whose fixed set contains at least one fixed circle and at least one isolated fixed point. Then $\H^{*,*}(X+n[FM]) \cong \H^{*,*}(X) \oplus \left(\Sigma^{1,0} \M \right)^{\oplus n}$.
\end{lemma}
\begin{proof}
The proof is similar to the proof of Lemma \ref{attach_antitube10}, though we want to consider a slightly different cofiber sequence in the induction step. Let $n\geq 1$ and consider $X+n[FM]$. Let $U$ be an equivariant, closed neighborhood around one of the attached M\"obius bands $M$ such that $U\simeq M$. If we quotient by $U$, then the attached M\"obius band is turned back into an isolated fixed point. Thus we have the cofiber sequence
\[
U \hookrightarrow X+n[FM] \to (X+n[FM])/U\simeq X+(n-1)[FM].
\]
Observe $U\simeq M\simeq S^{1,0}$ and from Lemma \ref{miscmani}, $H^{2,0}(X+(n-1)[FM])=0$. Thus the differential from the above cofiber sequence must be zero. This gives the short exact sequence
\[
0 \to \H^{*,*}(X+(n-1)[FM]) \to \H^{*,*}(X+n[FM] ) \to \H^{*,*}(S^{1,0})\to 0.
\]
Since $ \H^{*,*}(S^{1,0})$ is free, this splits, and so
\[
\H^{*,*}(X+n[FM] )\cong \H^{*,*}(X+(n-1)[FM]) \oplus \H^{*,*}(S^{1,0}) \cong \H^{*,*}(X) \oplus \left(\Sigma^{1,0} \M \right)^{\oplus n}.
\]
Note we do not get a copy of $\Sigma^{1,1}\M$ as we did in the proof of Lemma \ref{attach_antitube10} because the cofiber in this sequence is not a pinched space.
\end{proof}

\subsection{Orientable surfaces} We now consider nontrivial, nonfree $C_2$-surfaces whose underlying space is orientable. All such surfaces can be built by attaching antitubes to a free $C_2$-surface or to a surface of the form $S^{2,2}\#_2 T_g$ or $S^{2,1}\#_2 T_g$. We start by computing the cohomology of these connected sums, and then consider attaching handles to free $C_2$-surfaces.

\begin{lemma}\label{orientcomp1}
If $X=S^{2,\epsilon}\#_2 T_g$ for $\epsilon=1,2$, then $\H^{*,*}(X)\cong \left(\Sigma^{1,0}\A_0\right)^{g}\oplus \Sigma^{2,\epsilon}\M$.
\end{lemma}
\begin{proof}
We can include an equatorial circle from $S^{2,\epsilon}$ to get a cofiber sequence
\begin{equation}
S^{1,\epsilon-1}\hookrightarrow S^{2,\epsilon}\#_2T_g \to C_{2+}\wedge T_g.
\end{equation}
The long exact sequence for the associated cofiber sequence in the case $\epsilon=2$ is illustrated in Figure \ref{fig:s21tg} below.

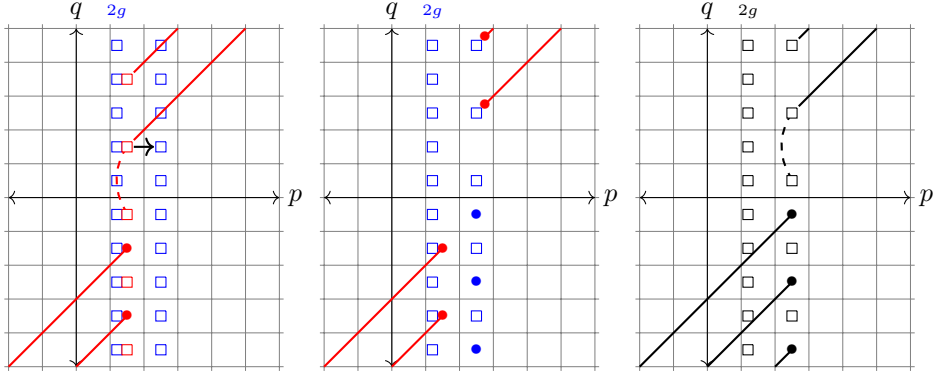
\begin{figure}[ht]
\begin{tikzpicture}[scale=0.45]
\draw[help lines,gray] (-2.125,-5.125) grid (6.125, 5.125);
\draw[<->] (-2,0)--(6,0)node[right]{$p$};
\draw[<->] (0,-5)--(0,5)node[above]{$q$};
	\zline{.7}{blue};
	\lab{.7}{2g}{blue};
	\zline{2}{blue};
	\zcone{1}{1}{red};
	\draw[->, thick] (1.7,1.5)--(2.3,1.5);
	\zbox{1}{-5}{red};
\end{tikzpicture}
\begin{tikzpicture}[scale=0.45]
\draw[help lines,gray] (-2.125,-5.125) grid (6.125, 5.125);
\draw[<->] (-2,0)--(6,0)node[right]{$p$};
\draw[<->] (0,-5)--(0,5)node[above]{$q$};
	\zline{.7}{blue};
	\lab{.7}{2g}{blue};
	\foreach \y in {-4,-2,0,2,4}
		\zbox{2}{\y}{blue};
	\foreach \y in {-3,-1,-5}
		\draw[blue] (2.5,\y+.5) node{\dot};
	\foreach \y in {2,4}
		\draw[red,thick] (2.75,\y+.75)node{\dot}--(7-\y,5);
	\foreach \y in {-4,-2}
		\draw[red,thick] (1.5,\y+.5)node{\dot}--(-4-\y,-5);
\end{tikzpicture}
\begin{tikzpicture}[scale=0.45]
\draw[help lines,gray] (-2.125,-5.125) grid (6.125, 5.125);
\draw[<->] (-2,0)--(6,0)node[right]{$p$};
\draw[<->] (0,-5)--(0,5)node[above]{$q$};
	\zline{.7}{black};
	\lab{.7}{2g}{black};
	\zcone{2}{2}{black};
	\zbox{2}{-4}{black};
	\draw[thick] (2,-5)--(2.5,-4.5) node{\dot};
\end{tikzpicture}
\caption{The differential $d:\H^{*,*}(S^{1,1})\to \H^{*,*}(C_{2+}\wedge T_g)$.}
\label{fig:s21tg}
\end{figure}

We know $\Sigma^{2,\epsilon}\M$ is a submodule of $\H^{*,*}(X)$ by Theorem \ref{topm}. In particular, $\H^{2,\epsilon-3}(X)$ must have a $2$-torsion subgroup coming from $\mu$. This can only happen if $d^{1,\epsilon-3}$ is multiplication by $2n$. We can see $n$ must be $\pm 1$ by noting $H^{2}_{sing}(X)=\Z$, and so by the forgetful long exact sequence, all torsion elements in $H^{2,q}(X)$ must be in the image of $\rho$ for all weights $q$. Using the module structure, we see that $d^{1,\epsilon-1}$ is an isomorphism. This then determines all of $d$.
 
It remains to solve the extension problem shown on the middle graph in Figure \ref{fig:s21tg}. Using the forgetful long exact sequence or another application of Theorem \ref{topm}, we can conclude the extension problem is solved by the picture on the right.
\end{proof}

\begin{lemma}\label{orientcomp2}
If $X=(S^2_a\#_2 T_g)+[S^{1,0}-AT]$ or $X=(T_1^{anti}\#_2 T_g)+[S^{1,0}-AT]$, then \[\H^{*,*}(X)\cong \left(\Sigma^{1,0}\A_0 \right)^{\beta(X)/2} \oplus \Sigma^{2,1}\M.\]
\end{lemma}
\begin{proof}
Let's first consider $X=(S^2_a\#_2 T_g)+[S^{1,0}-AT]$. We make use of the cofiber sequence arising from including the fixed circle in the attached antitube
\[
S^{1,0}\hookrightarrow X \to \left(S^{2}_a\#_2T_g\right)^\sim,
\]
where as usual $\left(S^{2}_a\#_2T_g\right)^\sim=\cof(C_2\hookrightarrow S^{2}_a\#_2T_g)$. We know the cohomology of $S^{2}_a\#_2T_g$ from Theorem \ref{freeanswer}, and it is now a routine computation to show that \[\H^{*,*}\left(\left(S^{2}_a\#_2T_g\right)^\sim\right) \cong \left(\Sigma^{1,0}\A_0\right)^{\oplus 2g} \oplus \Sigma^{1,1}\A_1.\] 
In Figure \ref{fig:basecofib}, we illustrate the differential
\[d:\tilde{H}^{*,*}(S^{1,0}) \to \tilde{H}^{*+1,*}\left(\left(S^{2}_a\#_2T_g\right)^\sim\right).\]

\begin{figure}[ht]
\begin{tikzpicture}[scale=0.45]
\draw[help lines,gray] (-2.125,-5.125) grid (6.125, 5.125);
\draw[<->] (-2,0)--(6,0)node[right]{$p$};
\draw[<->] (0,-5)--(0,5)node[above]{$q$};
\zline{.7}{blue};
\lab{.7}{2g}{blue};
\zcone{1}{0}{red};
\zantioo{1.3}{blue};
\draw[->, thick] (1.7,.5)--(2.7,.5);
\end{tikzpicture}
\begin{tikzpicture}[scale=0.45]
\draw[help lines,gray] (-2.125,-5.125) grid (6.125, 5.125);
\draw[<->] (-2,0)--(6,0)node[right]{$p$};
\draw[<->] (0,-5)--(0,5)node[above]{$q$};
\zline{.7}{blue};
\lab{0.7}{2g}{blue};
\draw[red, thick] (2.3,1.3)node{\small{$\bullet$}}--(6,5);
\draw[red, thick] (2.3,3.3)node{\small{$\bullet$}}--(4,5);
\draw[red, thick] (1.5,-2.5)node{\small{$\bullet$}}--(-1,-5);
\draw[red, thick] (1.5,-4.5)node{\small{$\bullet$}}--(1,-5);
\foreach \y in {-4,-2,0,2,4}
	\zbox{1}{\y}{red};
\foreach \y in {-5,-3}
	\draw[thick,blue] (1.8,\y+.6)--(2.7,\y+1.5)node{\small{$\bullet$}};
\foreach \y in {-5,-3,-1,1,3}{
	\zbox{1.3}{\y}{blue};
	\zbox{2.2}{\y}{blue};
	}
\end{tikzpicture}
\begin{tikzpicture}[scale=0.45]
\draw[help lines,gray] (-2.125,-5.125) grid (6.125, 5.125);
\draw[<->] (-2,0)--(6,0)node[right]{$p$};
\draw[<->] (0,-5)--(0,5)node[above]{$q$};
	\zline{1}{black};
	\lab{1}{2g+1}{black};
	\zcone{2}{1}{black};
	\zbox{2}{-5}{black};
\end{tikzpicture}
\caption{$d:\tilde{H}^{*,*}(S^{1,0})\to \tilde{H}^{*+1,*}(\left(S^{2}_a\#_2T_g\right)^\sim)$.}
\label{fig:basecofib}
\end{figure}

The fixed set $X^{C_2}$ has codimension $1$. Hence $H^{2,0}(X)=0$ by Lemma \ref{miscmani}, and so $d^{1,0}$ must be surjective. 

Now that we have determined the differential, we must solve the extension problem shown in the middle grid in Figure \ref{fig:basecofib}. From Theorem \ref{topm}, we know this must be nontrivial since $\Sigma^{2,1}\M$ is a submodule of $\H^{*,*}(X)$. This guarantees $\rho^2$-connections from bidegrees $(0,-2k-2)$ to $(2,-2k)$ for $k\geq 1$. Now consider bidegree $(2,2k+1)$ for some $k\geq 0$. If the extension of abelian groups was trivial, then the action by $\rho x^k$ from bidegree $(2,-1)$ to $(3,2k)$ would be forced to be trivial (note we can't have nontrivial $\rho$-extensions from $\coker(d)$ to $\ker(d)$ because $\coker(d)$ is the submodule of $\H^{*,*}(X)$). This contradicts that $\Sigma^{2,1}\M$ is a submodule, so the abelian group extension must be nontrivial. 

After a basis change in bidegrees $(1,-2k-1)$, we can conclude the cohomology is given by the right-hand grid. Lastly observe \[\beta((S^2_a\#_2 T_g)+[S^{1,0}-AT])=4g+2,\] so $\beta(X)/2=2g+1$, which completes this case. The computation for $X=(T_1^{anti}\#_2 T_g)+[S^{1,0}-AT]$ will follow the same steps.
\end{proof}

\begin{theorem}\label{nonfreeoranswerz} Let $X$ be a nonfree, nontrivial $C_2$-surface whose underlying space is orientable. Let $F$ be the number of isolated fixed points, $C$ be the number of fixed circles, and $\beta$ be the dimension of $H^{1}_{sing}(X;\Z/2)$. There are two cases for the cohomology of $X$:
\begin{enumerate}
\item[(i)] Suppose $F\neq 0$. Then \[\H^{*,*}(X)\cong \left(\Sigma^{1,1}\M\right)^{\oplus F-2} \oplus \left(\Sigma^{1,0}\A_0\right)^{\oplus \frac{\beta-F}{2}+1} \oplus \Sigma^{2,2}\M.\]
\item[(ii)] Suppose $C\neq 0$. Then
\[\H^{*,*}(X)\cong \left(\Sigma^{1,0}\M\right)^{\oplus C-1} \oplus \left(\Sigma^{1,1}\M\right)^{\oplus C-1}\oplus \left(\Sigma^{1,0}\A_0\right)^{\oplus \frac{\beta-2C}{2}+1} \oplus \Sigma^{2,1}\M.\]
\end{enumerate}
\end{theorem}
\begin{proof} This follows from the classification given in Theorem \ref{nonfreeclasstorus}, the way surgery affects $F$, $C$, and $\beta$ as described in Lemma \ref{surgeffects}, and the computations done in Lemmas \ref{attach_antitube10}, \ref{attach_antitube11}, \ref{orientcomp1}, and \ref{orientcomp2}.
\end{proof}

\subsection{Nonorientable surfaces}
We finish with $C_2$-surfaces whose action is nonfree and nontrivial, and whose underlying space is nonorientable. We start with surfaces whose fixed set only contains isolated fixed points. 

\begin{lemma}\label{onlyFnonor}
Suppose $X$ is a nonfree $C_2$-surface whose underlying space is nonorientable and whose fixed set consists only of isolated points. Then
\[
\H^{*,*}(X)\cong \left(\Sigma^{1,1}\M \right)^{F-2} \oplus \left( \Sigma^{1,0}\A_0\right)^{\frac{\beta-F}{2}} \oplus \Sigma^{1,1}\D_4.
\]
\end{lemma}
\begin{proof} From the classification in Theorem \ref{nonfreeclassnonorf}, it must be that
\[X\cong (S^{2,2}\#_2 N_r) +m[S^{1,1}-AT]\]
for some $r>0$ and $m\geq 0$. Recall in Theorem \ref{x2ncohom} we computed the cohomology of a $C_2$-CW complex $X_2$ that is homeomorphic to $S^{2,2}\#_2 \R P^2$. The computation done in that theorem can easily be modified for $S^{2,2}\#_2 N_r$ when $r>1$ as follows. 

The first cofiber sequence \ref{eq:x2ncofib1} can be modified to
\[
(C_2\times N_r')_+ \hookrightarrow S^{2,2}\#_2 N_r \to \tilde{S}^{2,2},
\]
and the long exact sequence in Figure \ref{fig:d4n_cof1} will then have $r$-summands of the form $\Sigma^{1,0}\A_0$. As in the case $r=1$, we can conclude $\H^{2,2k}(S^{2,2}\#_2 N_r)\cong \Z/4$ and $\H^{2,2k+1}(S^{2,2}\#_2 N_r)=0$ for $k\geq 1$ from this sequence. The second cofiber sequence \ref{eq:x2ncofib2} can be modified to
\[
S^{1,1}\hookrightarrow S^{2,2}\#_2 N_r \to C_{2+}\wedge N_r.
\]
Then in Figure \ref{fig:d4n_cof2}, there will be $(r-1)$ summands of the form $\Sigma^{1,0}\A_0$, but these new summands will not change the differentials or the extension problem. Thus
\[
\H^{*,*}(S^{2,2}\#_2 N_r) \cong \left(\Sigma^{1,0}\A_0 \right)^{\oplus r-1} \oplus \Sigma^{1,1}\D_4.
\]

Lastly, the $C_2$-surface $S^{2,2}\#_2 N_r$ has a fixed set that only contains isolated fixed points, so from Lemma \ref{attach_antitube11},
\[
\H^{*,*}(S^{2,2}\#_2 N_r +m[S^{1,1}-AT]) \cong \H^{*,*}(S^{2,2}\#_2 N_r) \oplus \left( \Sigma^{1,1}\M \right)^{\oplus 2m}.
\]
Observe $\beta(X)=2m+2r$ and $F(X)=2m+2$, so $F-2=2m$ and $\frac{\beta-F}{2}=r-1$. This completes the proof. 
\end{proof}

Now we consider the various ways to introduce fixed circles using surgery. 

\begin{lemma}\label{free_10tube}
Let $Y$ be a free $C_2$-surface such that $Y+[S^{1,0}-AT]$ is a $C_2$-surface whose underlying space is nonorientable. Then 
\[\H^{*,*}(Y+[S^{1,0}-AT])\cong (\Sigma^{1,0}\A_0)^{\oplus \beta(Y)/2} \oplus \Sigma^{1,0} \D_4.\] 
\end{lemma}
\begin{proof}
There are two cases: either $Y$ is nonorientable, or $Y$ is orientable and the action on $Y$ is orientation preserving (if it was orientation reversing, then $Y+[S^{1,0}-AT]$ would be orientable). In both cases, we make use of the cofiber sequence \[S^{1,0}\hookrightarrow Y +[S^{1,0}-AT] \to \tilde{Y},\] where $\tilde{Y}$ is the usual cofiber of $C_2\hookrightarrow Y$. To use this cofiber sequence, we need to know the cohomology of $\tilde{Y}$. We start by proving two general facts about the differential appearing in the cofiber sequence $Y_+\to \tilde{Y} \to \Sigma^{1,0}C_{2+}$. 

First, we show the differential $d^{0,0}:\H^{0,0}(Y_+)\to \H^{1,0}(\Sigma^{1,0}C_{2+})$ is an isomorphism. Note $\H^{*,*}(\Sigma^{1,0}C_{2+})\cong \Sigma^{1,0} \A_0$, so in particular $\H^{0,*}(\Sigma^{1,0}C_{2+})=0$ and $\H^{1,*}(\Sigma^{1,0}C_{2+})=\Z$. Thus $\H^{0,0}(\tilde{Y})\cong \ker (d^{0,0})$. The space $\tilde{Y}$ is connected, so $\H^{0,0}(\tilde{Y})=0$ which implies $\ker(d^{0,0})=0$. Thus $d^{0,0}$ is an injective map from $\Z$ to $\Z$. We have a short exact sequence 
\[0\to \coker(d^{0,0})\to \H^{1,0}(\tilde{Y})\to \ker(d^{1,0})\to 0,\] 
and by the quotient lemma $\H^{1,0}(\tilde{Y})\cong \H^{1}_{sing}(\tilde{Y}/C_2)$. This is torsion free because the first singular cohomology group is always torsion free (one can see this from the universal coefficient theorem). Thus $\coker(d^{0,0})$ cannot have any torsion, so the injective map $d^{0,0}:\Z \to \Z$ must be an isomorphism.

Next, we note $\H^{1,1}(\tilde{Y})$ is torsion free. To see this, recall $\H^{0,0}(\tilde{Y})=0$ and so the forgetful long exact sequence implies the forgetful map $\H^{1,1}(\tilde{Y}) \to \H^1_{sing}(\tilde{Y})$ is injective. The group $\H^1_{sing}(\tilde{Y})$ is torsion free, and hence $\H^{1,1}(\tilde{Y})$ is torsion free.

The cohomology of $\tilde{Y}$ is readily computed using these two facts. We outline the arguments and then leave the details to the reader. In the case that $Y$ is nonorientable, we have that $Y=S^2_a\#_2 Z$ or $Y=T_1^{anti}\#_2 Z$ for some nonorientable surface $Z$. In either case, we can use the answer for $H^{*,*}(Y)$ given in Theorem \ref{freeanswer}, and then the two facts above to get that $$\H^{*,*}(\tilde{Y})\cong \Sigma^{1,0}F_1 \oplus \Sigma^{1,1}F_1\oplus \left( \Sigma^{1,0}\A_0\right)^{\oplus \beta(Y)/2-1}.$$ In the case that $Y$ is orientable, we have that $Y\cong T_1^{rot}\#_2 T_g$ for some $g$, and get $$\H^{*,*}(\tilde{Y})\cong \Sigma^{1,0}\A_1 \oplus \left(\Sigma^{1,0} \A_0 \right)^{\oplus \beta(Y)/2}.$$ 

We can now compute the cohomology of $Y+[S^{1,0}-AT]$. There are two cases for the differential in the sequence associated to 
\[S^{1,0}\hookrightarrow Y +[S^{1,0}-AT] \hookrightarrow \tilde{Y}\] based on the orientability of $Y$; see Figures \ref{fig:10nonor} and \ref{fig:10or}. Note we have omitted the $(\Sigma^{1,0}\A_0)$-summands from both figures. In both cases, the differential has to be surjective because $H^{2,0}(Y+[S^{1,0}-AT])= 0$ by Lemma \ref{miscmani}. Thus we just have to solve the extension problem shown in the middle.

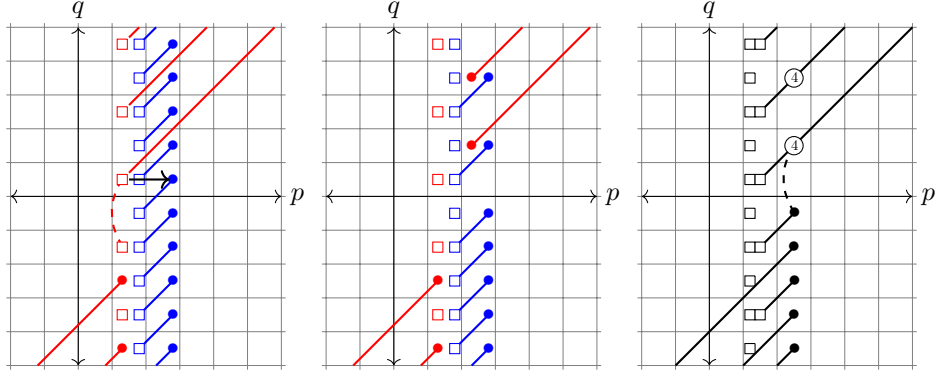
\begin{figure}[ht]
\begin{tikzpicture}[scale=0.45]
\draw[help lines,gray] (-2.125,-5.125) grid (6.125, 5.125);
\draw[<->] (-2,0)--(6,0)node[right]{$p$};
\draw[<->] (0,-5)--(0,5)node[above]{$q$};
\foreach \y in{-5,...,3}
	\draw[blue,thick] (1.8,\y+0.5)--(2.8,\y+1.5);
\draw[blue, thick] (2.3,-5)--(2.8,-4.5);
\draw[blue, thick] (1.8,4.5)--(2.3,5);
\zline{1.3}{blue};
\dotline{2.3}{blue};
\draw[->, thick] (1.5,0.5)--(2.7,0.5);
\zcone{0.8}{0}{red};
\end{tikzpicture}
\begin{tikzpicture}[scale=0.45]
\draw[help lines,gray] (-2.125,-5.125) grid (6.125, 5.125);
\draw[<->] (-2,0)--(6,0)node[right]{$p$};
\draw[<->] (0,-5)--(0,5)node[above]{$q$};
\foreach \y in{-5,-4,-3,-2,0,2}
	\draw[blue,thick] (1.8,\y+0.5)--(2.8,\y+1.5) node{\dot};
\draw[blue, thick] (2.3,-5)--(2.8,-4.5) node{\dot};
\zline{1.3}{blue};
\draw[red, thick] (2.3,1.5) node{\dot}--(5.8,5);
\draw[red, thick] (2.3,3.5) node{\dot}--(3.8,5);
\draw[red, thick] (-1.2,-5)--(1.3,-2.5)node{\dot};
\draw[red, thick] (0.8,-5)--(1.3,-4.5)node{\dot};
\foreach \x in {-4,-2,0,2,4}
	\zbox{0.8}{\x}{red};
\end{tikzpicture}
\begin{tikzpicture}[scale=0.45]
\draw[help lines,gray] (-2.125,-5.125) grid (6.125, 5.125);
\draw[<->] (-2,0)--(6,0)node[right]{$p$};
\draw[<->] (0,-5)--(0,5)node[above]{$q$};
\zline{.7}{black};
	\draw[thick,dashed] (2+1/2,1+1/2) arc (145:216:1.7);
	\foreach\y in {0,2,4}
	\draw[thick] (1.5,\y+.5)--(6-\y,5);
	\foreach\y in {-2,-4}
	\draw[thick] (1.5,\y+.5)--(2.5,\y+1.5);
	\draw[thick] (2.5,-4.5)--(2,-5);
	\foreach \y in {-4,-2,0,2,4}
	\zbox{1}{\y}{black};
	\foreach \y in {1,3}
	{\draw[fill=white](2.5,\y+.5) circle (.27cm);
	\draw (2.52,\y+.51) node{\scalebox{.55}{$4$}};
	};
	\foreach \y in {-1,-3,-5}
	\draw (2.52,\y+.51) node{\dot};
	\draw[thick] (2.5,-1.5)node{\dot}--(-1,-5);
	\draw[thick] (2.5,-3.5)node{\dot}--(1,-5);
\end{tikzpicture}
\caption{$d:\H^{*,*}(S^{1,0}) \to \H^{*,*}(\tilde{Y})$ when $Y$ is nonorientable.}
\label{fig:10nonor}
\end{figure}

\begin{figure}[ht]
\begin{tikzpicture}[scale=0.45]
\draw[help lines,gray] (-2.125,-5.125) grid (6.125, 5.125);
\draw[<->] (-2,0)--(6,0)node[right]{$p$};
\draw[<->] (0,-5)--(0,5)node[above]{$q$};
\zantizo{1.3}{blue};
\zcone{0.8}{0}{red};
\draw[->, thick] (1.5,0.5)--(2.7,0.5);
\end{tikzpicture}
\begin{tikzpicture}[scale=0.45]
\draw[help lines,gray] (-2.125,-5.125) grid (6.125, 5.125);
\draw[<->] (-2,0)--(6,0)node[right]{$p$};
\draw[<->] (0,-5)--(0,5)node[above]{$q$};
\draw[blue, thick] (1.8,4.5)--(2.3,5);
\foreach \y in{-4,-2,0,2}
	\draw[blue,thick] (1.8,\y+0.5)--(2.8,\y+1.5) node{\dot};
\foreach \y in {-4,-2,0,2,4}
	\zbox{1.3}{\y}{blue};
\draw[blue] (2.8,-1.5) node{\dot};
\draw[blue] (2.8,-3.5) node{\dot};
\draw[blue, thick] (2.3,-5)--(2.8,-4.5) node{\dot};
\draw[red, thick] (2.3,1.5) node{\dot}--(5.8,5);
\draw[red, thick] (2.3,3.5) node{\dot}--(3.8,5);
\draw[red, thick] (-1.2,-5)--(1.3,-2.5)node{\dot};
\draw[red, thick] (0.8,-5)--(1.3,-4.5)node{\dot};
\end{tikzpicture}
\begin{tikzpicture}[scale=0.45]
\draw[help lines,gray] (-2.125,-5.125) grid (6.125, 5.125);
\draw[<->] (-2,0)--(6,0)node[right]{$p$};
\draw[<->] (0,-5)--(0,5)node[above]{$q$};
	\draw[thick,dashed] (2+1/2,1+1/2) arc (145:216:1.7);
	\foreach\y in {0,2,4}
	\draw[thick] (1.5,\y+.5)--(6-\y,5);
	\foreach\y in {-2,-4}
	\draw[thick] (1.5,\y+.5)--(2.5,\y+1.5);
	\draw[thick] (2.5,-4.5)--(2,-5);
	\foreach \y in {-4,-2,0,2,4}
	\zbox{1}{\y}{black};
	\foreach \y in {1,3}
	{\draw[fill=white](2.5,\y+.5) circle (.27cm);
	\draw (2.52,\y+.51) node{\scalebox{.55}{$4$}};
	};
	\foreach \y in {-1,-3,-5}
	\draw (2.52,\y+.51) node{\dot};
	\draw[thick] (2.5,-1.5)node{\dot}--(-1,-5);
	\draw[thick] (2.5,-3.5)node{\dot}--(1,-5);
\end{tikzpicture}
\caption{$d:\H^{*,*}(S^{1,0}) \to \H^{*,*}(\tilde{Y})$ when $Y$ is orientable.}
\label{fig:10or}
\end{figure}
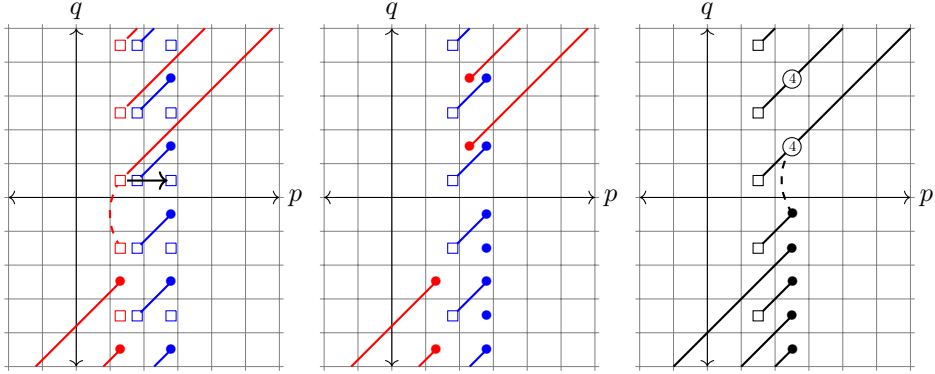

Let $k\geq 0$. We use the computation with $\underline{\Z/2}$-coefficients to see the abelian group extensions in topological degree $2$ are nontrivial. From \cite[Theorem 6.6]{Haz20}, \[\H^{2,2k+1}(Y+[S^{1,0}-AT];\underline{\Z/2})\cong \Z/2.\] Thus from the induced long exact sequence coming from the short exact sequence $\uZ\to \uZ\to \underline{\Z/2}$, we see that the extension in bidegrees $(2,2k+1)$ must be nontrivial. The cohomology with $\underline{\Z/2}$-coefficients also shows the classes in bidegrees $(0,-2k-4)$ map to the classes in bidegrees $(2,-2k-2)$ via $\rho^2$.

To finish the computation, note the actions of $x$ and $\theta$ can be determined from the actions on the cokernel and kernel. The actions of $\theta/x^j$ and $\mu/(x^i\rho^j)$ are then forced as in the proof of Theorem \ref{x2ncohom}. Lastly, a basis change, if needed, will yield the extension shown on the right in the figures. 
\end{proof}

\begin{lemma}\label{onlyf_10tube}
Let $Y$ be a nonfree $C_2$-surface whose fixed set only contains isolated fixed points. Then \[\H^{*,*}(Y+[S^{1,0}-AT])\cong \left(\Sigma^{1,1}\M \right)^{\oplus F(Y)-1} \oplus \left(\Sigma^{1,0}\A_0 \right)^{\oplus \frac{\beta(Y)-F(Y)}{2}+1} \oplus \Sigma^{2,1} \M_2.\]
\end{lemma}
\begin{proof}
We consider the usual cofiber sequence given by including the fixed circle contained in the antitube:
\[
S^{1,0} \hookrightarrow Y+[S^{1,0}-AT] \to \tilde{Y}.
\]
The $C_2$-surface $Y$ has a fixed point, and thus $\tilde{Y}\simeq Y \vee S^{1,1}$ by Lemma \ref{pinched}.

There are two cases for the differential based on the orientability of $Y$. If $Y$ is orientable, then by Theorem \ref{nonfreeoranswerz}
\[
\H^{*,*}(Y) \cong \left(\Sigma^{1,1}\M \right)^{F(Y)-2} \oplus \left( \Sigma^{1,0}\A_0\right)^{\frac{\beta(Y)-F(Y)}{2}+1} \oplus \Sigma^{2,2}\M.
\]
If $Y$ is nonorientable, then by Lemma \ref{onlyFnonor}
\[
\H^{*,*}(Y)\cong \left(\Sigma^{1,1}\M \right)^{F(Y)-2} \oplus \left( \Sigma^{1,0}\A_0\right)^{\frac{\beta(Y)-F(Y)}{2}} \oplus \Sigma^{1,1}\D_4.
\]
Recall $\H^{*,*}(\tilde{Y}) \cong \H^{*,*}(Y) \oplus \Sigma^{1,1}\M$. The differential $d: \H^{*,*}(S^{1,0}) \to \H^{*,*}(\tilde{Y})$ cannot interact with any of the $\M$- or $\A_0$-summands in topological degree one for degree reasons. Thus we have drawn the differentials in Figures \ref{fig:addingC_or} and \ref{fig:addingC_nonor} without these summands.

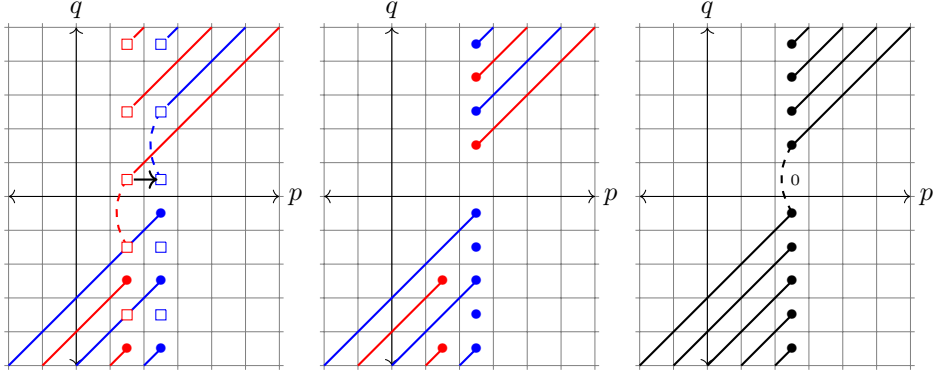
\begin{figure}[ht]
\begin{tikzpicture}[scale=0.45]
\draw[help lines,gray] (-2.125,-5.125) grid (6.125, 5.125);
\draw[<->] (-2,0)--(6,0)node[right]{$p$};
\draw[<->] (0,-5)--(0,5)node[above]{$q$};
\zcone{2}{2}{blue};
\zbox{2}{-4}{blue};
\draw[blue,thick] (2.5,-4.5) node{\dot}--(2,-5);
\zcone{1}{0}{red};
\draw[->, thick] (1.7,0.5)--(2.4,0.5);
\end{tikzpicture}
\begin{tikzpicture}[scale=0.45]
\draw[help lines,gray] (-2.125,-5.125) grid (6.125, 5.125);
\draw[<->] (-2,0)--(6,0)node[right]{$p$};
\draw[<->] (0,-5)--(0,5)node[above]{$q$};
\draw[red, thick] (2.5,1.5) node{\dot}--(6,5);
\draw[red, thick] (2.5,3.5) node{\dot}--(4,5);
\draw[red, thick] (1.5,-2.5) node{\dot}--(-1,-5);
\draw[red, thick] (1.5,-4.5) node{\dot}--(1,-5);
\draw[blue, thick] (2.5,2.5) node{\dot}--(5,5);
\draw[blue, thick] (2.5,4.5) node{\dot}--(3,5);
\draw[blue, thick] (2.5,-0.5) node{\dot}--(-2,-5);
\draw[blue, thick] (2.5,-2.5) node{\dot}--(0,-5);
\draw[blue, thick] (2.5,-4.5) node{\dot}--(2,-5);
\draw[blue] (2.5,-1.5) node{\dot};
\draw[blue] (2.5,-3.5) node{\dot};
\end{tikzpicture}
\begin{tikzpicture}[scale=0.45]
\draw[help lines,gray] (-2.125,-5.125) grid (6.125, 5.125);
\draw[<->] (-2,0)--(6,0)node[right]{$p$};
\draw[<->] (0,-5)--(0,5)node[above]{$q$};
\draw[thick] (2.5,1.5) node{\dot}--(6,5);
\draw[thick] (2.5,3.5) node{\dot}--(4,5);
\draw[thick] (2.5,-1.5) node{\dot}--(-1,-5);
\draw[thick] (2.5,-3.5) node{\dot}--(1,-5);
\draw[thick] (2.5,2.5) node{\dot}--(5,5);
\draw[thick] (2.5,4.5) node{\dot}--(3,5);
\draw[thick] (2.5,-0.5) node{\dot}--(-2,-5);
\draw[thick] (2.5,-2.5) node{\dot}--(0,-5);
\draw[thick] (2.5,-4.5) node{\dot}--(2,-5);
\draw[thick,dashed] (2+1/2,1+1/2) arc (145:216:1.7);
\draw (2.6,.5) node{\tiny{$0$}};
\end{tikzpicture}
\caption{$d:\H^{*,*}(S^{1,0}) \to \H^{*,*}(\tilde{Y})$ when $Y$ is orientable.}
\label{fig:addingC_or}
\end{figure}

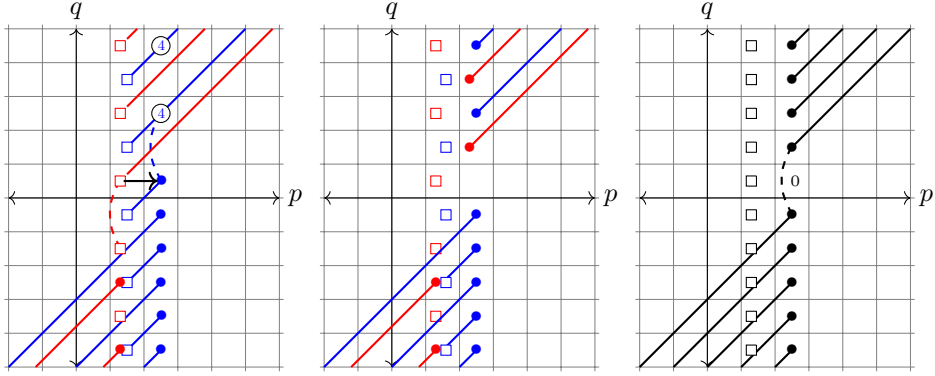
\begin{figure}[ht]
\begin{tikzpicture}[scale=0.45]
\draw[help lines,gray] (-2.125,-5.125) grid (6.125, 5.125);
\draw[<->] (-2,0)--(6,0)node[right]{$p$};
\draw[<->] (0,-5)--(0,5)node[above]{$q$};
	\draw[blue,thick,dashed] (2+1/2,2+1/2) arc (145:216:1.7);
	\foreach\y in {1,3}
	\draw[blue,thick] (1.5,\y+.5)--(6-\y,5);
	\foreach\y in {-1,-3,-5}
	\draw[blue, thick] (1.5,\y+.5)--(2.5,\y+1.5);
	\foreach \y in {-5,-3,-1,1,3}
	\zbox{1}{\y}{blue};
	\foreach \y in {2,4}
	{\draw[fill=white](2.5,\y+.5) circle (.27cm);
	\draw[blue] (2.52,\y+.51) node{\scalebox{.55}{$4$}};
	};
	\foreach \y in {0,-2,-4}
	\draw[blue] (2.52,\y+.51) node{\dot};
	\draw[blue, thick] (2.5,-0.5)node{\dot}--(-2,-5);
	\draw[blue, thick] (2.5,-2.5)node{\dot}--(0,-5);
	\draw[blue, thick] (2.5,-4.5)node{\dot}--(2,-5);
	\zcone{0.8}{0}{red};
	\draw[thick,->] (1.4,0.5)--(2.4,0.5);
\end{tikzpicture}
\begin{tikzpicture}[scale=0.45]
\draw[help lines,gray] (-2.125,-5.125) grid (6.125, 5.125);
\draw[<->] (-2,0)--(6,0)node[right]{$p$};
\draw[<->] (0,-5)--(0,5)node[above]{$q$};
\foreach \y in {-4,-2,0,2,4}
	\zbox{0.8}{\y}{red};
\draw[thick, blue] (1.5,-2.5)--(2.5,-1.5) node{\dot};
\draw[thick, blue] (1.5,-4.5)--(2.5,-3.5) node{\dot};
\foreach \y in {-5,-3,-1,1,3}
	\zbox{1.1}{\y}{blue};
\draw[thick, red] (2.3,1.5)node{\dot}--(5.8,5);
\draw[thick, red] (2.3,3.5)node{\dot}--(3.8,5);
\draw[thick, red] (1.3,-2.5)node{\dot}--(-1.2,-5);
\draw[thick, red] (1.3,-4.5)node{\dot}--(0.8,-5);
\draw[blue, thick] (2.5,-0.5)node{\dot}--(-2,-5);
\draw[blue, thick] (2.5,-2.5)node{\dot}--(0,-5);
\draw[blue, thick] (2.5,-4.5)node{\dot}--(2,-5);
\draw[thick, blue] (2.5,2.5) node{\dot}--(5,5);
\draw[thick, blue] (2.5,4.5) node{\dot}--(3,5);
\end{tikzpicture}
\begin{tikzpicture}[scale=0.45]
\draw[help lines,gray] (-2.125,-5.125) grid (6.125, 5.125);
\draw[<->] (-2,0)--(6,0)node[right]{$p$};
\draw[<->] (0,-5)--(0,5)node[above]{$q$};
\draw[thick] (2.5,1.5) node{\dot}--(6,5);
\draw[thick] (2.5,3.5) node{\dot}--(4,5);
\draw[thick] (2.5,-1.5) node{\dot}--(-1,-5);
\draw[thick] (2.5,-3.5) node{\dot}--(1,-5);
\draw[thick] (2.5,2.5) node{\dot}--(5,5);
\draw[thick] (2.5,4.5) node{\dot}--(3,5);
\draw[thick] (2.5,-0.5) node{\dot}--(-2,-5);
\draw[thick] (2.5,-2.5) node{\dot}--(0,-5);
\draw[thick] (2.5,-4.5) node{\dot}--(2,-5);
\draw[thick,dashed] (2+1/2,1+1/2) arc (145:216:1.7);
\draw (2.6,.5) node{\tiny{$0$}};
\zline{0.8}{black};
\end{tikzpicture}
\caption{$d:\H^{*,*}(S^{1,0}) \to \H^{*,*}(\tilde{Y})$ when $Y$ is nonorientable.}
\label{fig:addingC_nonor}
\end{figure}

The fixed set of $X$ contains a fixed circle, so $\H^{2,0}(X)=0$ by Lemma \ref{miscmani}. Thus $d^{1,0}$ must be surjective, which entirely determines the differential in both cases. It just remains to solve the extension problems shown in the middle of Figures \ref{fig:addingC_or} and \ref{fig:addingC_nonor}. One can use the forgetful long exact sequence to conclude the extension is solved by the picture on the right. Note the action by $x$ is 0 on the class $\alpha$ in bidegree $(2,-1)$ because otherwise $\rho x \alpha \neq 0$ which would imply $\rho \alpha \neq 0$, but this is not possible for degree reasons. Remembering that we omitted summands from these figures then yields the correct answer. 
\end{proof}

\begin{lemma}\label{onlyf_fm}
Let $Y$ be a $C_2$-surface whose fixed set only contains isolated fixed points. Then 
\[\H^{*,*}(Y+[FM[)\cong   \left(\Sigma^{1,1}\M \right)^{\oplus F(Y)-2} \oplus \left(\Sigma^{1,0}\A_0 \right)^{\oplus \frac{\beta(Y)-F(Y)}{2}+1} \oplus \Sigma^{2,1} \M_2.\]
\end{lemma}
\begin{proof}
This is similar to the previous proof, so we just provide an outline. Let $U$ be a closed neighborhood of the attached M\"obius band in $Y+[FM]$ such that $U$ is homotopic to the M\"obius band. Crushing the neighborhood $U$ to a fixed point gives us a $C_2$-space homotopic to the original $C_2$-surface $Y$. Thus we have a cofiber sequence
\[
U\hookrightarrow Y+[FM] \to Y.
\]
The two cases for the cofiber sequences based on the orientability of $Y$ will follow just as in Lemma \ref{onlyf_10tube}. Note the only difference is that the cofiber is $Y$ instead of $\tilde{Y}$, which is why the final answer does not have an additional copy of $\Sigma^{1,1}\M$ as it did in Lemma \ref{onlyf_10tube}.
\end{proof}

We can now prove the final theorem.

\begin{theorem}\label{nonfreenonoranswer}
 Let $X$ be a nontrivial, nonfree $C_2$-surface whose underlying space is nonorientable. Then there are three cases for the cohomology of $X$ based on the fixed set:
 \begin{enumerate}
 \item[(i)] $F \neq 0 $, $C_+=C_-=0$. Then 
\[\tilde{H}^{*,*}(X) \cong \left(\Sigma^{1,1}\M\right)^{\oplus F-2} \oplus \left(\Sigma^{1,0}\A_0\right)^{\oplus \frac{\beta -F}{2}}\oplus\Sigma^{1,1}\D_4.\]
 \item[(ii)] $F =0 $, $C_+\neq0$, $C_-=0$. Then
 \[\tilde{H}^{*,*}(X) \cong \left(\Sigma^{1,0}\M\right)^{\oplus C-1}\oplus \left(\Sigma^{1,1}\M\right)^{\oplus C-1} \oplus \left(\Sigma^{1,0}\A_0\right)^{\oplus \frac{\beta -2C}{2}}\oplus \Sigma^{1,0}\D_4.\]
 \item[(iii)] $F \neq 0 $, $C_+\neq0$, $C_-=0$ or $F \geq 0 $, $C+\geq0$, $C_-\neq0$. Then 
 \[\tilde{H}^{*,*}(X) \cong  \left(\Sigma^{1,0}\M\right)^{\oplus C-1}\oplus \left(\Sigma^{1,1}\M\right)^{\oplus F+C-2} \oplus \left(\Sigma^{1,0}\A_0\right)^{\oplus \frac{\beta -(F+2C)}{2}+1}\oplus \Sigma^{2,1}\M_2.\]
 \end{enumerate}
\end{theorem}
\begin{proof}
Observe (i) follows immediately from Lemma \ref{onlyFnonor}. We break into three cases to prove (ii) and (iii).\medskip

\emph{Case 1: $F =0 $, $C_+\neq0$, $C_-=0$.} We induct on the number of fixed circles. The base case is $C=1$. Consider an equivariant tubular neighborhood of the fixed circle. The neighborhood must be an $S^{1,0}$-antitube since the circle is two-sided. As described in Section \ref{sec:backgroundsurg}, we can do equivariant surgery to remove this tube and then attach two conjugate disks in order to get a possibly disconnected surface with a free $C_2$-action. 

If the resulting surface is disconnected, then it must be isomorphic to $C_2\times N_r$. Hence $X\cong S^{2,1}\#_2 N_r$. Following the outline to compute $\H^{*,*}(S^{2,2}\#_2N_r)$ given in the proof of Lemma \ref{onlyFnonor} with everything shifted down in weight by one, we get
\[\H^{*,*}(X)\cong \left(\Sigma^{1,0}\A_0 \right)^{\oplus r-1} \oplus \Sigma^{1,0}\D_4.\] 
Note $\beta=2r$ and $C=1$, so $(\beta-2C)/2=r-1$, as desired. 

If the resulting surface is connected, then $X\cong Y+[S^{1,0}-AT]$ where $Y$ is the resulting free $C_2$-surface. From Lemma \ref{free_10tube}, 
\[\H^{*,*}(X)\cong \left(\Sigma^{1,0}\A_0\right)^{\oplus \beta(Y)/2}\oplus \Sigma^{1,0}\D_4.\] Observe $\beta(X)=\beta(Y)+2$ and $C(X)=1$, so $(\beta(X)-2C(X))/2=\beta(Y)/2$. This completes the base case.

For the inductive step, suppose $C\geq 2$. We can then do surgery around one of the two-sided fixed circles to see $X\cong Y + [S^{1,0}-AT]$ where $Y$ is a $C_2$-surface with $C-1$ fixed circles. Note $Y$ must be connected because $Y^{C_2}$ is nonempty. From Lemma \ref{attach_antitube10}, 
\[\H^{*,*}(X)\cong \H^{*,*}(Y) \oplus \Sigma^{1,0}\M \oplus \Sigma^{1,1}\M.\] Note $Y$ must be nonorientable because if it were orientable, then the action on $Y$ would have to be orientation reversing because $C(Y)\geq 1$. But then $Y+[S^{1,0}-AT]$ would also be orientable, which contradicts that $X$ is nonorientable. Thus we can use the inductive hypothesis to conclude
\[
\H^{*,*}(Y) \cong \left(\Sigma^{1,0}\A_0\right)^{\oplus \frac{\beta(Y) -2C(Y)}{2}}\oplus \left(\Sigma^{1,0}\M\right)^{\oplus C(Y)-1}\oplus \left(\Sigma^{1,1}\M\right)^{\oplus C(Y)-1} \oplus\Sigma^{1,0}\D_4.
\]
Plugging in $C(Y)=C(X)-1$ and $\beta(Y)=\beta(X)-2$, and then using that $\H^{*,*}(X)\cong \H^{*,*}(Y) \oplus \Sigma^{1,0}\M \oplus \Sigma^{1,1}\M$ completes this case.
\medskip

\emph{Case 2: $F \neq 0 $, $C_+\neq0$, $C_-=0$}. We again induct on $C$. If $C=1$, then we can do surgery as in Case 2 to get that $X\cong Y+[S^{1,0}-AT]$. Note $F(X)=F(Y)$, so it must be that $Y$ contains isolated fixed points. From Lemma \ref{onlyf_10tube}, $$\H^{*,*}(X) \cong   \left(\Sigma^{1,1}\M \right)^{\oplus F(Y)-1} \oplus \left(\Sigma^{1,0}\A_0 \right)^{\oplus \frac{\beta(Y)-F(Y)}{2}+1} \oplus \Sigma^{2,1}\M_2.$$
Plugging in $\beta(Y)=\beta(X)-2$, $F(X)=F(Y)$, and $C(X)=1$ then yields the desired answer.

For the inductive step, suppose $C\geq 2$. As in Case 2, do surgery around one of the fixed circles to get that $X\cong Y + [S^{1,0}-AT]$ where $Y$ is a $C_2$-surface with $C-1$ fixed circles such that $F(Y)\geq 1$. From Lemma \ref{attach_antitube10}, 
\[\H^{*,*}(X)\cong \H^{*,*}(Y) \oplus \Sigma^{1,0}\M \oplus \Sigma^{1,1}\M.\] 
 From the inductive hypothesis,
 \begin{align*}
 \tilde{H}^{*,*}(Y) \cong& \left(\Sigma^{1,0}\A_0\right)^{\oplus \frac{\beta(Y) -(F(Y)+2C(Y))}{2}+1}\oplus \left(\Sigma^{1,0}\M\right)^{\oplus C(Y)-1}\\
 &\oplus \left(\Sigma^{1,1}\M\right)^{\oplus F(Y)+C(Y)-2} \oplus\Sigma^{2,1}\M_2.
 \end{align*}
Observe $F(Y)=F(X)$, $\beta(Y)=\beta(X)-2$, $C(Y)=C(X)-1$. Substituting those values into the above will then give the desired answer.
\medskip

\emph{Case 3: $F \geq 0 $, $C+\geq0$, $C_-\neq0$. } We induct on $C_-$. For the base case, suppose $C_-=1$. We can do surgery to get that $X\cong Y+[FM]$ for a $C_2$-surface $Y$ where $F(Y)=F(X)+1$, $\beta(Y)=\beta(X)-1$, $C_-(Y)=0$, and $C_+(Y)=C_+(X)$. If $C_+(Y)=0$, then from Lemma \ref{onlyf_fm}
\[\H^{*,*}(X)\cong \left(\Sigma^{1,1} \M \right)^{\oplus F(Y)-2} \oplus \left(\Sigma^{1,0} \A_0 \right)^{\oplus \frac{\beta(Y)-F(Y)}{2}+1} \Sigma^{2,1}\M_2.\] 
Observe 
\[F(Y)-2=F(X)+1-2=F(X)+C(X)-2, \text{ and}\]  \[\beta(Y)-F(Y)=\beta(X)-F(X)-2=\beta(X)-(F(X)+2C(X)),\] so there are the correct amount of summands. If $C_+(Y)>0$, then \[\H^{*,*}(Y+[FM])\cong \H^{*,*}(Y)\oplus \Sigma^{1,0}\M\] by Lemma \ref{attach_fm}. Since $Y^{C_2}$ contains circles and isolated points, we can use Case 3 to get
 \begin{align*}
 \tilde{H}^{*,*}(Y) \cong& \left(\Sigma^{1,0}\A_0\right)^{\oplus \frac{\beta(Y) -(F(Y)+2C(Y))}{2}+1}\oplus \left(\Sigma^{1,0}\M\right)^{\oplus C(Y)-1}\\
 &\oplus \left(\Sigma^{1,1}\M\right)^{\oplus F(Y)+C(Y)-2} \oplus\Sigma^{2,1}\M_2.
 \end{align*}
Substituting the appropriate values then completes this case.

For the inductive step, we again do surgery to get that $X\cong Y+[FM]$, but now $C_-(Y)>0$, so we can use Lemma \ref{attach_fm} and the inductive hypothesis to complete the proof.
\end{proof}

\begin{rmk}\label{nonfreenoninv}
In \cite{Dug19} it was shown the list of values $\beta$, $F$, $C_+$, $C_-$ is not a complete invariant for $C_2$-surfaces. The $\M$-module structure of the cohomology only depends on these values, so it is also not a complete invariant. For example,  
\[ (S^{2}_a\#_2 \R P^2) +[S^{1,0}-AT] \not \cong T^{rot}_1 + [S^{1,0}-AT]\]
because the quotient spaces are not homeomorphic. But these $C_2$-surfaces have the same underlying space and fixed set, and thus as $\M$-modules
\[H^{*,*}( (S^{2}_a\#_2 \R P^2) +[S^{1,1}-AT] ) \cong H^{*,*}(T^{rot}_1 + [S^{1,0}-AT]).\]
\end{rmk}

\bibliography{mybib}
\bibliographystyle{plain}

\end{document}